%% file: main_preprint.tex
\documentclass{article}

\usepackage[utf8]{inputenc} 
\usepackage[T1]{fontenc}
\usepackage{color}
\usepackage{url}              
\usepackage{cite}             
\usepackage{authblk}

\usepackage[margin=1in]{geometry}
\usepackage{amsmath, amsfonts, amsthm, amssymb}
\usepackage{graphicx}         
\usepackage{booktabs}         

\input{commands.tex}
\newcommand{\Cref}{Section \ref}

\newtheorem{theorem}{Theorem}
\newtheorem{proposition}{Proposition}
\newtheorem{lemma}{Lemma}
\newtheorem{corollary}{Corollary}
\newtheorem{definition}{Definition}
\theoremstyle{remark}
\newtheorem{remark}{Remark}

\begin{document}

\title{On Separability of Covariance in Multiway Data Analysis} 
\author[1]{Dogyoon Song}
\author[2,3]{Alfred O. Hero}
\affil[1]{Department of Statistics, University of California, Davis}
\affil[2]{Department of EECS, University of Michigan}
\affil[3]{Department of Statistics, University of Michigan}

\date{\today}

\maketitle

\input{contents/00_abstract}

\clearpage
\tableofcontents
\clearpage

\input{contents/01_introduction}

\input{contents/02_background}
\input{contents/03_separable_covariance}
\input{contents/04_general_results}
\input{contents/05_experiments}
\input{contents/06_discussions}

\section*{Acknowledgments}
This work was partially supported by ARO grant number W911NF-23-1-0343, and NSF grant number CCF-2246213.

\bibliographystyle{alpha}
\bibliography{sep_cov}

\newpage
\appendix
\input{contents/SuppA_additional_KS_discussion}
\input{contents/SuppB_proof_abundance}

\input{contents/SuppC_NPHard}
\input{contents/SuppD_Frank_Wolfe}
\input{contents/SuppE_experiments}

\end{document}

%% file: commands.tex
\usepackage{amsmath, amsfonts, amssymb}

\usepackage{mathtools}
\usepackage{adjustbox}
\usepackage{enumitem}
\usepackage{booktabs}
\usepackage{array}
\usepackage{multirow}
\usepackage{makecell}
\usepackage{caption}
\usepackage{subcaption}

\usepackage{algorithm}
\usepackage{algpseudocode}

\usepackage{hyperref}
\hypersetup{
    colorlinks=true,
    linkcolor=blue,
    citecolor=red,      
    urlcolor=blue,
    pdfpagemode=FullScreen,
    }

\newcommand{\FF}{\mathbb{F}}
\newcommand{\NN}{\mathbb{N}}

\newcommand{\RR}{\mathbb{R}}
\newcommand{\CC}{\mathbb{C}}
\newcommand{\bbB}{\mathbb{B}}
\newcommand{\bbS}{\mathbb{S}}

\newcommand{\be}{\mathbf{e}}
\newcommand{\bg}{\mathbf{g}}
\newcommand{\bs}{\mathbf{s}}
\newcommand{\bu}{\mathbf{u}}
\newcommand{\bv}{\mathbf{v}}
\newcommand{\bw}{\mathbf{w}}
\newcommand{\bx}{\mathbf{x}}
\newcommand{\by}{\mathbf{y}}
\newcommand{\bz}{\mathbf{z}}
\newcommand{\bA}{\mathbf{A}}
\newcommand{\bB}{\mathbf{B}}

\newcommand{\bF}{\mathbf{F}}
\newcommand{\bG}{\mathbf{G}}
\newcommand{\bI}{\mathbf{I}}
\newcommand{\bM}{\mathbf{M}}
\newcommand{\bP}{\mathbf{P}}

\newcommand{\bX}{\mathbf{X}}

\newcommand{\bZ}{\mathbf{Z}}
\newcommand{\bSigma}{\mathbf{\Sigma}}

\newcommand{\tw}{\widetilde{w}}

\newcommand{\tbv}{\widetilde{\bv}}

\newcommand{\cA}{\mathcal{A}}
\newcommand{\cB}{\mathcal{B}}
\newcommand{\cC}{\mathcal{C}}
\newcommand{\cD}{\mathcal{D}}
\newcommand{\cH}{\mathcal{H}}
\newcommand{\cK}{\mathcal{K}}

\newcommand{\cN}{\mathcal{N}}
\newcommand{\cO}{\mathcal{O}}
\newcommand{\cP}{\mathcal{P}}
\newcommand{\cS}{\mathcal{S}}
\newcommand{\cT}{\mathcal{T}}
\newcommand{\cV}{\mathcal{V}}

\newcommand{\ovbSigma}{\overline{\bSigma}}

\newcommand{\hbSigma}{\widehat{\bSigma}}
\newcommand{\hZ}{\widehat{Z}}
\newcommand{\hbZ}{\widehat{\bZ}}

\newcommand{\tr}{\mathrm{tr}\,}
\newcommand{\Tr}{\mathrm{Tr}\,}
\newcommand{\Cov}{\mathrm{Cov}}

\newcommand{\diam}{\mathrm{diam}\,}

\newcommand{\bbE}{\mathbb{E}}

\newcommand{\vol}{\mathrm{vol}}
\newcommand{\vrad}{\mathrm{vr}}

\newcommand{\RRp}{\RR_{+}}

\newcommand{\Napprox}{N_{\mathrm{appr}}}
\newcommand{\Ntrue}{N_{\mathrm{true}}}
\newcommand{\softmax}{\texttt{SoftMax}}
\newcommand{\dimsa}{\kappa}
\newcommand{\wmem}{ \normalfont \texttt{WMEM}}
\newcommand{\clique}{\normalfont \texttt{CLIQUE}}
\newcommand{\rsdf}{\normalfont \texttt{RSDF}}
\newcommand{\wopt}{\normalfont \texttt{WOPT}}
\newcommand{\ttA}{\texttt{A}}
\newcommand{\ttB}{\texttt{B}}
\newcommand{\NP}{\mathrm{NP}}
\newcommand{\NPC}{\mathrm{NPC}}
\newcommand{\bell}{\mathrm{Bell}}
\newcommand{\leqT}{\leq_{\texttt{T}}}
\newcommand{\leqK}{\leq_{\texttt{K}}}
\newcommand{\Unif}{\mathrm{Unif}}

\DeclareMathOperator{\conv}{\mathrm{conv}}
\DeclareMathOperator{\cone}{\mathrm{cone}}
\DeclareMathOperator{\ext}{\mathrm{ext}}
\DeclareMathOperator{\pisep}{\pi_{\mathrm{sep}}}
\DeclareMathOperator{\upisep}{\pi^{\mathrm{unit}}_{\mathrm{sep}}}
\DeclareMathOperator{\argmin}{\mathrm{arg~min}}
\DeclareMathOperator{\argmax}{\mathrm{arg~max}}

\newcommand{\rvec}{\mathrm{vec}}
\newcommand{\rank}{\mathrm{rank}}
\newcommand{\ranks}{\mathrm{rank_{\mathrm{KS}}}}

\newcommand{\cBsa}{\cB^{\mathrm{sa}}}
\newcommand{\cBsep}{\cB^{\mathrm{sep}}_+}
\newcommand{\cSsa}{\cS^{\mathrm{sa}}}
\newcommand{\cSsep}{\cS^{\mathrm{sep}}_+}

\newcommand{\iso}{\mathrm{iso}}
\newcommand{\Werner}{\mathrm{Wer}}

\newtheorem{question}{Question}
\newtheorem{problem}{Problem}
\newtheorem{example}{Example}
\newtheorem{cexample}{Counterexample}

%% file: contents/00_abstract.tex
\begin{abstract}
Multiway data analysis aims to uncover patterns in data structured as multi-indexed arrays, with multiway covariance playing a crucial role in many applications. 
However, the high dimensionality of multiway covariance presents significant computational challenges. 
To overcome these challenges, factorized covariance models have been proposed that rely on a separability assumption: the multiway covariance can be accurately expressed as a sum of Kronecker products of mode-wise covariances. 
This paper addresses the representability, certification, and approximation of such separable models, leaving statistical estimation or finite-sample properties aside. 
We reduce the question of whether a given covariance can be decomposed into a separable multiway form to an equivalent question about the separability of quantum states. 
Leveraging results from quantum information theory, we show that generic multiway covariances are typically \emph{not} separable and that determining the best separable approximation is NP-hard. 
These findings suggest that factorized covariance models can be overly restrictive and difficult to fit without additional structural assumptions. 
Nevertheless, our numerical experiments indicate that standard iterative algorithms, namely Frank-Wolfe and gradient descent, often converge close to the best separable approximation. 
As NP-hardness concerns worst-case computational complexity, Kronecker-separable approximations to multiway covariance could still be tractable to apply for analyzing many real-world datasets.

\medskip
\noindent
{\small 
\textbf{\textit{Keywords:}}
Multiway data, Separable covariance, Kronecker PCA, Low-rank covariance model, Tensor decomposition
}

\end{abstract}

%% file: contents/01_introduction.tex
\section{Introduction}\label{sec:introduction}

Probabilistic modeling is central to statistics, data science and machine learning, underpinning various tasks such as regression, classification, clustering, signal detection and model selection \cite{gelman1995bayesian, bishop2006pattern, hastie2009elements, koller2009probabilistic, jolliffe1982note, meinshausen2006high, friedman2008sparse}.  
A key concept in such modeling is the covariance, which captures joint interactions between variables and goes beyond standard first-order analysis that only attempt to fit the mean. 
Specifically, the covariance captures linear pairwise dependencies among variables and forms the foundation for second-order analysis.

Modern datasets often naturally possess multi-dimensional tensor structure \cite{kolda2009tensor}, where each dimension represents a distinct feature of the observations. 
Examples include chemometrics (sample\,$\times$\,sensor$\,\times$\,time), brain imaging (voxel\,$\times$\,time\,$\times$\,subject), trade networks (country\,$\times$\,product\,$\times$\,year), and recommendation systems (user\,$\times$\,item\,$\times$\,context).  
As data complexity grows, analyzing interactions across multiple modes becomes essential.

However, analyzing multiway data poses significant computational and modeling challenges.  
The number of variables in an unstructured $K$-way real-valued tensor of dimension $d_1 \times \cdots \times d_K$ scales like $d = \prod_{i=1}^K d_i$ while the number of variables in its covariance matrix scales as $d(d+1)/2$, for which standard covariance-based models quickly become infeasible. 
To manage this complexity, it is common to impose low-dimensional structures or decompositions---akin to the singular value decomposition (SVD) for the case of matrices ($K=2$)---to represent the data in a more tractable form.  
Such approaches aim to reduce the effective number of parameters by partitioning or factorizing large data structures into smaller mode-specific components.

In multiway covariance modeling, \emph{separable covariance models} represent the full covariance as a product of mode-wise covariance factors.
For example, the matrix normal model \cite{dawid1981some} posits that the covariance of a two-way (matrix) random variable $X\in\RR^{m\times n}$ takes the form $\Sigma = \Sigma_1 \otimes \Sigma_2$, where $\otimes$ denotes the Kronecker product, and each $\Sigma_i$ captures the covariance structure along one mode (rows vs.\ columns of $X$). 
This factorization simplifies both estimation and interpretation, since the large covariance can be represented through fewer parameters. 
Even when a true covariance is not strictly separable, one may seek to \emph{approximate} it by a sum of products of positive semidefinite (PSD) factors in order to reduce complexity and improve interpretability. 

This raises fundamental questions about the \emph{representability} of arbitrary PSD matrices by such sums, and the feasibility of \emph{computing} the ``best'' of such approximations with respect to a fitting criterion.
A central question is the following: how accurately can large-scale covariances be approximated by Kronecker-separable (K-S) models composed of sums of Kronecker products of PSD factor matrices, and under what conditions is this tractable? 

Motivated by these questions, we investigate the expressive power and limitations of \emph{multiway covariance structures} that are K-S representable as a finite sum of Kronecker products of PSD matrices. 
We demonstrate that, without additional modeling assumptions on data, it is not always possible to decompose a given covariance into such a separable form. 
Concretely, in this paper, we show that: 
\begin{itemize}
    \item 
    A significant fraction of PSD matrices is \emph{not} representable by such sums of Kronecker products (not K-S representable) unless additional structure is assumed;
    \item 
    Determining whether or not a given covariance is K-S can be \emph{NP-hard};
    \item 
    Finding the \emph{best} K-S approximation is also NP-hard, illustrating that there are severe computational challenges for large-scale scenarios.
\end{itemize}

These results rely on a connection to quantum physics, where similar questions arise in the study of entangled versus separable quantum states. 
By translating known results from that domain and supplementing them with new analyses bridging them, we clarify when and why multiway covariances can or cannot be decomposed into PSD Kronecker factors, highlighting both the promise and the pitfalls of separable modeling in high dimensions.
Notwithstanding these worst-case computational complexity implied by NP-hardness, our numerical experiments suggest that certain classes of iterative K-S approximations may remain useful, especially when the dimensions are moderate or when structural patterns are present.

We note that the focus of this paper is on the theoretical limits of separable decompositions, such as identifiability and computational tractability, and we do not study statistical estimators or finite‐sample behavior.

\subsection{Overview of contributions}
Consider a random tensor of order $K$ denoted $\bX \in \RR^{n_1 \times n_2 \times \cdots \times n_K}$. 
Its covariance is an $n \times n$ positive semidefinite (PSD) matrix $\bSigma$, where $n = \prod_{k=1}^K n_k$. 
We say $\bSigma$ is \emph{Kronecker-separable} (K-S)\footnote{See \Cref{sec:sep_covariance} for a formal definition.} if it can be expressed as a sum of Kronecker products of PSD matrices. 
Formally, $\bSigma$ is K-S if there exists a positive integer $r$ and a sequence of $K$-tuples $\big( (\bSigma^{(1)}_a, \dots, \bSigma^{(K)}_a) \big)_{a=1}^r$ such that
\begin{equation}\label{eqn:factorization.simple}
    \bSigma 
        = \sum_{a=1}^r \bSigma^{(1)}_a \otimes \dots \otimes \bSigma^{(K)}_a
    \quad
    \text{where}
    \quad
    \bSigma^{(k)}_a \in \RR^{n_k \times n_k} \text{ is PSD for all } (k,a).
\end{equation}
Here, $r$ sets the complexity of the decomposition, called the Kronecker-separation rank, and each factor $\Sigma^{(k)}_a$ serves as a ``covariance component'' along the $k$-th mode. 
Note that the well-known matrix normal model \cite{dawid1981some} arises as the special case $r=1$ and $K=2$.

A natural starting point is to ask whether \emph{every} PSD matrix $\bSigma$ can be written in the form \eqref{eqn:factorization.simple}.
Because the factors are constrained to be PSD, the answer to this question is nontrivial. 
In fact, even in the simplest non-trivial configuration ($K =2$ and $n_1 = n_2 = 2$), a simple dimension counting argument\footnote{The set of $n \times n$ PSD matrices forms a full-dimensional cone in the space of $n \times n$ symmetric matrices, which has dimension $\binom{n+1}{2}$. For $K=2$ and $n_1 = n_2 = 2$, the set of all possible $\bSigma$ spans a vector space of dimension $\binom{n_1 \times n_2 + 1}{2} = 10$. However, the right hand side of \eqref{eqn:factorization.simple} can only span a space of dimension $\binom{n_1 + 1}{2} \times \binom{n_2 + 1}{2} = 9$.} establishes that some $\bSigma$ cannot be decomposed in such a way. 
This motivates the following quantitative question about K-S representability:
\begin{question}
\label{question:representability}
     Among all multiway covariances, what fraction are K-S? 
\end{question}

Given that not every multiway covariance is (exactly) K-S, two further questions arise regarding \emph{certification} and \emph{approximation}.
\begin{question}
\label{question:certifiability}
    Given a covariance $\bSigma$, can we efficiently determine whether it is K-S? 
\end{question}

\begin{question}
\label{question:approximability}
    If $\bSigma$ is not exactly separable, can we compute its best K-S approximation? 
\end{question}

In \Cref{sec:questions}, we address these questions by linking the Kronecker-separability of covariances to the separability of quantum states. 
A quantum state is modeled by a density operator (PSD and trace~1); it is ``separable'' if it can be written as a convex combination of tensor products of local states (smaller-sized density operators).
Normalizing K-S covariance $\bSigma$ by its trace makes it precisely a ``separable'' quantum state.
We leverage known results on the abundance of entangled (non-separable) states \cite{aubrun2017alice} and the computational hardness of certifying separability \cite{gurvits2003classical, gharibian2008strong} to affirm negative answers to Questions \ref{question:representability}--\ref{question:approximability} in the following sense: 
(1) Most PSD matrices are \emph{not} K-S in high dimensions (Question \ref{question:representability}); 
(2) Deciding separability is \emph{NP-hard} (Question \ref{question:certifiability}); 
(3) Finding the \emph{best} separable approximation is also \emph{NP-hard} (Question \ref{question:approximability}). 
Based on these findings, we conclude that Kronecker-separability is not universally applicable---unlike singular value decomposition or other tensor decompositions---without additional structural assumptions on the covariance.

Despite these negative results, we suggest that K-S covariance models may still be useful in practical scenarios, by investigating:
\begin{question}
\label{question:practical} 
    Given the worst-case hardness of certifying separability or computing K-S approximations, can we still develop effective algorithms that yield high-quality K-S approximations in practice? 
    Under what specific conditions can such algorithms be found? 
\end{question}
We provide a partial answer to Question \ref{question:practical} in \Cref{sec:experiments}, showing that standard iterative methods (e.g., Frank–Wolfe or gradient-based approaches) often yield \emph{good} approximations, even though worst-case hardness precludes universal \emph{goodness} guarantees. 
This suggests that the aforementioned worst-case NP-hardness results may be overly conservative in practice, so computing or certifying Kronecker-separability may be tractable in many instances. 

Overall, this paper shows that K-S covariance modeling offers advantages (simplicity, interpretability) despite its inherent limitations (ubiquity of non-separability, NP-hardness). 
Connections to quantum information deepen our understanding of the underlying barriers, while standard heuristics often suffice in structured, moderate-scale domains. 
A summary of our results and their implications is given in Table~\ref{tab:summary}.

\medskip
\paragraph{Organization}
The remaining sections of this paper are organized as follows. 
In \Cref{sec:background}, we review mathematical background and introduce the notation used throughout. 
In \Cref{sec:sep_covariance}, we formally define the notion of multiway covariance and the concept of Kronecker-separability. 
Next, in \Cref{sec:questions}, we restate Questions \ref{question:representability}, \ref{question:certifiability} and \ref{question:approximability} in a mathematically precise manner and investigate them in individual subsections. 
In \Cref{sec:experiments}, we describe iterative methods for K-S approximation and present numerical simulations. 
Finally, in \Cref{sec:conclusion}, we offer concluding remarks and potential future research directions. 
Additional technical details, proofs, and experiments are given in the supplementary materials.

\begin{table}[t]
    \caption{Summary of results in this paper and their implications.}
    \label{tab:summary}

    \begin{adjustbox}{max width=\textwidth}
    \begin{tabular}{l c c c}
        \toprule
            &   \textbf{Section} &   \textbf{Results}  &   \textbf{Implications}\\
        \midrule
        Question \ref{question:representability}    &   \Cref{sec:separability}     &   Theorem \ref{thm:abundance} &  Abundance of non- \\
        (Representability) &    &    & Kronecker-separable covariances   \\
        \midrule
        Question \ref{question:certifiability} &  \Cref{sec:certification}    &  Some sufficient/necessary conditions &  Verifiable (partial) tests for separability  \\
        (Certification) &    &   Theorem \ref{thm:NP-hard}   &   NP-hardness (worst case)\\
        \midrule
        Question \ref{question:approximability}   & \Cref{sec:approximation}  &  Theorem \ref{thm:approx}  &  NP-hardness of approximation  \\
        (Approximation)   &  &    & (worst case)  \\
        \midrule
        Question \ref{question:practical}  & \Cref{sec:experiments}  & Theorem \ref{thm:convergence_FW} & Oracle complexity bound\\
        (Practical method)               &   & Numerical experiments & Practical utility\\ 
        \bottomrule
    \end{tabular}
    \end{adjustbox}
\end{table}

\subsection{Related work}

\paragraph{Kronecker product and tensor decomposition}
The Kronecker product (KP) combines two matrices into a larger block-structured matrix, algebraically acting as a tensor product on matrix spaces.  
It is widely used in scientific and engineering applications for compact, efficient representations of complex systems.
For example, in control theory, KPs model large systems as compositions of smaller subsystems; in signal processing, they model multi-channel sensing systems; and in network science, they represent hierarchical, multi-level graph structures. 
The approximation properties of KPs were systematically analyzed by van Loan and Pitsianis~\cite{van1993approximation}.

Tensor decompositions extend matrix factorizations to higher-dimensional arrays, i.e., tensors, enabling efficient representations of multiway data. 
Common examples include CANDECOMP/PARAFAC (CP) \cite{harshman1970foundations, carroll1970analysis}, Tucker/HOSVD \cite{tucker1966some, hitchcock1928multiple}, and the tensor-train decompositions \cite{oseledets2011tensor}. 
Advances in algorithms, such as the Alternating Least Squares (ALS), randomized tensor decompositions and spectral representations, have made these methods scalable; see \cite{kolda2009tensor} for a survey. 
Their applications span chemometrics, computer vision, neuroscience, and natural language processing, enabling extraction of latent factors and compression of complex data.  
However, most tensor decomposition frameworks focus on first-order low-rank approximations, without enforcing interpretability constraints or accounting for second-order interactions between variables. 
In contrast, enforcing PSD is essential when the factors are to be interpreted as mode‑wise covariances, a restriction that motivates the KP models studied in this paper.

\medskip
\paragraph{Separable covariance models}
Separable covariance models reduce the effective dimensionality of multiway data by factoring the covariance along each mode, as in the matrix normal model \cite{dawid1981some}.  
This structure simplifies representation and estimation of covariance \cite{strobach1995low, dutilleul1999mle}.  
Extensions to higher-order tensors ($K > 2$) have broadened their use in various domains, including spatiotemporal modeling, gene expression, and radar communications \cite{mardia1993spatial, galecki1994general, hoff2011separable}. 
These models also have been used for inverse covariance estimation in graphical modeling. 
Methods such as transposable covariance \cite{allen2010transposable}, BiGlasso \cite{kalaitzis2013bigraphical}, TeraLasso \cite{greenewald2019tensor}, and SyGlasso \cite{wang2020sylvester} incorporate graph-structural (sparsity) constraints to manage complex dependencies efficiently. 
To increase expressive power, Kronecker-product-based generalizations of PCA (KPCA) \cite{tsiligkaridis2013covariance, greenewald2015robust} approximate the covariances as sums of Kronecker products of mode-wise factors---but in doing so, lose the guarantee that each factor be PSD, limiting the interpretation of the factors as genuine mode-wise covariances.
This paper addresses this representational gap by examining sums of PSD Kronecker factors.

\medskip
\paragraph{Connection to quantum entanglement}
Kronecker-separability is closely connected to separability of quantum states, a fundamental concept in quantum physics. 
A quantum state is represented by a unit-trace positive semi-definite (PSD) operator, called a density operator. 
It is called \emph{separable} if it can be expressed as a convex combination of tensor products of local states; otherwise, it is \emph{entangled} \cite{peres1996separability, horodecki2001separability}. 
Determining separability of quantum states is NP-hard \cite{gurvits2003classical}, and so is finding their best separable approximation.  
While partial criteria for separability like the Positive Partial Transpose (PPT) criterion \cite{peres1996separability, horodecki2001separability} or more advanced tests based on semidefinite programming hierarchies \cite{doherty2004complete} exist, to our knowledge, there is no known criterion that can fully characterize separability. 

Despite this hardness, heuristic methods often succeed in approximating the closest separable state in practice \cite{dahl2007tensor, girardin2022building}. 
These methods perform well empirically, but lack theoretical guarantees, mirroring the behavior observed in Kronecker-separable modeling. 
Motivated by such connections, this paper shows that certifying or approximating Kronecker-separability faces similar challenges, yet standard optimization methods often succeed in finding near-optimal separable approximations, highlighting both the representational limits and practical potential of separable covariance models in data science.

%% file: contents/02_background.tex
\section{Background and preliminaries}\label{sec:background}
In this section, we introduce the basic notation used throughout the paper and briefly review essential mathematical preliminaries. 

\medskip
\paragraph{Notation}
Let $\NN$ be the set of positive integers. 
For $n \in \NN$, define $[n]\coloneqq \{1, \dots, n\}$. 
We let $\RR$ and $\CC$ denote the fields of real and complex numbers, respectively, and let $\RRp \coloneqq \{ \lambda \in \RR: \lambda \geq 0 \}$. 
Boldface symbols (e.g., $\bv$, $\bSigma$) represent vectors, matrices, or operators, while calligraphic or blackboard bold symbols (e.g., $\cS$, $\RR$) denote sets.

\medskip
\paragraph{Linear algebra essentials}
Let $\cH$ be a Hilbert space (=complete inner product space) over the field $\FF \in \{ \RR, \CC \}$,with inner product $\langle \cdot, \cdot \rangle_{\cH}$. 
For any $\bv \in \cH$, let $\|\bv\| \coloneqq \langle \bv, \bv \rangle_{\cH}^{1/2}$. 
We write $\dim_{\FF} \cH$ for the dimension of $\cH$ as a vector space over $\FF$; we may omit the subscript if $\FF = \RR$. 
The dual space of $\cH$ is denoted by $\cH^*$. 

Given two Hilbert spaces $\cH_1, \cH_2$, we write $\cH_1 \cong \cH_2$ if they are isomorphic, $\cH_1 \otimes \cH_2$ for their tensor product, and $\cB(\cH_1, \cH_2)$ for the space of linear operators from $\cH_1$ to $\cH_2$. 
We define $\cB(\cH) \coloneqq \cB(\cH, \cH)$ as shorthand. 
An operator $\bA$ is self-adjoint if $\bA = \bA^*$, where $\bA^*$ is the Hermitian adjoint of $\bA$. 
The space of self-adjoint operators on $\cH$, denoted $\cBsa(\cH)$, forms a real vector subspace\footnote{A vector space over $\RR$. This fact is important because it enables defining a cone, which would be impossible if $\cBsa(\cH)$ were only a complex vector space.} of $\cB(\cH)$, even when $\FF = \CC$. 
The dimension of $\cBsa(\cH)$ is 
\begin{equation}\label{eqn:dimsa}
    \dimsa(\cH) \coloneqq
        \dim_{\RR} \cBsa(\cH) = 
            \begin{cases}
                \binom{ n + 1 }{ 2}  &   \text{if } \cH = \RR^n\\
                n^2                 &   \text{if } \cH = \CC^n.
            \end{cases}
\end{equation}
A self-adjoint operator $\bA$ is positive semidefinite, denoted $\bA \succeq 0$, if all its eigenvalues are nonnegative\footnote{If $\bA \in \cBsa(\cH)$, then all the eigenvalues of $\bA$ are real, whether the underlying field is $\RR$ or $\CC$.}. 
We define $\cBsa_+(\cH) \coloneqq \{ \bA \in \cBsa(\cH): \bA \succeq 0 \}$, which forms a proper, full-dimensional, convex cone in $\cBsa(\cH)$.

\medskip
\paragraph{Convex geometry basics}
A set $\cS \subseteq \RR^d$ is convex if, for every $\bx, \by \in \cS$, the line segment between $\bx$ and $\by$ is contained in $\cS$. 
For any nonempty set $\cS \in \RR^d$, let $\conv \cS$ denote its convex hull, i.e., the smallest convex set containing $\cS$. 
A set $\cC \subseteq \RR^d$ is a cone if it is invariant under positive scaling, meaning $\lambda \cdot \bx \in \cC$ for all $\lambda \in \RRp$ and $\bx \in \cC$. 
A nonzero element $\bv \in \cC$ is an extreme ray if there do not exist linearly independent $\bx, \by \in \cC$ and positive scalars $\alpha, \beta$ such that $\bv = \alpha \bx + \beta \by$. 
We define $\ext \cC \coloneqq \{ \bv \in \cC: \bv \text{ is an extreme ray of }\cC ~\text{and}~\|\bv\|=1 \}$.

By the spectral theorem \cite[Chapter 7]{hall2013quantum}, we have
\begin{equation}\label{eqn:cone_psd}
    \cBsa_+(\cH) = \cone \left\{ \bv \otimes \bv^*: \bv \in \cH \right\},
\end{equation}
where $\cone \cS$ is the conic hull of $\cS$, and moreover, $\ext \cBsa_+(\cH) = \left\{ \bv \otimes \bv^*: \bv \in \cH ~\text{and}~\|\bv\|=1 \right\}$. 
Consequently, any $\bA \in \cBsa_+(\cH)$ can be expressed as a nonnegative linear combination of at most $\dim_{\FF} \cH$ ``rank-1'' projectors $\bv \otimes \bv^* \in \ext \cBsa_+(\cH)$, by the spectral theorem.

\medskip
\paragraph{Computational complexity and NP‑hardness}
An algorithm runs in \emph{polynomial‑time} if its worst‑case runtime is $O(n^{c})$ for some constant $c$, where $n$ is the input size. 
The class $\mathbf P$ comprises decision problems\footnote{A decision problem is a computational problem that has a "yes" or "no" answer based on its input.} whose instances are 
solvable in polynomial time, whereas the class $\mathbf{NP}$ consists of problems whose proposed solutions (certificates) can be verified in polynomial time. 
NP hardness is a worst-case notion: a problem $\Pi$ is NP-hard if \emph{every} problem in $\mathbf{NP}$ can be reduced to it in polynomial time \cite{cook1971complexity,karp1972reducibility}. 
The widely held conjecture $\mathbf P\neq\mathbf{NP}$ implies that no polynomial‑time algorithm can solve NP-hard problems in the worst-case \cite{garey1979computers}. 
Alternative frameworks for assessing computational complexity, such as average‑case complexity \cite{levin1986average} and smoothed analysis \cite{spielman2004smoothed}, investigate algorithmic performance on typical or slightly perturbed problem inputs, illustrating practical tractability can coexist with worst‑case intractability. 

%% file: contents/03_separable_covariance.tex
\section{Multiway covariance and Kronecker-separability}\label{sec:sep_covariance}

\subsection{Covariance of multiway data}
Multiway data, represented as multi-dimensional arrays (e.g., matrices or tensors), can be vectorized by unpacking. 
For example, a real-valued $K$-way tensor $\bX\in\mathbb{R}^{d_1\times \cdots \times d_K}$ with modal dimensions $d_1, \dots, d_K$ can be represented as a vector $\bx \in \RR^d$ where $d = \prod_{i=1}^K d_i$. 
When modeled as a random variable, such data lie in a Hilbert space $\cH$ that is the tensor product of several Hilbert spaces. 
For instance, an $m \times n$ random matrix can be viewed as a random vector in $\RR^{m \times n} \cong \RR^m \otimes \RR^n$. 
More generally, a random multi-dimensional array with $K$ modes corresponds to a random vector in $\cH = \bigotimes_{k=1}^K \cH_k$.

\begin{definition}
\label{defn:multiway_random_vector}
    Let $K \in \NN$ and $\cH_1, \dots, \cH_K$ be Hilbert spaces. 
    A random vector $\bX \in \cH = \bigotimes_{k=1}^K \cH_k$ is called a \emph{$K$-way random vector}. 
    If $K \geq 2$ and there is no risk of confusion, we simply refer to $\bX$ as a \emph{multiway random vector}. 
\end{definition}

Next, we define the covariance of a random vector $\bX \in \cH$ as a positive semidefinite (PSD), self-adjoint operator on $\cH$. 
\begin{definition}\label{def:covariance}
    Let $\bX$ be a random vector that takes value in $\cH$. The \emph{covariance} of $\bX$ is defined as the expected value
    \[
        \Cov(\bX) \coloneqq \bbE \big[ (\bX - \bbE \bX) \otimes (\bX - \bbE \bX)^* \big].
    \]
    If $\bX$ is a $K$-way random vector ($K \geq 2$), the covariance of $\bX$ is referred to as a \emph{$K$-way covariance}, or simply as a \emph{multiway covariance}.  
\end{definition}

\begin{remark}\label{rem:cov_psd}
    For any random vector $\bX \in \cH$, the covariance $\Cov(\bX) \in \cBsa_+(\cH)$.
\end{remark}

The definition of covariance as a PSD operator aligns with the traditional definition as a PSD matrix when $\cH$ is finite-dimensional. 
For example, when $\cH = \RR^n$ and $\bX$ is represented as an $n \times 1$ column vector, the covariance of $\bX$ can be written an $n \times n$ PSD matrix $\Cov(\bX) = \bbE \big[ (\bX - \bbE \bX) (\bX - \bbE \bX)^{\top} \big]$. 
Conversely, any PSD operator $\bSigma \in \cBsa_+(\cH)$ corresponds to the covariance of some random vector in $\cH$, e.g., a standard Gaussian scaled by $\bSigma^{1/2}$. 

The covariance of a random vector $\bX$ in $\cH$ takes value in $\cBsa_+(\cH)$, a full-dimensional convex cone in the real vector space $\cBsa(\cH)$. 
Since $\cBsa(\cH)$ has dimension $\dimsa(\cH)$, expressing an arbitrary $\bSigma \in \cBsa_+(\cH)$ requires $\dimsa(\cH)$ real-valued parameters, where $\dimsa(\cH) \propto \dim_{\FF}(\cH)^2$, cf. \eqref{eqn:dimsa}. 
Given that $\dim_{\FF} \cH = \prod_{k=1}^K \dim_{\FF} \cH_k$, processing and storing large-scale multiway covariances pose significant computational challenges, unless $\bSigma$ has additional structure.

\subsection{Kronecker-separable covariance}\label{sec:separable_cov}
We define a class of multiway covariances that can be expressed as a sum of products of mode-wise covariances. 
\begin{definition}
\label{defn:separable}
    Let $\cH$ and $\cH_1, \dots, \cH_K$ be Hilbert spaces ($K \geq 2$). 
    A covariance $\bSigma \in \cBsa_+(\cH)$ is \emph{$(\cH_1, \cdots, \cH_K)$-Kronecker-separable} (K-S)
    if there exists a sequence $\Big( \big(\bSigma^{(1)}_a, \dots, \bSigma^{(K)}_a \big) \Big)_{a=1}^r$ for some $r \in \NN$ such that 
    \begin{equation}\label{eqn:separable}
            \bSigma = \sum_{a=1}^r \bSigma^{(1)}_a \otimes \cdots \otimes \bSigma^{(K)}_a
            \qquad\text{where}\quad
            \bSigma^{(k)}_a \in \cBsa_+(\cH_k), ~\forall (k, a).
    \end{equation}
    In particular, if $\cH = \bigotimes_{k=1}^{K} \cH_k$, we may simply say $\bSigma$ is \emph{Kronecker-separable}, omitting $(\cH_1, \dots, \cH_K)$. 
    The representation in \eqref{eqn:separable} is called a \emph{Kronecker decomposition} of $\bSigma$, and the smallest such $r$ is the \emph{Kronecker-separation rank} of $\bSigma$, denoted $\ranks(\bSigma)$.
\end{definition}
 
Without the PSD constraint on each factor, such a decomposition always exists for any multiway covariance, with $ r \leq d^2 / \max\{ d_1^2, \cdots, d_K^2 \}$ where $d = \prod_{a=1}^r d_a$. 
Such representation can be computed, for example, by Kronecker PCA \cite{tsiligkaridis2013covariance,greenewald2015robust}. 
However, the PSD requirement distinguishes Kronecker-separability from other decompositions, such as singular value decomposition (SVD) or tensor rank decomposition \cite{hitchcock1927expression, kruskal1977three, kolda2009tensor}. 
As we will see shortly in \Cref{sec:counter_examples}, some covariances do not admit a Kronecker decomposition as defined in \eqref{eqn:separable}, whereas SVD and tensor rank decomposition always exist with a finite decomposition rank.

Next, we present an alternative, equivalent definition of K-S covariances that is often more convenient for geometric analysis.
\begin{definition}\label{defn:separable.2}
    Let $\cH$ and $\cH_1, \dots, \cH_K$ be Hilbert spaces. 
    A covariance $\bSigma \in \cBsa_+(\cH)$ is \emph{$(\cH_1, \cdots, \cH_K)$-separable} 
    if there exists a sequence of vectors $\left( \big( \bv_a^{(k)} \in \cH_k \setminus \{0\}: k \in [K] \big)\right)_{a=1}^r$ for some $r \in \NN$ such that 
    \begin{equation}\label{eqn:separable.2}
        \bSigma = \sum_{a=1}^r \bv_a \otimes \bv_a^*  \quad\text{where}\quad \bv_a = \bigotimes_{k=1}^K \bv_a^{(k)}.
    \end{equation}
    If $\cH = \bigotimes_{k=1}^{K} \cH_k$, we simply say $\bSigma$ is separable, omitting $(\cH_1, \dots, \cH_K)$. 
    The smallest $r$ for which \eqref{eqn:separable.2} holds is called the \emph{separation rank} of $\bSigma$, denoted $\rank_s(\bSigma)$.
\end{definition}

If $\bSigma$ has separation rank $1$, then it satisfies two layers of ``rank-1'' constraints: 
(i) $\bSigma = \bv \otimes \bv^*$ is a rank-1 matrix as an operator on $\cH$; and 
(ii) $\bv = \bv^{(1)} \otimes \cdots \otimes \bv^{(k)}$ is itself a rank-1 tensor (each $\bv^{(k)}$ spanning one mode). 
Note that Condition (ii) is what differentiates separable covariances from non-separable ones. 

\begin{example}\label{example:rank_1}
    Let $\cH_1 = \cH_2 = \RR^2$ and let $\cH = \cH_1 \otimes \cH_2$. 
    Any vector $\bv \in \cH$ can be identified\footnote{With bases for $\cH_1$ and $\cH_2$ fixed, this corresponds to reshaping a $4 \times 1$ coordinate vector to a $2 \times 2$ matrix.} with a $2 \times 2$ matrix $\bM_{\bv}$, via the canonical isomorphism $\cH_2 \cong \cH_2^*$. 
    Observe that there exists a tuple $\big(\bv^{(1)}, \bv^{(2)} \big) \in \cH_1 \times \cH_2$ such that $\bv = \bv^{(1)} \otimes \bv^{(2)}$ if and only if $\bM_{\bv}$ is a rank-1 matrix.
\end{example}

Kronecker-separability (Definition \ref{defn:separable}) and separability (Definition \ref{defn:separable.2}) are equivalent.
\begin{proposition}\label{prop:separable_equivalence}
    A multiway covariance $\bSigma \in \cBsa_+(\cH)$ is K-S if and only if $\bSigma$ is separable. 
    Moreover, if $\bSigma$ is K-S, then $\ranks(\bSigma) \leq \rank_s(\bSigma) \leq \dimsa(\cH)$.
\end{proposition}

\medskip
\paragraph{Motivation: Low-rank covariance model}
Consider a $2$-way random vector with covariance $\bSigma \in \cBsa_+(\cH_1 \otimes \cH_2)$ that is K-S, for which $\ranks(\bSigma) \ll \min\{ \dim_{\FF} \cH_1, \dim_{\FF} \cH_2 \}$. 
The number of parameters required to represent $\bSigma$ is $\ranks(\bSigma) \cdot \left[ 1 + \dimsa(\cH_1) + \dimsa(\cH_2) \right] $, which is significantly fewer than the $\dimsa( \cH_1 \otimes \cH_2 )$ parameters needed without Kronecker-separability. 
For example, when $\cH_1 = \RR^{n_1}$ and $\cH_2 = \RR^{n_2}$, we have:
\begin{align*}
    \ranks(\bSigma) \cdot \left[ 1 + \dimsa(\cH_1) + \dimsa(\cH_2) \right]
        &= \ranks(\bSigma) \cdot \left[ 1 + \binom{n_1 + 1}{2} + \binom{n_2 + 1}{2} \right]\\
        &\ll \binom{n_1 n_2 + 1}{2} 
        = \dimsa( \cH_1 \otimes \cH_2 ).
\end{align*}
The model complexity is significantly reduced for K-S multiway covariances, especially when $K$ is large or the modal dimensions $d_1, d_2$ are comparable to each other. 

\subsection{Examples of Kronecker-separable covariances}\label{sec:examples}
We provide examples of K-S covariance models in two categories: (i) those with a stochastic representation, interpretable as the covariances of some random variables; and (ii) those introduced as mathematical constructs without a stochastic representation. 
Though our exposition focuses on the case $K=2$, these examples readily generalize to $K \geq 2$.

\medskip
\paragraph{Kronecker-separable covariances with stochastic representation}
Consider a random matrix $\bX \in \RR^{m \times n}$. 
We describe two data models for $\bX$ that admit simple stochastic representations.

\begin{example}[Matrix normal model \cite{dawid1981some}]\label{example:matrix_normal}
    Let $\bX = \bA \bZ \bB$ where $\bA \in \RR^{m \times m}$, $\bB \in \RR^{n \times n}$, and $\bZ$ is a random matrix with i.i.d. standard Gaussian entries. 
    Then $\bX$ is a Gaussian random matrix with mean zero and covariance $\bA \bA^{\top} \otimes \bB \bB^{\top}$, called the matrix normal covariance. 
\end{example}

\begin{example}[Sylvester model \cite{wang2020sylvester}]\label{example:sylvester}
    Let $\bX$  be generated by the Sylvester matrix equation:
    \[
        \bX = \bA \bZ + \bZ \bB^{\top},
    \]
    where $\bA \in \RR^{m \times m}$, $\bB \in \RR^{n \times n}$ are symmetric matrices, and $\bZ$ is a random matrix with i.i.d. standard Gaussian entries. 
    With vectorization, this equation can be equivalently written as
    \begin{equation}\label{eqn:sylvester_vec}
        \rvec(\bX) = \left( \bA \oplus \bB \right)  \rvec(\bZ),
    \end{equation}
    yielding covariance $\bSigma = \Cov(\bX) = \left( \bA \oplus \bB \right) \left( \bA \oplus \bB \right)^{\top} = \bA^2 \otimes \bI_n + 2 \bA \otimes \bB + \bI_m \otimes \bB^2$.
\end{example}

In the matrix normal model, the covariance $\bSigma = \bA \bA^{\top} \otimes \bB \bB^{\top}$ represents the simplest K-S form with $K=2$ and $\ranks(\bSigma) = 1$. 
This model extends to $K \geq 3$ \cite{ohlson2013multilinear, manceur2013maximum} via the so-called $K$-mode Kronecker product (KP) model\footnote{The KP model is also referred to as the ``separable'' covariance model in some literature, e.g., \cite{hoff2011separable, ohlson2013multilinear}.}, which posits
\begin{equation}\label{eqn:kp_model}
    \bSigma = \bSigma^{(1)} \otimes \cdots \otimes \bSigma^{(K)}.
\end{equation}
The KP model has been applied in many data science applications including longitudinal data analysis \cite{galecki1994general, werner2008estimation}, spatial data analysis \cite{cressie2015statistics}, and missing data imputation \cite{allen2010transposable, hoff2011separable}. 

The ``Sylvester model'' is widely used in physical modeling \cite{kressner2010krylov, wang2020sylvester}, where $\bX$ represents the state of a system governed by a physical operator $\bA \oplus \bB$ (e.g., the Laplacian diffusion operator), and $\bZ$ denotes the external random excitation of the system. 
This model, and the Sylvester equation, have been extensively studied in numerical analysis, particularly in the context of finite-difference discretization of elliptical PDEs \cite{grasedyck2004existence, kressner2010krylov}. 

Both examples illustrate how the model assumptions constrain the covariance structure, reducing the number of parameters needed for representation from $\Theta(m^2 n^2)$ to $\Theta(m^2 + n^2)$.

\medskip
\paragraph{Kronecker-separable covariances without a stochastic representation}
Below are two examples where Kronecker-separability is utilized to approximate the (inverse) covariance. 

\begin{example}[Kronecker PCA model \cite{greenewald2015robust}]\label{example:KPCA_model}
    Let $K = 2$ and $r \in \NN$. 
    The Kronecker PCA model for the covariance of $\bX$ is of the form
    \[
        \bSigma = \sum_{a=1}^r \bA_a \otimes \bB_a  
    \]
    where $\bA_a \in \RR^{m \times m}$ and $\bB_a \in \RR^{n \times n}$ are symmetric matrices (however, not necessarily PSD). 
\end{example}

\begin{example}[Kronecker sum model \cite{kalaitzis2013bigraphical, greenewald2019tensor}]\label{example:KS_model}
    The Kronecker sum model assumes that the covariance or inverse covariance of $\bX$ can be expressed as a Kronecker sum of $K$ PSD factors. 
    For $K=2$, it takes the form  
    \[
        \bSigma = \bA \oplus \bB
            = \bA\otimes \bI_n + \bI_m \otimes \bB
    \]
    where $\bA \in \cBsa_+(\RR^m)$ and $\bB \in \cBsa_+(\RR^n)$.
\end{example}

Like standard PCA, Kronecker PCA (Example \ref{example:KPCA_model}) captures the most significant components through low-rank approximation ($r\ll \min\{m,n\}$) \cite{tsiligkaridis2013covariance, greenewald2015robust}. 
However, unlike the Kronecker sum model (Example \ref{example:KS_model}), the lack of PSD constraints on $\bA_a$ and $\bB_a$ prevents interpreting them as mode-wise covariance factors. 
The Kronecker sum model is especially useful in probabilistic graphical modeling, where it has a natural interpretation as a graph Cartesian product \cite{kalaitzis2013bigraphical, greenewald2019tensor}. 
Both models substantially reduce the dimensionality required to represent the covariance, making them effective for high-dimensional data analysis.

\subsection{Counter-examples to Kronecker-separability}\label{sec:counter_examples}
We present three examples of covariance matrices that are not K-S, drawing on and adapting results from quantum physics. 
First, when $\min\{ \dim \cH_1, \dim \cH_2 \} \geq 2$, there exist $\bSigma \in \cBsa_+(\cH_1 \otimes \cH_2)$ that are not K-S. 

\begin{cexample}[Bell state \cite{bell1964einstein}] \label{example:entangled}
    Let $\cH_1 = \cH_2 = \FF^2$ and let $\bv_{\bell} = \be^{(1)}_1 \otimes \be^{(2)}_1 + \be^{(1)}_2 \otimes \be^{(2)}_2 \in \cH_1 \otimes \cH_2$ where $\{ \be_1, \be_2 \}$ is the standard basis of $\FF^2$. Then the covariance $\bSigma_{\bell} = \bv_{\bell} \otimes \bv_{\bell}^*$ is not K-S. 
    In matrix form,
    \begin{equation}\label{eqn:Sigma_entangled}
        \bSigma_{\bell} = \begin{bmatrix} 1&0&0&1\\0&0&0&0\\0&0&0&0\\1&0&0&1\end{bmatrix} \in \FF^{4 \times 4} \cong \FF^{2 \times 2} \otimes \FF^{2 \times 2}.
    \end{equation}
\end{cexample}

The Bell state $\bSigma_{\bell}$ is a classic example of an entangled quantum state and can be proven non-separable using the Peres-Horodecki criterion \cite{peres1996separability, horodecki2001separability}, which asserts that if $\bSigma$ is K-S, then its partial transpose must be PSD; see \Cref{sec:verifiable}. 
This example generalizes to higher dimensions. 
For $K =3$ and $\cH_1 = \cH_2 = \cH_3 = \FF^2$, the covariance $\bSigma_{\bell} \otimes \bI_2$ is not K-S; similarly, such examples exist for any $K \geq 3$. 
If $\dim \cH_1, \dim \cH_2 > 2$, we can embed $\bSigma_{\bell}$ into arbitrary $2$-dimensional subspaces of $\cH_1$ and $\cH_2$. 
Thus, for any nontrivial configuration with $K \geq 2$ and at least two component dimensions $\min\{ d_1, d_2 \} \geq 2$, non-K-S covariances exist.

To address potential concerns about degeneracy ($\bSigma$ being PSD but not positive definite), we provide an example of a positive definite covariance that is not K-S.
\begin{cexample}\label{example:aug_bell}
    Consider the matrix $\bSigma_{\bell} + \lambda \cdot \bI_4$ with $\lambda \in \RRp$. 
    The eigenvalues of this matrix are $\lambda$ (multiplicity 3) and $2 + \lambda$, making it positive definite for any $\lambda > 0$. 
    For $\lambda < 1$, $\bSigma_{\bell} + \lambda \cdot \bI_4$ is not K-S because its partial transpose\footnote{See \Cref{sec:verifiable}.} $\Phi(\bSigma_{\bell} + \lambda \cdot \bI_4)$ has a negative eigenvalue ($-1 + \lambda$). 
    Therefore, for any $\lambda \in (0,1)$, the $\bSigma_{\bell} + \lambda \cdot \bI_4$ is positive definite but not K-S. 
\end{cexample}

Lastly, we consider a random covariance model that is not K-S with high probability.

\begin{cexample}\label{example:Wishart}
    Let $d \in \NN$ and let $\bX \in \RR^{d^2}$ be the vectorization of a Gaussian random matrix with mean zero and $\Cov(\bX) = \bI_{d^2}$. 
    Let $\big( \bX_1, \dots , \bX_n \big)$ be $n$ i.i.d. random samples drawn from this distribution. 
    According to Theorem \ref{thm:wishart} below, with overwhelming probability, the sample covariance $\hbSigma_n = \frac{1}{n} \sum_{i=1}^n \bX_i \bX_i^{\top}$ is not K-S, unless $n$ is sufficiently large.
\end{cexample}

\begin{theorem}[adapted from {\cite[Theorem 2.3]{aubrun2014entanglement}}]\label{thm:wishart}
    There exist constants $C, c > 0$ and a function $\theta: \NN \to \RRp$ satisfying $c d^3 \leq \theta(d) \leq C d^3 \log^2 d$ for all $d \in \NN$ such that for any $\varepsilon > 0$,
    \begin{align*}
        \text{if } n \leq (1-\varepsilon) \, \theta(d), \qquad &\text{then } \Pr \big[ \hbSigma_n \text{ is K-S} \big] \leq 2 \exp \left( - c_{\varepsilon} \, d^3 \right),\\
        \text{if } n \geq (1+\varepsilon) \, \theta(d), \qquad &\text{then } \Pr \big[ \hbSigma_n \text{ is not K-S} \big] \leq 2 \exp \left( - c_{\varepsilon} \, n \right),
    \end{align*}
    where $c_{\varepsilon}$ is a positive constant depending only on $\varepsilon$. 
    \label{thm:example3}
\end{theorem}
\begin{proof}[Proof of Theorem \ref{thm:wishart}]
     Translating the phase transition result for quantum entanglement {\normalfont\cite[Theorem 2.3]{aubrun2014entanglement}} in the context of Kronecker-separability of covariance implies this theorem.
\end{proof}

Since $\rank\, \hbSigma_n = \min\{ d^2, n \}$ with probability $1$, $\hbSigma_n$ is positive definite with probability $1$ when $n \geq d^2$. 
Theorem \ref{thm:wishart} implies that when $d$ is sufficiently large so that $\theta(d) \geq 4 d^2$, a random covariance $\hbSigma_n$ with $n = 2d^2$ is positive definite with probability $1$, but it is not K-S with probability at least $1 - 2 \exp( - c_{1/2} \cdot d^3 )$. 

%% file: contents/04_general_results.tex
\section{Negative results on Kronecker-separability}\label{sec:questions}

In this section, we explore the limitations of Kronecker-separability by addressing three key questions from \Cref{sec:introduction}: representability (Question \ref{question:representability}), certifiability (Question \ref{question:certifiability}), and approximability (Question \ref{question:approximability}). 
We examine how common non-separable covariances are, then we address the complexity of certifying separability, and finally we discuss the difficulty of Kronecker-separable (K-S) approximation.

\subsection{Abundance of non-Kronecker-separable covariances}\label{sec:separability}
In \Cref{sec:sep_covariance}, we presented examples of K-S covariances and counterexamples. 
Here, we consider the broader question: \emph{“How frequently do K-S covariances arise in nature?”} 
We begin by introducing a way to quantify the size of the subset of K-S covariances, relative to the set of all multiway covariances. 
Then we show that the relative size of the non-separable subset dominates that of the K-S subset, when the dimension of the underlying multiway space $\cH$ is large (Theorem~\ref{thm:abundance}).


\subsubsection{Quantifying the size of sets of covariances}\label{sec:size_set}
Let $\cH = \bigotimes_{k=1}^K \cH_k$ be a multiway Hilbert space. 
Recall that $\cBsa_+(\cH)$ denotes the set of all possible multiway covariances on $\cH$, or equivalently, all positive semidefinite (PSD) operators on $\cH$. 
Define 
\[
    \cBsep(\cH) \coloneqq \left\{ \bSigma \in \cBsa_+(\cH) : \bSigma \text{ is Kronecker-separable} \right\}.
\]

Both $\cBsa_+(\cH)$ and $\cBsep(\cH)$ are convex cones in $\cBsa(\cH)$, the space of self-adjoint operators on $\cH$. 
Since these cones are unbounded, directly comparing their sizes is not meaningful. 
To remedy this, we focus on their \emph{unit-trace} affine slices, which are bounded convex sets:
\begin{equation}\label{eqn:unit_trace}
    \begin{aligned}
        \cSsa_+(\cH) \coloneqq \cBsa_+(\cH) \cap \cT_1,
        \qquad
        \cSsep(\cH) \coloneqq \cBsep(\cH) \cap \cT_1,
    \end{aligned}
\end{equation}
where $\cT_1 \coloneqq \{ \bSigma \in \cBsa(\cH): \tr \bSigma = 1 \}$ is the unit-trace affine subspace\footnote{This subspace is orthogonal to the ray spanned by the identity operator on $\cH$, which lies in the interiors of both $\cSsa_+(\cH)$ and $\cSsep(\cH)$.} of codimension $1$.

\medskip
\paragraph{``Width'' of a set}
We use two standard ``width'' measures from convex geometry to gauge the relative sizes of $\cSsep(\cH)$ and $\cSsa_+(\cH)$; see the supplement (\ref{sec:width}) for more details.

\begin{definition}\label{defn:volume_radius}
    For a nonempty, bounded, Lebesgue measurable set $\cS \subset \RR^d$, the \emph{volume radius} of $\cS$ is 
     $
        \vrad(\cS) \coloneqq \Big(\frac{\vol(\cS)}{\vol(\bbB^d_2)}\Big)^{1/d}
    $
    where $\bbB_2^d$ is the $d$-dimensional unit $\ell_2$-norm ball and $\vol(\cdot)$ denotes the $d$-dimensional Lebesgue measure.
\end{definition}

\begin{definition}\label{defn:mean_width}
    For a nonempty, bounded set $\cS \subset \RR^d$, the \emph{mean width} of $\cS$ is 
    $
        w(\cS) \coloneqq \bbE_{\bu \sim \mu} \big[ \sup_{\bx \in \cS} \langle \bu, \bx \rangle \big]
            = \int_{\bbS^{d-1}} \sup_{\bx \in \cS} \langle \bu, \bx \rangle ~d\mu(\bu)
    $
    where $\bbS^{d-1}$ is the unit sphere in $\RR^d$ and $\mu$ is the Haar probability measure on $\bbS^{d-1}$.
\end{definition}

When $\cS$ is an $\ell_2$-ball in $\RR^d$ of radius $r$, its volume radius and mean width are both equal to $r$. 
These notions of ``width'' enable comparing the size of $\cSsep(\cH)$ relative to $\cSsa_+(\cH)$.

\subsubsection{Asymptotic ratio of Kronecker-separable covariances}\label{sec:abundance}
Now we restate the representability question (Question~\ref{question:representability}) more formally as follows.

\begin{problem}
    \label{problem:separable}
    Given $\cH = \bigotimes_{k=1}^K \cH_k$ with $K \geq 2$ and $\dim \cH_k \geq 2, ~\forall k \in [K]$, how large is $\cSsep(\cH)$ relative to $\cSsa_+(\cH)$, measured by the ratios:
    \[
        \rho_v (\cH) \coloneqq \frac{ \vrad \big( \cSsep(\cH)\big) }{ \vrad\big( \cSsa_+(\cH) \big) }
        \qquad
        \text{and}
        \qquad
        \rho_w(\cH) \coloneqq \frac{ w \big( \cSsep(\cH) \big) }{ w\big( \cSsa_+(\cH) \big) }?
    \]
\end{problem}

Both ratios $\rho_v(\cH)$ and $\rho_w(\cH)$ lie in $[0, 1]$. 
If most $\bSigma \in \cBsa_+(\cH)$ were K-S, we would expect these ratios to approach 1. 
However, the following theorem shows that they diminish to $0$ as the dimensionality of $\cH$ increases.

\begin{theorem}\label{thm:abundance}
    Let $d, K \in \NN$ such that $d \geq 2$ and $K \geq 2$. 
    The following bounds hold:
    \begin{align*}
        \max\left\{ \rho_v\left( \CC^{d} \otimes \CC^{d} \right), ~ \rho_w\left( \CC^{d} \otimes \CC^{d} \right) \right\}
            &\leq \frac{10}{\sqrt{d}},\\
        \max\left\{ \rho_v\left( \big(\CC^{2}\big)^{\otimes K} \right), ~ \rho_w\left( \big(\CC^{2}\big)^{\otimes K} \right)\right\}
            &\leq \sqrt{48 e} \frac{\sqrt{ K \log K }}{2^{K/2}}.
    \end{align*}
\end{theorem}

Hence, both $\rho_v(\cH)$ and $\rho_w(\cH)$ diminish to $0$ in the two scenarios considered: (1) the bipartite case, where $\cH = \CC^d \otimes \CC^d$ and $d \to \infty$; and (ii) the multipartite case, where $\cH = (\CC^2)^{\otimes K}$ and $K \to \infty$. 
Thus, almost all covariances are non-separable in high dimensions---i.e., the probability that a randomly drawn PSD operator $\bSigma \in \cBsa_+(\cH)$ admits a Kronecker decomposition of the form \eqref{eqn:separable} becomes negligible, regardless of rank parameter $r \in \NN$.
The proof of Theorem \ref{thm:abundance} can be found in the supplement (\ref{sec:proof_thm.1}). 

Given this abundance of non-K-S covariances, it is natural to consider Questions \ref{question:certifiability} and \ref{question:approximability} that can be restated as follows. 
\begin{itemize}[topsep=6pt, itemsep=3pt]
    \item
    \emph{\textbf{Certification}} (\Cref{sec:certification}): 
    Can we efficiently determine whether a given covariance $\bSigma \in \cBsa_+(\cH)$ is K-S or not?
    
    \item 
    \emph{\textbf{Approximation}} (\Cref{sec:approximation}): 
    If $\bSigma \in \cBsa_+(\cH)$ is not K-S, can we compute a K-S approximation of it?
\end{itemize}
We address these questions in the sections that follow.

\subsection{Certification of Kronecker-separability}\label{sec:certification}
We restate Question \ref{question:certifiability} as follows.

\begin{problem}\label{problem:certification}
    Let $\cH = \bigotimes_{k=1}^K \cH_k$. Given $\bSigma \in \cBsa_+(\cH)$, decide whether $\bSigma \in \cBsep(\cH)$ or not.
\end{problem}

When restricted to the unit-trace setting ($\tr(\bSigma)=1$), this reduces to the so-called \emph{quantum separability problem}. 
Since Kronecker-separability of covariances is invariant under positive scaling, the unit-trace constraint does not alter the essence of Problem \ref{problem:certification}, making it equivalent to the quantum separability problem.

For clarity, we focus on the bipartite case $\cH = \cH_1 \otimes \cH_2$ in this section. 
We first overview some sufficient and necessary conditions for certifying separability in \Cref{sec:verifiable}, and then discuss NP-hardness of Problem \ref{problem:certification} in \Cref{sec:NPhard_wmem}. 
Since certifying separability is already NP-hard in the bipartite setting, this hardness extends to the general case with $K \geq 2$.

\subsubsection{General certificates of Kronecker-separability}\label{sec:verifiable}
In quantum physics, several sufficient or necessary conditions have been developed that can help certify the separability of quantum states. 
These criteria can prove or disprove whether a given density operator represents a separable quantum state, but there is no general criterion that exactly characterize all separable states. 
Here, we adapt two fundamental results from quantum information theory to our setting. 
For more details on quantum separability, see \cite{horodecki2009quantum, guhne2009entanglement} and the references therein. 

\medskip
\paragraph{Sufficient condition: Proximity to the identity \cite{gurvits2002largest}} 
For any $\bSigma \in \cBsa_+(\FF^{d_1} \otimes \FF^{d_2})$, 
\begin{equation}\label{eqn:suff_gurvits_barnum}
    \left\| \frac{1}{\tr \bSigma} \bSigma - \frac{1}{d_1 d_2} \bI_{d_1 d_2} \right\|^2 \leq \frac{1}{d_1 d_2 ( d_1 d_2 - 1 )}
        ~~\implies~~
        \bSigma \in \cBsep\big( \FF^{d_1} \otimes \FF^{d_2} \big).
\end{equation}
Here, $\| \cdot \|$ is the Hilbert-Schmidt norm and $\bI_{d_1 d_2}$ denotes the identity matrix of size $d_1 d_2$ by $d_1 d_2$. 
The above implies that covariances close to the scaled identity are K-S.

\medskip
\paragraph{Necessary condition: The Peres-Horodecki criterion \cite{peres1996separability, horodecki2001separability}} 
Also known as the \emph{positive partial transpose} (PPT) criterion, the Peres–Horodecki criterion states that if $\bSigma$ is the density matrix of a separable bipartite quantum state, then its partial transpose $\Phi(\bSigma)$ must be PSD. 

To define the partial transpose, consider the canonical isomorphism $\cB(\cH_1 \otimes \cH_2) \cong \cB(\cH_1)\otimes \cB(\cH_2)$.  Any operator $\bA \in \cB(\cH_1 \otimes \cH_2)$ can be written as 
$\bA = \sum_{a=1}^N \bA_a^{(1)} \otimes \bA_a^{(2)}$, 
where each $\bA_a^{(1)} \in \cB(\cH_1)$ and $\bA_a^{(2)} \in \cB(\cH_2)$. 
The partial transpose with respect to the second factor is the linear map $\Phi: \cB(\cH_1 \otimes \cH_2)\!\to\!\cB(\cH_1 \otimes \cH_2)$ defined by
\[
  \Phi(\bA)\;=\;\sum_{a=1}^N \bA_a^{(1)} \otimes \bigl(\bA_a^{(2)}\bigr)^\top.
\]
Hence, if $\bSigma \in \cBsa_+(\cH_1 \otimes \cH_2)$ is K-S, say $\bSigma = \sum_{a=1}^r \bSigma_a^{(1)} \otimes \bSigma_a^{(2)}$ with each $\bSigma_a^{(k)}\in\cBsa_+(\cH_k)$, then $\Phi(\bSigma)\;=\;\sum_{a=1}^r \bSigma_a^{(1)} \otimes \bigl(\bSigma_a^{(2)}\bigr)^\top$ is also K-S.  
Thus, for $\bSigma$ to be separable, a necessary condition is that $\Phi(\bSigma)$ be PSD.  
For instance, $\bSigma_{\bell}$ in Example~\ref{example:entangled} is not K-S because its partial transpose $\Phi(\bSigma_{\bell})$ is not PSD (see \ref{sec:addtl_bell_state}).  

While the PPT criterion can detect certain entangled (non-separable) states when $\Phi(\bSigma)$ is not PSD, it is only a necessary condition.  
It becomes sufficient in special cases, e.g.\ when $d_1d_2 \le 6$~\cite{horodecki2001separability} or when $\rank(\bSigma)\le\max(d_1,d_2)$~\cite{horodecki2000operational}.

\subsubsection{NP-hardness of certifying Kronecker-separability}\label{sec:NPhard_wmem}
Although several conditions can test separability, the quantum separability problem---and by extension, the Kronecker-separability problem---is fundamentally difficult. 
Specifically, certifying separability of a quantum state is known to be NP-hard. 
Consequently, certifying Kronecker-separability for a multiway covariance (Problem \ref{problem:certification}) is also NP-hard. 
This section outlines that complexity result.

Gurvits~\cite{gurvits2003classical} proved the quantum separability problem to be NP-hard by relating it to the weak membership problem for separable density operators. The weak membership problem is a decision problem that verifies membership in a set up to a specified error tolerance.

\begin{definition}[Weak membership problem]\label{def:weak_membership}
    Let $\cK \subset \RR^d$ be a convex body\footnote{A convex body is a compact convex set with non-empty interior.} and let $\delta \geq 0$ be a rational number. 
    The \emph{weak membership problem for $\cK$ with error $\delta$}, denoted $\wmem_{\delta}(\cK)$, is defined as follows: given a rational vector $\bx \in \RR^d$,  
    \begin{enumerate}[label=\arabic*)]
        \item
        either assert $\bx \in \cK_{+\delta} \coloneqq \bigcup_{\bx \in \cK} B(\bx, \delta)$,
        \item
        or assert $\bx \not\in \cK_{-\delta} \coloneqq \{\bx \in \cK: B(\bx, \delta) \subset \cK \}$,
    \end{enumerate}
    where $B(\bx, \delta) \coloneqq \{ \by \in \RR^d: \| \by - \bx \| \leq \delta \}$.
\end{definition}

Gurvits \cite{gurvits2003classical} showed that $\wmem_{\delta}(\cSsep(\CC^{d_1}\otimes \CC^{d_2}))$ is NP-hard, meaning no polynomial-time (in $d_1,d_2,\delta$) algorithm can be guaranteed to correctly decide separability with a small tolerance $\delta$.  
Later, Aaronson (as documented by Ioannou \cite{ioannou2007computational}) noted that Gurvits' result requires $\delta$ to be exponentially small in $d_1,d_2$. 
A stronger NP-hardness result was subsequently proved by Gharibian \cite{gharibian2008strong}, showing that separability remains NP-hard under weaker (i.e.\ larger) error tolerances. 
As this result is needed for the K-S result to be stated in Corollary \ref{coro:kron_sep}, we rephrase Gharibian’s theorem below and provide a proof in the supplement (\ref{sec:sketch_NPhard}).

\begin{theorem}[{\cite[Theorem 1]{gharibian2008strong}}]\label{thm:NP-hard}
    Let $d_1, d_2 \in \NN$ such that $d_1 \geq d_2 \geq 2$. 
    Then there exists a constant $C > 0$ such that $\wmem_{\delta}\left( \cSsep( \CC^{d_1} \otimes \CC^{d_2} ) \right)$ is NP-hard for all $\delta \leq C \cdot d_1^{-16} d_2^{-20.5}$. 
    In other words, $\wmem_{\delta}\left( \cSsep( \CC^{d_1} \otimes \CC^{d_2} ) \right)$ is strongly NP-hard.
\end{theorem}

Translating this to Kronecker-separability of covariances, we obtain the following.

\begin{corollary}\label{coro:kron_sep}
    Let $d_1, d_2 \in \NN$ such that $d_1 \geq d_2 \geq 2$, and let $\cH = \CC^{d_1} \otimes \CC^{d_2}$. 
    Then certifying Kronecker-separability of multiway covariances (Problem \ref{problem:certification}) is strongly NP-hard.
\end{corollary}
\begin{proof}[Proof of Corollary \ref{coro:kron_sep}]
    A covariance $\bSigma\in\cBsa_+(\cH)$ is K-S if and only if $\tfrac{1}{\tr(\bSigma)}\bSigma$ is a separable density operator. 
    Theorem~\ref{thm:NP-hard} implies the NP-hardness of Problem~\ref{problem:certification}.
\end{proof}

Thus, certifying whether a given covariance is K-S is computationally intractable in the worst case, mirroring the NP-hardness of the quantum separability problem.

\subsection{Kronecker-separable approximation of covariance}\label{sec:approximation} 

Even when a covariance is not exactly K-S, approximating it with a K-S model may still be useful. 
Given $\bSigma \in \cBsa_+(\cH)$, we seek a surrogate $\bSigma' \in \cBsep(\cH)$ close to $\bSigma$. 
We define the \emph{best K-S approximation} as follows.

\begin{definition}
\label{defn:sep_approx}
    The \emph{best Kronecker-separable (K-S) approximation} of $\bSigma \in \cBsa(\cH)$ is 
    \begin{equation}\label{eqn:best_separable}
        \pisep(\bSigma) \coloneqq \argmin_{\bSigma' \in \cBsep(\cH)} \| \bSigma' - \bSigma \|^2
    \end{equation}
    where $\| \cdot \|$ is the Hilbert-Schmidt norm in $\cBsa(\cH)$.
\end{definition}

Since $\cBsep(\cH)$ is a nonempty, closed, convex subset of $\cBsa_+(\cH)$, every $\bSigma \in \cBsa_+(\cH)$ admits a unique minimizer in \eqref{eqn:best_separable}, so $\pisep(\bSigma)$ is well-defined. 
Now we restate \Cref{question:approximability} formally.
\begin{problem}\label{problem:approximation}
    Let $\cH = \bigotimes_{k=1}^K \cH_k$. 
    Given $\bSigma \in \cBsa_+(\cH)$, compute $\pisep(\bSigma)$.
\end{problem}

Computing the best K-S approximation is challenging, as stated below.

\begin{theorem}\label{thm:approx}
    Let $d_1, d_2 \in \NN$ such that $d_1 \geq d_2 \geq 2$, and let $\cH = \CC^{d_1} \otimes \CC^{d_2}$. 
    Then the K-S approximation problem (Problem~\ref{problem:approximation}) is NP-hard.
\end{theorem}
\begin{proof}[Proof of Theorem \ref{thm:approx}]
    Suppose there were a polynomial-time algorithm that can approximate $\pisep(\bSigma)$ within Hilbert–Schmidt distance $\delta$, for any arbitrary $\bSigma$ and $\delta > 0$. 
    Then, for any $\bSigma\in\cBsa_+(\cH)$, this would allow solving the weak membership problem (Definition~\ref{def:weak_membership}) in polynomial time by checking whether $\|\hbSigma - \bSigma\|\le\delta$. 
    Since weak membership is NP-hard (\Cref{thm:NP-hard}), no such polynomial-time algorithm exists.
\end{proof}

Thus, there is no polynomial-time method to find $\pisep(\bSigma)$ for all $\bSigma\in\cBsa_+(\cH)$, highlighting the computational difficulty of such approximation problems in general.

\subsection{Summary and implications} 
We have identified several fundamental barriers to efficiently approximating multiway covariances using K-S models. 

\begin{enumerate}
    \item
     A K-S decomposition is not always possible for all PSD operators. 
     A large fraction of multiway covariances cannot be represented in this way.
     
    \item 
    The determination of whether or not a given covariance is K-S is NP-hard in general.
    
    \item
    The computation of the best K-S approximation is also NP-hard.
\end{enumerate}
These results highlight the need for caution when applying K-S models, especially in high-dimensional settings. 
They establish that K-S representations are not universally valid models for multiway covariance without additional assumptions on data or covariance structure.

%% file: contents/05_experiments.tex
\section{Kronecker-separable approximation: Numerical experiments}
\label{sec:experiments}

Notwithstanding the NP-hardness results in \Cref{sec:questions}, Kronecker-separable (K-S) models have often proved effective in approximating multiway covariances in real-world applications. 
In this section, we investigate two iterative algorithms---Frank–Wolfe (FW) and Gradient descent (GD)---for computing K-S approximations, and provide empirical evidence that these methods often succeed in finding high-quality K-S approximations, even though no polynomial-time method can universally guarantee a globally optimal solution.

We begin by reviewing the \emph{unit-trace} formulation of the approximation problem (a common convention in quantum mechanics). 
Next, we describe the two algorithms (FW and GD) in a \emph{general $K$-partite setting}, i.e., $\cH = \bigotimes_{i=1}^K \cH_i$. 
Then we present two experiments:
\begin{itemize}
    \item 
    \textbf{Experiment 1} (\Cref{sec:exp_random_separable}) considers random separable states to verify whether each method can reliably recover the correct decomposition.
    \item 
    \textbf{Experiment 2} (\Cref{sec:quantum_states}) focuses on canonical bipartite quantum states (isotropic and Werner), whose exact separability thresholds and “best separable approximations” (BSA) are known, providing a more challenging benchmark.
\end{itemize}
For concreteness, most numerical illustrations focus on bipartite ($K=2$) scenarios, but the methods naturally extend to higher-order tensor decompositions.

\medskip
\paragraph{Unit-trace formulation of Kronecker-separable approximation} 
Recall from \eqref{eqn:best_separable} that the K-S approximation seeks to solve $\min_{\bSigma' \in \cBsep(\cH)} \| \bSigma' - \bSigma \|^2$. 
This is equivalent to computing the best \emph{unit-trace} separable approximation $\upisep(\bSigma)$ of the normalized operator $\bSigma/\tr(\bSigma)$, i.e., 
\begin{equation}\label{eqn:KS_approx_unit}
    \upisep(\bSigma) 
    \coloneqq \underset{\bX \,\in\, \cSsep(\cH)}{\argmin}
        \,\biggl\| \, \bX - \frac{\bSigma}{\tr(\bSigma)} \, \biggr\|^2,
   \qquad
   \cSsep(\cH)=\cBsep(\cH)\,\cap\,\cT_1.
\end{equation}
Once $\upisep(\bSigma)$ is found, scaling by trace recovers $\pisep(\bSigma)$ because
\[
    \upisep(\bSigma) 
    = \frac{1}{\tr\!\bigl(\pisep(\bSigma)\bigr)} 
        \,\pisep(\bSigma)
    \qquad\text{and}\qquad
    \pisep(\bSigma) 
    = \frac{\langle\upisep(\bSigma), \,\bSigma\rangle}{\|\upisep(\bSigma)\|^2}
        \,\upisep(\bSigma).
\]
Hence, focusing on the unit-trace setting does not alter the essential nature of the problem. 
This formulation aligns with quantum-mechanical conventions, where $\bSigma/\tr(\bSigma)$ is viewed as a density matrix, and $\cSsep(\cH)$ as the set of separable states.

\subsection{Iterative algorithms for Kronecker-separable approximation}
\label{sec:exp_setup_algos}
In this section, we present two algorithms for approximating the solution of 
\begin{equation}\label{eqn:KS_approx_unit_iter}
   \min_{\bX \in \cSsep(\cH)} \bigl\| \bX - \bSigma \bigr\|^2,
   \quad
   \bSigma \in \cSsa_+(\cH).
\end{equation}

\subsubsection{Frank-Wolfe algorithm}
\label{sec:alg_Frank_Wolfe}

The Frank-Wolfe (FW) method~\cite{frank1956algorithm} is a classical first-order algorithm well-suited for optimization over convex hulls. 
In our setting, the feasible set $\cSsep(\cH)$ is the convex hull $\conv(\cA)$ of all rank-1 product projectors, where
\begin{equation}\label{eqn:setA}
    \cA 
        = \Bigl\{
          (\bv_1 \otimes \cdots \otimes \bv_K) (\bv_1 \otimes \cdots \otimes \bv_K)^\ast
          \,\Bigm|\,
          \|\bv_i\|=1,\, \bv_i \in \cH_i,\, 1 \leq i \leq K
       \Bigr\}.
\end{equation}

\begin{algorithm}[ht]
\caption{Frank-Wolfe algorithm to find $\upisep(\bSigma)$}
\label{alg:frank_wolfe}
\begin{algorithmic}[1]
    \Require $\bSigma \in \cSsa_+(\cH)$ where $\cH = \bigotimes_{i=1}^K \cH_i$;
        \ maximum iterations $T$;
        \ stepsize rule $\{ \gamma^{(t)} \}$
    \Ensure K-S approximation $\upisep(\bSigma) \in \cSsep(\cH)$
    
    \State \textbf{Initialize:} $\hbSigma^{(0)} \gets \frac{1}{\dim(\cH)}\,\bI$ \Comment{Or any other initialization in $\cSsep(\cH)$}
    
    \For{$t = 0, 1, \dots, T-1$}
        \State \textbf{(Linear minimization subproblem)}:
        \[
           \bZ^{(t)} \in \argmin_{\bZ \in \cA}\;
              \bigl\langle \bZ,\; \hbSigma^{(t)} - \ovbSigma \bigr\rangle
        \]
    
        \State \textbf{(Update)}:
        \[
           \hbSigma^{(t+1)} \gets
              \bigl(1 - \gamma^{(t)}\bigr)\,\hbSigma^{(t)} + \gamma^{(t)}\,\bZ^{(t)} 
        \]
    \EndFor
    
    \State $\upisep(\bSigma) \gets \hbSigma^{(T)}$
\end{algorithmic}
\end{algorithm}

\medskip
\paragraph{Procedure}  
As summarized in Algorithm~\ref{alg:frank_wolfe}, the FW method has three main steps.

\begin{enumerate}[topsep=2pt,itemsep=2pt,leftmargin=2em]
    \item \emph{Initialization:}
    Choose any \(\hbSigma^{(0)} \in \cSsep(\cH)\), e.g., the trace-normalized identity $\frac{1}{\dim(\cH)}\,\bI$.
    \item \emph{Linear minimization subproblem (LMS):}
    Given the current solution $\hbSigma^{(t)}$, find
    \begin{equation}\label{eqn:LMS}
        \bZ^{(t)} \in\argmin_{\bZ \,\in\,\cA}
            \bigl\langle \bZ,\;\,\hbSigma^{(t)} - \bSigma \bigr\rangle.
    \end{equation}
    Note that finding $\bZ^{(t)}$ is equivalent to finding a rank-1 tensor product $\bv_1\otimes\cdots\otimes\bv_K$ that maximizes $(\bv_1 \otimes \cdots \otimes \bv_K)^T (\bSigma - \hbSigma^{(t)}) (\bv_1 \otimes \cdots \otimes \bv_K)$. 
    \item \emph{Frank-Wolfe update:} Update the incumbent solution as
    \[
        \hbSigma^{(t+1)} 
            \gets (1-\gamma^{(t)}) \,\hbSigma^{(t)} +\gamma^{(t)} \,\bZ^{(t)},
    \]
    where the step size $\gamma^{(t)}$ can follow a fixed diminishing schedule (e.g., $\gamma^{(t)}= 2/(t+2)$) or it can be determined by other rules such as line-search or fully corrective updates \cite{jaggi2013revisiting}. 
\end{enumerate}

\medskip
\paragraph{Convergence guarantees and computational complexity}
If the LMS \eqref{eqn:LMS} is solved (approximately) well in each iteration, the FW algorithm (Algorithm \ref{alg:frank_wolfe}) converges to the global optimum $\upisep(\bSigma)$. 
A proof of the following theorem is in the supplement (\ref{sec:proof_thm.5}). 

\begin{theorem}\label{thm:convergence_FW}
    Let $\bSigma \in \cSsa_+(\cH)$ and let $\eta > 0$. 
    Suppose that Algorithm \ref{alg:frank_wolfe} can (approximately) solve the linear minimization subproblem in Line 3 by finding $\hbZ^{(t)}$ such that 
    \begin{equation}\label{eqn:approx_grad}
        \left\langle \hbZ^{(t)}, \, \hbSigma^{(t)} - \bSigma \right \rangle \leq \left \langle \bZ^{(t)}, \, \hbSigma^{(t)} - \bSigma  \right \rangle + \gamma^{(t)} \eta,
        \qquad \forall t \in \NN.
    \end{equation}
    Then
    \[
        \big\| \hbSigma^{(t)} - \upisep(\bSigma) \big\|^2 
            \leq  \frac{8 (1 +\eta)}{t+2},
        \qquad \forall t \in \NN.
    \]
\end{theorem}
If the LMS \eqref{eqn:LMS} can be solved to accuracy $\eta$, the K-S approximation $\upisep(\bSigma)$ can be computed to an arbitrary precision $\delta \geq 0$ within $O(\eta/\delta^2)$ iterations. 
However, for $K\geq2$, the LMS \eqref{eqn:LMS} reduces to a tensor spectral norm problem, which is NP-hard \cite{hillar2013most}.
In practice, heuristics such as non-convex parameterizations\footnote{When $\cH = \FF^{d_1} \otimes \FF^{d_2}$, it is customary to parameterize $Z = \bv_1\bv_1^{*} \otimes \bv_2\bv_2^{*}$ where $(\bv_1, \bv_2) \in \bbS^{d_1-1} \times \bbS^{d_2-1}$.} and alternating minimization (see \ref{sec:supplement_experiments} for details) can be used to approximately solve the LMS \eqref{eqn:LMS} in order to find a nearly optimal $\bZ^{(t)}$. 
Due to the NP-hardness result (Theorem \ref{thm:approx}), there cannot exist a polynomial-time algorithm that solves (\ref{eqn:best_separable}) unconditionally. 
Despite these theoretical barriers, FW can still perform effectively in moderate-dimensional settings, especially if local searches can identify an approximately optimal $\hat{Z}^{(t)}$.

\subsubsection{Gradient-descent methods}
\label{sec:alg_Gradient_Descent}
As an alternative to Frank-Wolfe, one can parameterize a candidate separable operator directly, and optimize its parameters directly via gradient-based methods. 
Concretely, fix a model size $\Napprox \in \NN$ and write
\[
  \hbSigma(\theta)
  = \sum_{i=1}^{\Napprox}
      w_i(\theta)\;\bigl(\bv_{i,1}(\theta)\,\otimes\cdots\otimes\,\bv_{i,K}(\theta)\bigr)
      \bigl(\bv_{i,1}(\theta)\,\otimes\cdots\otimes\,\bv_{i,K}(\theta)\bigr)^\ast,
\]
where $\theta$ collects all trainable parameters, from which we obtain (1) scalar weights $\{w_i(\theta) \}$ such that $w_i(\theta)\ge 0$ with $\sum_i w_i(\theta)=1$, and (2) unit vectors $\{\bv_{i,j (\theta) }: (i,j) \}$ with $\bv_{i,j}(\theta)\in\cH_j$. 
By construction, $\hbSigma(\theta) \in \cSsep(\cH)$ for all $\theta$, as it is a sum of $N_{\mathrm{approx}}$ rank-1 expansions (cf. Definition \ref{defn:separable.2}). 
Then we optimize $\theta$ by solving
\[
    \min_{\theta}\;
        L(\theta)
        \qquad\text{where}\qquad L(\theta) = \bigl\|\hbSigma(\theta) - \bSigma\bigr\|^2
\] 
using standard gradient-descent-type optimizers like SGD or Adam \cite{kingma2015adam}.

\begin{algorithm}[t!]
\caption{Shallow Gradient Descent (GD) for K-S Approximation}
\label{alg:gradient_descent}
\begin{algorithmic}[1]
    \Require
      $\bSigma \in \cSsa_+(\cH)$ where $\cH = \bigotimes_{i=1}^K \cH_i$;\
      model size $N_{\mathrm{approx}}$;\
      max iterations $T$;\
      stepsize rule $\{ \eta_t \}$.
    \Ensure
      K-S approximation $\upisep(\bSigma) \in \cSsep(\cH)$.
    \State \textbf{Initialize:} 
      randomly choose parameters $\theta^{(0)} \gets \bigl(\{ \tw_i\},\;\{\tbv_{i,j}\}\bigr) \in \bigl( \RR \times (\RR^d)^K \bigr)^{N_{\mathrm{approx}}}$
    \For{$t = 0, 1, \dots, T-1$}
        \State \textbf{(Forward pass)}: 
        Construct the current approximation 
        \[
            \hbSigma(\theta^{(t)})
            \;=\;
            \sum_{i=1}^{N_{\mathrm{approx}}} 
            w_i
            \Bigl(\bv_{i,1}\,\otimes\cdots\otimes\,\bv_{i,K}\Bigr)
            \Bigl(\bv_{i,1}\,\otimes\cdots\otimes\,\bv_{i,K}\Bigr)^{\ast}.
        \]
        \qquad where
        \[
            w_i = \softmax_i( \tw_1, \dots, \tw_{\Napprox}) \coloneqq \frac{e^{\tw_i}}{\sum_{i=1}^{\Napprox} e^{\tw_i}} \quad \text{and} \quad \bv_{i,j} = \frac{ \tbv_{i,j}}{ \|\tbv_{i,j}\|}
        \]
        \State \textbf{(Backward pass)}: 
        Obtain $\nabla_\theta L(\theta^{(t)})$ where $L(\theta) \;=\; \bigl\|\hbSigma(\theta)\;-\;\bSigma\bigr\|^2$ 
        \State \textbf{(Update)}: 
        $\theta^{(t+1)} \gets \theta^{(t)} \;-\;\eta^{(t)}\,\nabla_\theta L\big(\theta^{(t)} \big)$
            \Comment{Or variants of GD, such as Adam}
    \EndFor
    \State $\upisep(\bSigma) \gets \hbSigma(\theta^{(T)})$
\end{algorithmic}
\end{algorithm}

\medskip
\paragraph{Procedure}
The GD method consists of three main steps, as summarized in Algorithm~\ref{alg:gradient_descent}.
\begin{enumerate}[topsep=2pt,itemsep=2pt,leftmargin=2em]
    \item \emph{Initialization:}
    Randomly choose an initial parameter $\theta^{(0)}$.
    \item \emph{Construction (forward pass):}
    Construct $\hbSigma(\theta^{(t)})$ by converting raw parameters $\theta^{(t)}$ to valid $\{ w_i(\theta), \bv_{i,k}(\theta)$ via softmax and normalization. 
    \item \emph{Gradient-descent update (backward pass):}
    Update with a learning rate schedule $\eta{(t)}$:
    \[
        \theta^{(t+1)} \gets 
            \theta^{(t)} -\eta^{(t)}\,\nabla_\theta L(\theta^{(t)}).
    \]
\end{enumerate}

Although the objective function is non-convex in $\theta$, we often observe good performance in experiments (see \Cref{sec:exp_random_separable} and \Cref{sec:quantum_states}), particularly for moderate $K$ and $d$.

\medskip
\paragraph{GD with deep neural network parameterization} 
Instead of directly treating $w_i$ and $\bv_{i,k}$ (specifically, the raw parameters $\{ \tw_i\},\;\{\tbv_{i,j}\}$) as decision variables, we can reparameterize them using a neural network that takes $i \in [\Napprox]$ as input and outputs $\{ w_i(\theta), \bv_{i,k}(\theta) \}$. 
We call this algorithm as ``Deep Parameterized GD'' (DP-GD); see \ref{sec:DP-GD_details} for more details. 

\medskip
\paragraph{Comparison between FW and GD}
FW iterates directly in the space of separable covariances, expanding a convex combination of rank-1 atoms in $\cA$; see \eqref{eqn:setA}. 
In contrast, GD iterates in the parameter space $\theta$, which can result in different optimization landscape and require additional hyperparameters (e.g., $\Napprox$ and $\eta$). 
Either of the methods can be practically effective, each with advantages in different dimensional or structural regimes.

\subsection{Experiment 1: Random separable states}
\label{sec:exp_random_separable}

We first consider a scenario where $\bSigma$ is K-S, generated by mixing $\Ntrue$ rank-1 Kronecker products.  
Hence, the best K-S approximation is $\bSigma$ itself, and the ideal approximation error is zero.  

\medskip
\paragraph{Generative model}
Let $\cH=\bigotimes_{j=1}^K \cH_j$ where $\cH_j = \RR^d$ for all $j \in [K]$. 
\begin{enumerate}
    \item 
    Draw random weights $\{w_a\}_{a=1}^{\Ntrue}$ such that $w_i>0, ~\forall i$ and $\sum_{i=1}^{\Ntrue} w_i = 1$.  Specifically, we draw $w_i$ from $\Unif[0,1]$ i.i.d. and normalize them by their sum. 
    \item 
    Draw random unit vectors $\bv_{i,j} \in \cH_j$ with $\|\bv_a^{(k)} \| = 1, ~\forall a, k$. We draw $\bv_a^{(k)}$ from standard Gaussian distribution and normalize by its norm.
    \item 
    We generate $\bSigma\in\cSsa_+(\cH)$ by summing up the rank-1 Kronecker products:
    \begin{equation}\label{eqn:random_separable_gen}
        \bSigma
        = \sum_{a=1}^{\Ntrue} w_a   
            \left( \bigotimes_{k=1}^K \bv_a^{(k)} \right)\left( \bigotimes_{k=1}^K \bv_a^{(k)} \right)^T
        = \sum_{a=1}^{\Ntrue} w_a \bigl( \bv_a^{(1)} {\bv_a^{(1)}}^\top \bigr)\,\otimes\,\cdots \, \otimes \,\bigl(\bv_a^{(K)} {\bv_a^{(K)}}^\top \bigr).
    \end{equation}
\end{enumerate}

\medskip
\paragraph{Results and observations}
We evaluate Frank–Wolfe (FW), gradient descent (GD), and DP-GD on 100 random separable covariance instances per \eqref{eqn:random_separable_gen} in bipartite setting, across various $d$ and $\Ntrue$.  
Each method is run three times per instance, and Figure~\ref{fig:random_separable} shows boxplots of the best errors. 
FW is run for 3000 iterations with a fixed diminishing step size. 
GD is run with various $\Napprox$ using Adam for 10 epochs, 300 steps each, with initial learning rate $\eta = 1/10$ decayed by $1/\sqrt{2}$ per epoch. 
DP-GD (depth 2, width 100) is run similarly with initial learning rate $\eta = 1/40$, compensated by $4\times$ steps per epoch, to stabilize training. 
All three methods accurately recover the separable covariances to error below $10^{-2}$ in most configurations, except when $\Ntrue > \Napprox$ for GD/DP-GD.
We report on other algorithmic design choices (e.g., FW step-size rules, GD/DP-GD learning rates) in the supplement (\ref{sec:ablation_studies}).

\begin{figure}[t]
    \centering
    \includegraphics[width=\linewidth]{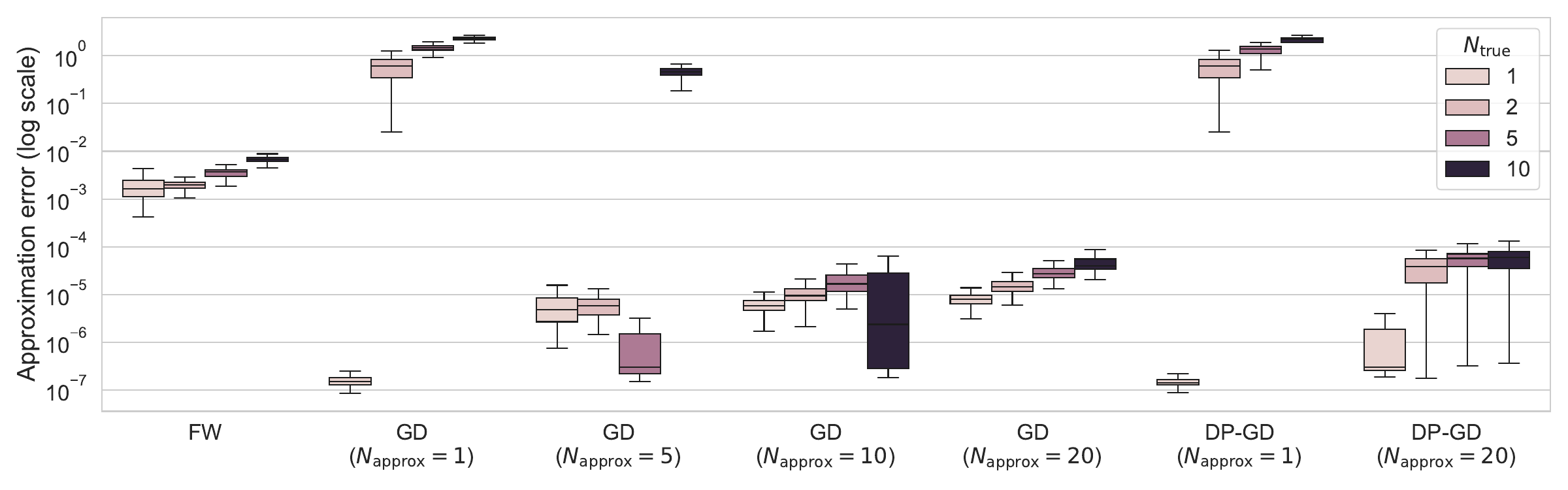}
    \caption{
        Approximation error (boxplots) for FW, GD and DP-GD on 100 random separable covariance instances. 
        FW is run for 3000 iterations with a fixed diminishing step size. 
        GD uses $\Napprox \in \{1, 5, 10, 20\}$, trained via Adam optimizer for 10 epochs of 300 steps each, with initial learning rate $\eta = 1/10$ decayed by $1/\sqrt{2}$ each epoch. 
        DP-GD (depth 2, width 100) runs at $\Napprox \in \{1, 20\}$ using Adam for 10 epochs of 1200 steps ($4\times$ more steps per epoch compared to GD) each, with initial $\eta = 1/40$ decayed by $1/\sqrt{2}$ per epoch to stabilize training. 
        All methods perform well, except when $\Ntrue > \Napprox$ in the GD/DP-GD parameterization.
    }
    \label{fig:random_separable}
\end{figure}

\subsection{Experiment 2: Quantum states with known separability thresholds}
\label{sec:quantum_states}

Next, we test our K-S approximation algorithms on two canonical families of bipartite quantum states 
whose exact separability thresholds and best separable approximations (BSA) have been established or conjectured in the quantum physics literature \cite{werner1989quantum,horodecki1999general}. 
These states serve as a natural benchmark, enabling a direct comparison between our approximations and known BSA results.

\medskip
\paragraph{Isotropic states}
Let $\cH = \cH_1 \otimes \cH_2$ with $\cH_1 = \cH_2 = \CC^d$. 
We call $d$ the local dimension. 
The \emph{isotropic state} $\bSigma^{\iso}_d(\alpha)$ for $\alpha \in [0,1]$ is defined as
\[
    \bSigma^{\iso}_d(\alpha) 
        =(1-\alpha)\,\frac{1}{d^2}\,\bI_{d^2} + \alpha \cdot  \Phi_+ \Phi_+^{\top},
        \quad\text{where}\quad 
        \Phi_+ \coloneqq \frac{1}{\sqrt{d}} \sum_{i=1}^d \be_{i}^{(1)}\otimes\be_{i}^{(2)}.
\]
Here, $\Phi_+$ is known as the maximally entangled vector in $\CC^d \otimes \CC^d$. 
It is known \cite{lewenstein1998separability,horodecki1999general} that
\begin{itemize}[itemsep=2pt,topsep=2pt,leftmargin=1.5em]
    \item
    $\bSigma^{\iso}_d(\alpha)$ is separable if and only if \(\alpha \,\le\, 1/(d+1)\).  
    \item
    For $\alpha > \frac{1}{d+1}$, the best separable approximation of $\bSigma^{\iso}_d(\alpha)$ is simply $\bSigma^{\iso}_d(\frac{1}{d+1})$. 
    Therefore, the distance between the true state and its BSA is
    \[
        \left\| \upisep\big( \bSigma^{\iso}_d(\alpha) \big) - \bSigma^{\iso}_d(\alpha) \right\|_{\mathrm{HS}}
            = \frac{\sqrt{d^2 - 1}}{d} \left( \alpha - \frac{1}{d+1} \right).
    \] 
\end{itemize}

\medskip
\paragraph{Werner states}
Again, let $\cH = \CC^d \otimes \CC^d$. 
The \emph{Werner state} $\bSigma^{\Werner}_d(\alpha)$ for $\alpha \in [0,1]$ is
\[
    \bSigma^{\Werner}_d(\alpha)
    =   (1-\alpha)\, \frac{2}{d(d+1)}\bP^+ + \alpha\, \frac{2}{d(d-1)}\bP^-,
\]
where 
\[
    \bP^+ = \tfrac{1}{2}\bigl(\bI_{d^2} + \bF\bigr),
    \quad
    \bP^- = \tfrac{1}{2}\bigl(\bI_{d^2} - \bF\bigr),
    \quad
    \bF = \sum_{i,j=1}^d  \bigl(\be_i^{(1)}\,{\be_j^{(1)}}^\top\bigr)
        \,\otimes\, \bigl(\be_j^{(2)}\,{\be_i^{(2)}}^\top\bigr).
\]
Here, $\bP^+$ and $\bP^-$ are the orthogonal projectors onto the symmetric and anti-symmetric subspaces of $\CC^d \otimes \CC^d$, respectively, where $\bI_{d^2}$ is the $d^2 \times d^2$ identity operator, and $\bF$ is called the flip (swap) operator.
From the quantum physics literature, e.g., \cite{lewenstein1998separability,horodecki1999general}, it is known that:
\begin{itemize}[itemsep=2pt,topsep=2pt,leftmargin=1.5em]
    \item
      $\bSigma^{\Werner}_d(\alpha)$ is separable if and only if $\alpha \,\le\, \tfrac{1}{2}$.  
    \item
    For $\alpha > \tfrac{1}{2}$, the BSA is $\bSigma^{\Werner}_d(\frac{1}{2})$, and thus,
    \[
        \left\| \upisep\bigl(\bSigma^{\Werner}_d(\alpha)\bigr) - \bSigma^{\Werner}_d(\alpha) \right\|_{\mathrm{HS}}
        = \frac{2}{\sqrt{d^2}-1} \left( \alpha - \frac{1}{2} \right).
    \]
    This is formally established for $d=2, 3$ \cite{lewenstein1998separability}, and is believed to hold for $d > 3$, but to our knowledge, there is no published reference that explicitly states these results. 
\end{itemize}



\begin{figure}[t]
    \centering
    \includegraphics[width=\linewidth]{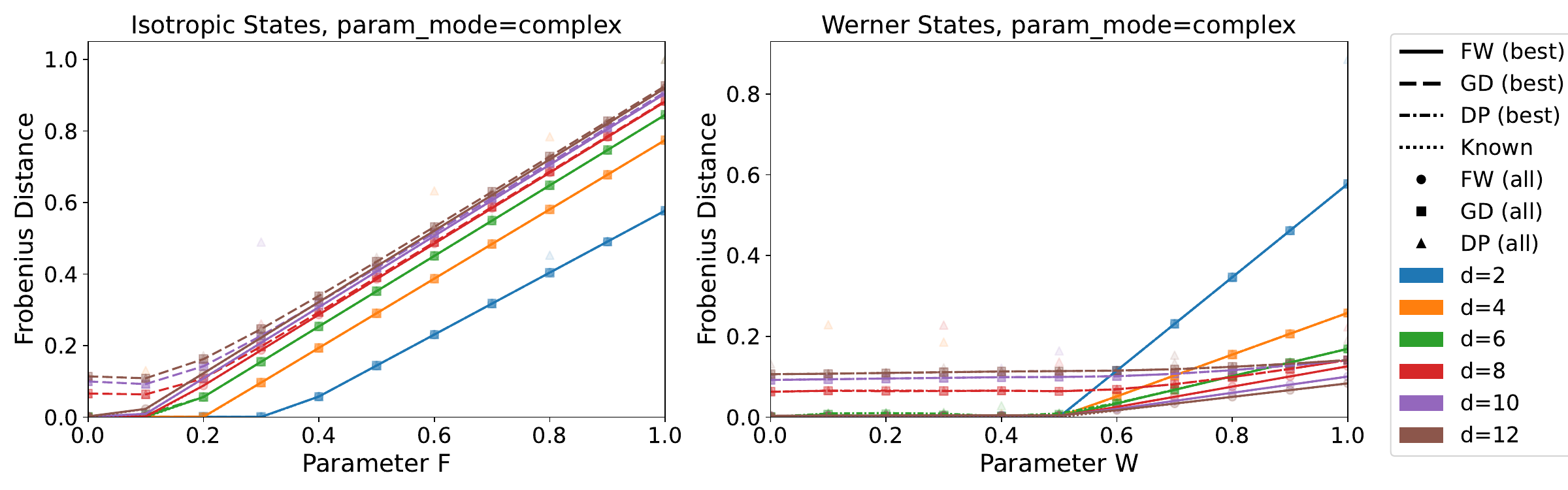}
    \caption{
        Approximation of isotropic and Werner states at various local dimensions $d$. 
        FW, GD, and DP-GD all recover the correct best K-S approximation for moderate $d$ ($\leq 6$), but GD and DP-GD sometimes struggle at larger $d$, possibly due to the limited model size $\Napprox=50$.
    }
    \label{fig:quantum_separability}
\end{figure}

\medskip
\paragraph{Results and observations}
We run FW, GD, and DP-GD on both $\bSigma^{\iso}_d(\alpha)$ and $\bSigma^{\Werner}_d(\alpha)$ for $\alpha \in [0,1]$ at local dimensions $d=2,4, 6, 8, 10, 12$. 
FW is run with a diminishing step-size rule for up to 1500 iterations, and GD/DP-GD are run with $\Napprox=50$ over 10 epochs (150 steps each) starting at the learning rate $\eta = 0.1$ and decaying by a factor of $1/\sqrt{2}$ per epoch. 
All three methods---with \emph{complex parameterization} of $\bv_i^{(1)}, \bv_i^{(2)}$---accurately recover the known separability thresholds ($\alpha = 1/(d+1)$ or $\alpha=1/2$) and the associated distances from the true state to its BSA for $d \leq 6$. 
For higher dimensions ($d=8,10, 12$), GD/DP-GD may fail to match the exact thresholds, possibly due to the limited approximation budget $\Napprox$, whereas FW continues to converge reliably (Figure~\ref{fig:quantum_separability}). 
This suggests that a larger $\Napprox$ may be necessary in these larger-dimensional settings. 
Additional results with increased budgets $\Napprox$ and further discussions are provided in the supplement (\ref{sec:additional_quantum_exp}).

We also observe that real parameterizations---i.e., enforcing $\{\bv_i^{(1)}, \bv_i^{(2)}\}$ to be real-valued---often yield worse approximations than complex ones, as they cannot capture crucial phase relationships. 
This difference is especially noticeable for isotropic and Werner states (see \ref{sec:real_vs_complex}) and reflects the fact that $\cBsep(\RR^d\otimes\RR^d)$ is a strict subset of $\cBsep(\CC^d\otimes\CC^d) \cap \cBsa(\RR^d\otimes\RR^d)$. 
In practice, the choice between real and complex K-S models may depend on the domain: complex representations seem essential in quantum-mechanical settings, while in data-science contexts, one may trade slight performance loss for more interpretable real-valued outputs. 
We explore these phenomena, along with additional plots, in the supplement (\ref{sec:real_vs_complex}).

\subsection{Connection to prior work}
Our FW framework generalizes the ``DA algorithm'' of Dahl \emph{et al.}~\cite{dahl2007tensor}, while our GD/DP-GD approach is closely related to the ``neural network method'' of Girardin \emph{et al.}~\cite{girardin2022building}. 
Although both works report promising performance, they omit discussion of convergence guarantees or real vs.\ complex parameterization. 
In contrast, we provide a broader, more systematic treatment of FW and GD algorithms---along with extensive ablation studies and experiments---to assess their performance and limitations.
For instance, while \cite{girardin2022building} attributes large gains to deep network parameterization over ``naive'' GD, our results suggest that properly tuned shallow GD can also perform competitively. 

%% file: contents/06_discussions.tex
\section{Conclusion}\label{sec:conclusion}
In this paper, we investigated the representation and approximation of multiway covariances by Kronecker-separable (K-S) models, drawing on insights from quantum information, optimization, and complexity theory. 
K-S models represent the covariance or inverse covariance matrices as sums of Kronecker products of positive semidefinite (PSD) matrices, forming the basis for approaches such as the matrix normal model, the Kronecker PCA (K-PCA), and the Teralasso. 
We showed that a large fraction of PSD matrices is not K-S, and that computing the best K-S approximation is NP-hard. 
In particular, our results imply that, under the worst-case analysis, it is impossible to extend K-PCA to impose that all its factors be PSD without sacrificing the uniform approximation property of the original unconstrained K-PCA. 
Notwithstanding these negative results, our numerical experiments indicated that iterative methods such as Frank–Wolfe and gradient-descent can converge to high-quality K-S approximations, suggesting that worst-case hardness may be overly conservative in practice.

There are several promising directions for extension of our analysis. 
It would be interesting to identify structured classes of covariances for which the linear subproblem (LMS \eqref{eqn:LMS} in Frank–Wolfe) is tractable, clarifying conditions under which efficient and accurate K-S approximations are possible. 
Another direction is to explore conditions or regularizations---such as adding a small multiple of the identity---to promote separability in otherwise non-separable covariances. 
Finally, we hypothesize that there may exist a universal data model that naturally leads to K-S covariances, akin to how certain structured matrices tend toward sparsity or low rank \cite{udell2019big}. 

%% file: contents/SuppA_additional_KS_discussion.tex
\section{Additional discussion on Kronecker-separability}
\subsection{Deferred proof of Proposition \ref{prop:separable_equivalence}}
Here we provide a proof of Proposition \ref{prop:separable_equivalence} in Section \ref{sec:separable_cov}, which establishes the equivalence between the two notions of separability in Definition \ref{defn:separable} and Definition \ref{defn:separable.2}.
\begin{proof}[Proof of Proposition \ref{prop:separable_equivalence}]
    Suppose that $\bSigma \in \cBsa_+(\cH)$ is Kronecker-separable. 
    By definition (Definition \ref{defn:separable}), $\bSigma$ admits a Kronecker decomposition of the form \eqref{eqn:separable}. 
    By considering the spectral decompositions of the factor covariances, we can express $\bSigma$ as follows:
    \begin{align*}
        \bSigma 
            &= \sum_{a=1}^r \bSigma^{(1)}_a \otimes \cdots \otimes \bSigma^{(K)}_a\\
            &= \sum_{a=1}^r \left( \sum_{i_1=1}^{d_1} \bv^{(1)}_{a,i_1} \otimes {\bv^{(1)}_{a,i_1}}^* \right) \otimes \cdots \otimes \left( \sum_{i_K=1}^{d_K} \bv^{(K)}_{a,i_1} \otimes {\bv^{(K)}_{a,i_1}}^* \right)\\
            &= \sum_{a=1}^r \left[ \sum_{i=1}^{d_1} \cdots \sum_{i=k}^{d_k} \left( \bigotimes_{k=1}^K \bv^{(k)}_{a, i_k} \right) \otimes \left( \bigotimes_{k=1}^K \bv^{(k)}_{a, i_k} \right)^* \right],
    \end{align*}
    where $d_k = \dim \cH_k$ and $\bv^{(k)}_{a,i_k} \in \cH_k$ for all $k \in [K]$, $a \in [r]$, and $i_k \in [d_k]$.
    Therefore, $\bSigma$ is separable. 
    
    Conversely, suppose that $\bSigma \in \cBsa_+(\cH)$ is separable, as described in Definition \ref{defn:separable.2}. 
    In this case, there exists a sequence of vector tuples $\big( ( \bv^{(1)}_a, \dots, \bv^{(K)}_a) \in \cH_1 \times \dots \times \cH_K : a \in [N] \big)$ for some $N \in \NN$ such that 
    \begin{align*}
        \bSigma
            &= \sum_{a=1}^N \big( \bv^{(1)}_a \otimes \dots \otimes \bv^{(K)}_a \big) \otimes \big( \bv^{(1)}_a \otimes \dots \otimes \bv^{(K)}_a \big)^* \\
            &= \sum_{a=1}^N  \big( \bv^{(1)}_a \otimes {\bv^{(1)}_a}^* \big) \otimes \dots \otimes \big( \bv^{(K)}_a \otimes {\bv^{(K)}_a}^* \big).
    \end{align*}
    For instance, we may let $N = \rank_s(\bSigma)$. 
    This indicates $\bSigma$ is Kronecker-separable with $\ranks(\bSigma) \leq \rank_s(\bSigma)$ because $\bv^{(k)}_a \otimes {\bv^{(k)}_a}^* \in \cBsa_+(\cH_k)$ for all $(k,a) \in [K] \times [N]$. 
    
    Since $\cBsa(\cH)$ is a $\dimsa(\cH)$-dimensional real vector space, there exists a sequence of tuples with $N \leq \dimsa(\cH)$ due to the Carath\'eodory's theorem. Thus, $\rank_s(\bSigma) \leq \dimsa(\cH)$.
\end{proof}

\subsection{More on Counterexample \ref{example:entangled}}\label{sec:addtl_bell_state}
In the setting of Counterexample \ref{example:entangled}, the partial transpose operation, with respect to the second factor, is the linear endomorphism $\Phi$ on the vector space of $4 \times 4$ matrices, $\cB(\cH_1 \otimes \cH_2) \cong \cB(\cH_1) \otimes \cB(\cH_2)$ such that $\Phi: \bA_1 \otimes \bA_2 \mapsto \bA_1 \otimes \bA_2^{\top}$. 
Recall that $\cB(\cH)$ denotes the space of all bounded linear operators on $\cH$, and thus, the factor operators $\bA_1, \bA_2$ are neither necessarily PSD nor even self-adjoint. 
By linearity, $\Phi$ extends to the entire space $\cB(\cH_1 \otimes \cH_2)$ as follows:
\begin{equation}
    \bA = \sum_{a=1}^{N} \bA_a^{(1)} \otimes \bA_a^{(2)}
    \quad\implies\quad
    \Phi(\bA) = \sum_{a=1}^{N} \Phi( \bA_a^{(1)} \otimes \bA_a^{(2)} )
         = \sum_{a=1}^{N} \bA_a^{(1)} \otimes {\bA_a^{(2)}}^{\top},
\end{equation}
where the sum at the far right is the partial transpose of $\bA$.

Specializing to the Bell state covariance $\bSigma_{\bell}$, we recall its definition:
\begin{align*}
    \bSigma_{\bell}
        &= \big( \be^{(1)}_1 \otimes \be^{(2)}_1 + \be^{(1)}_2 \otimes \be^{(2)}_2 \big) \otimes \big( \be^{(1)}_1 \otimes \be^{(2)}_1 + \be^{(1)}_2 \otimes \be^{(2)}_2 \big)^*\\
        &= \sum_{i,j=1}^2 \big( \be^{(1)}_i \otimes{ \be^{(1)}_j }^* \big) \otimes \big( \be^{(2)}_i \otimes{ \be^{(2)}_j }^* \big)
            \\
        &= 
        \begin{bmatrix}1&0\\0&0\end{bmatrix} \otimes \begin{bmatrix}1&0\\0&0\end{bmatrix} 
        +
        \begin{bmatrix}0&1\\0&0\end{bmatrix}\otimes \begin{bmatrix}0&1\\0&0\end{bmatrix} 
        +
        \begin{bmatrix}0&0\\1&0\end{bmatrix}\otimes \begin{bmatrix}0&0\\1&0\end{bmatrix}
        +
        \begin{bmatrix}0&0\\0&1\end{bmatrix}\otimes \begin{bmatrix}0&0\\0&1\end{bmatrix}.
\end{align*}
Therefore, the partial transpose is
\begin{align}
    \Phi(\bSigma_{\bell})
        &= 
        \begin{bmatrix}1&0\\0&0\end{bmatrix} \otimes \begin{bmatrix}1&0\\0&0\end{bmatrix}^{\top}
        +
        \begin{bmatrix}0&1\\0&0\end{bmatrix}\otimes \begin{bmatrix}0&1\\0&0\end{bmatrix}^{\top}
        +
        \begin{bmatrix}0&0\\1&0\end{bmatrix}\otimes \begin{bmatrix}0&0\\1&0\end{bmatrix}^{\top}
        +
        \begin{bmatrix}0&0\\0&1\end{bmatrix}\otimes \begin{bmatrix}0&0\\0&1\end{bmatrix}^{\top}
            \nonumber\\
        &=
        \begin{bmatrix} 1&0&0&0\\0&0&1&0\\0&1&0&0\\0&0&0&1\end{bmatrix}.    \label{eqn:bell_pt}
\end{align}
Since $\Phi( \bSigma_{\bell} )$ is not positive semidefinite, $\bSigma_{\bell}$ is not Kronecker-separable, meaning that $\bSigma_{\bell}$ cannot be expressed as a sum of Kronecker products of $2 \times 2$ PSD matrices.

\subsection{Kronecker decomposition vs rank decomposition for the Bell state}\label{sec:KF_vs_CPD}
The tensor rank is the smallest integer $R$ for which the tensor can be expressed as the sum of $R$ rank-one tensors \cite{hitchcock1927expression, kruskal1977three}. 
This minimal decomposition is known as the rank decomposition, also called the canonical polyadic decomposition (CPD). 
The definition of tensor rank is analogous to the definition of matrix rank, and they coincide for the case $K=2$. 

By definition of the tensor product and the tensor rank, every $\bSigma \in \cBsa(\cH)$ possesses a finite tensor rank ($R$) whenever $\cH = \bigotimes_{k=1}^K \cH_k$ has finite dimension. 
For illustrative purposes, consider the $\bSigma_{\bell}$ instance described in \eqref{eqn:Sigma_entangled}, which accommodates a rank decomposition as:
\begin{align}
    \bSigma_{\bell} 
        &= \sum_{i,j=1}^2 \big( \be^{(1)}_i \otimes \be^{(2)}_i \big) \otimes \big( \be^{(1)}_j \otimes \be^{(2)}_j \big)^*   
            \label{eqn:cpd_example.0}\\
        &= \sum_{i,j=1}^2 \big( \be^{(1)}_i \otimes {\be^{(1)}_j}^* \big) \otimes \big( \be^{(2)}_i \otimes {\be^{(2)}_j}^* \big)
            \nonumber\\
        &= \begin{bmatrix} 1&0\\0&0 \end{bmatrix}^{\otimes 2}
        + \begin{bmatrix} 0&1\\0&0 \end{bmatrix}^{\otimes 2}
        + \begin{bmatrix} 0&0\\1&0 \end{bmatrix}^{\otimes 2}
        + \begin{bmatrix} 0&0\\0&1 \end{bmatrix}^{\otimes 2}.    \label{eqn:cpd_example}
\end{align}
Here, the notation $\bA^{\otimes 2}$ denotes the self-Kronecker product $\bA \otimes \bA$. 
Note that \eqref{eqn:cpd_example} is also a Singular Value Decomposition (SVD) of $\bSigma_{\bell}$, when $\bSigma_{\bell}$ is considered as an operator from $\cB(\cH_1)$ to $\cB(\cH_2)$; this follows by (i) reconfiguring the matrix $\bSigma_{\bell}$ in \eqref{eqn:Sigma_entangled} into a $2 \times 2 \times 2 \times 2$ tensor of order 4; and (ii) interchanging the second and third indices.

However, the series expansion \eqref{eqn:cpd_example} is not a Kronecker decomposition of $\bSigma_{\bell}$ due to the presence of ``cross products" in \eqref{eqn:cpd_example.0}. 
In particular, neither $\be^{(k)}_1 \otimes {\be^{(k)}_2}^*$ nor $\be^{(k)}_2 \otimes {\be^{(k)}_1}^*$ (the second and third terms in \eqref{eqn:cpd_example}) are PSD (nor even symmetric).

Importantly, we note that $\bSigma_{\bell}$ cannot be represented by a sum of self-outer products of tensor products, and neither can $\bSigma_{\bell} + \lambda \bI$ for all $\lambda \in (0,1)$.
This limitation underscores a fundamental constraint within the expressive capacity of Kronecker-separable covariance models. 
This issue cannot be resolved by considering alternative potential Kronecker decompositions of $\bSigma_{\bell}$ with different factors; refer to the paragraphs following Counterexample \ref{example:entangled}. 

\subsection{Additional discussions on certifying Kronecker-separability}\label{sec:more_verifiable}
Here, we continue on Section \ref{sec:verifiable} by framing the separability decision as a hypothesis test and presenting further criteria from quantum information, including additional sufficient conditions (beyond proximity to identity) and a hierarchy of necessary conditions.

\paragraph{A hypothesis-testing viewpoint} 
Given a covariance $\bSigma \in \cBsa_+(\cH)$, we consider two complementary hypotheses:
\begin{equation}\label{eqn:binary_hypotheses}
    (H_0):~\bSigma \text{ is Kronecker-separable}
    \qquad\text{vs.}\qquad
    (H_1):~\bSigma \text{ is not Kronecker-separable.}
\end{equation}
Ideally, we want to accept $H_0$ if $\bSigma$ is indeed Kronecker-separable, and $H_1$ otherwise. 
However, in practice, we may commit (i) Type I error: mistakenly reject $H_0$ when $\bSigma$ is separable, or (ii) Type II error: wrongly accept $H_0$ when $\bSigma$ is actually non-separable. 
A \emph{one-sided test} is a procedure that avoids one type of error entirely (e.g.\ a certified acceptance with zero Type~II error), but may allow the other. 
Many existing criteria, such as the PPT condition, fall into this category: failing PPT definitely implies $H_1$, but satisfying PPT does not guarantee $H_0$.

In the quantum physics literature, there exist numerous well-established one-sided tests. 
Below we provide a few of those, augmenting those in the main text (Section \ref{sec:verifiable}).
For more details, see the comprehensive surveys on quantum entanglement \cite{horodecki2009quantum, guhne2009entanglement} and the references therein.

\paragraph{An additional sufficient condition} 
Beyond the Gurvits–Barnum proximity criterion (Section \ref{sec:verifiable}), one can also use the smallest eigenvalue condition from \cite{zyczkowski1998volume,vidal1999robustness}, namely
\begin{equation}\label{eqn:min_eigval}
    \lambda_{\min} \left( \frac{1}{\tr \bSigma} \bSigma \right) \geq \frac{1}{ d_1 d_2 + 2 }
        ~~\implies~~
        \bSigma \in \cBsep\big( \FF^{d_1} \otimes \FF^{d_2} \big).
\end{equation}
Since $\bSigma$ with at least one zero eigenvalue are on the boundary of the cone $\cBsa_+\big( \FF^{d_1} \otimes \FF^{d_2} \big)$, this implies that $\bSigma$ deep in the interior of $\cBsa_+\big( \FF^{d_1} \otimes \FF^{d_2} \big)$ are Kronecker-separable. 
Also, note that the minimum eigenvalue condition \eqref{eqn:min_eigval} implies $\lambda_{\min} \left( \frac{1}{\tr \bSigma} \bSigma \right) \leq \frac{3}{d_1d_2 + 2}$, which in turn leads to $\big\| \frac{1}{\tr \bSigma} \bSigma - \frac{1}{d_1 d_2} \bI_{d_1 d_2} \big\|^2 \leq \frac{ 4 (d_1 d_2 - 1 )}{d_1 d_2 ( d_1 d_2 + 2 )^2}$.

\paragraph{A hierarchy of necessary conditions based on $k$-extension and semidefinite program (SDP)} 
Consider $\cH = \cH_1 \otimes \cH_2$; the partial trace over $\cH_2$ is a linear map $\Tr_{\cH_2}: \cB( \cH_1 ) \otimes \cB( \cH_2 ) \to \cB(\cH_1)$ such that $\Tr_{\cH_2}(\bA \otimes \bB) = \Tr(\bB) \cdot \bA$. 
For $k \in \NN$ where $k \geq 2$, a covariance $\bSigma \in \cBsa_+(\cH_1 \otimes \cH_2)$ is said to be \emph{$k$-extendible} (with respect to $\cH_2$) if there exists $\widetilde{\bSigma}_k \in \cBsa_+ \big(\cH_1 \otimes \cH_2^{\otimes k} \big)$ such that $\Tr_{\otimes_{j=1: j \neq i}^{k} \cH_2^{(j)}}( \widetilde{\bSigma}_k ) = \bSigma$ for all $i \in [k]$. 
In \cite{doherty2002distinguishing}, it was shown that searching for a $k$-extension $\widetilde{\bSigma}$ of $\bSigma$ can be formulated as a semidefinite program (SDP) for any fixed $k \in \NN$. 
In \cite{caves2002unknown}, it was proven that a quantum density matrix in a complex Hilbert space has a de Finetti-type representation if and only if it is exchangeable. 
Based on these results, it was established in \cite[Theorem 1]{doherty2004complete} that a quantum state is separable if and only if it is $k$-extendible for all $k \geq 2$. 
This result yields a complete hierarchy of necessary criteria for a quantum state to be separable, each of which can be checked by solving a SDP. 
While this sequence of necessary conditions ultimately becomes necessary and sufficient in the limit $k \to \infty$, the computational cost of solving associated SDPs increases significantly as $k$ increases.

By specializing the quantum density operator to the covariance of a finite-dimensional complex-valued random vector, it follows that this is equivalent to: 
$\bSigma \in \cBsa_+(\cH_1 \otimes \cH_2)$ is Kronecker-separable if and only if $\bSigma$ is $k$-extendible for all $k \geq 2$.

\paragraph{Summary}
Taken together, we have a variety of one-sided tests, from simple ones (like the smallest-eigenvalue condition) to more sophisticated SDP-based ones (the $k$‐extension hierarchy).  
Failing any necessary test definitively implies non-separability, while passing a strong sufficient test (e.g.\ large $\lambda_{\min}$) guarantees separability.  
Between these extremes lie inconclusive cases where no definitive criterion is known yet.

%% file: contents/SuppB_proof_abundance.tex

\section{Deferred proof of Theorem \ref{thm:abundance}}\label{sec:proof_thm.1}
Our proof of Theorem \ref{thm:abundance} uses results from asymptotic convex geometry. 
We present these results as lemmas below, accompanied by concise overviews of their proofs while omitting overly technical details. 
These lemmas have analogs in quantum information theory \cite{aubrun2017alice}.

\subsection{Volume radius and Gaussian mean width}\label{sec:width}
We recall two standard ``width'' measures introduced in the main text (Section \ref{sec:size_set}), which quantify the size of a set and are widely used in convex geometry. 
In the proof of Theorem~\ref{thm:abundance}, we use these measures to compare the Kronecker-separable subset to the entire PSD cone.

\paragraph{Volume radius and mean width}
As stated in Definition \ref{defn:volume_radius}, for a nonempty, bounded, Lebesgue measurable set $\cS \subset \RR^d$, the \emph{volume radius} of $\cS$ is 
\[
    \vrad(\cS) \coloneqq \left(\frac{\vol(\cS)}{\vol(\bbB^d_2)}\right)^{1/d}
\]
where $\bbB_2^d$ is the $d$-dimensional unit $\ell_2$-norm ball and $\vol(\cS)$ is the $d$-dimensional Lebesgue measure `of $\cS$.

Similarly, we recall from Definition \ref{defn:mean_width} that for a nonempty, bounded set $\cS \subset \RR^d$, the \emph{mean width} of $\cS$ is defined as
\[
    w(\cS) \coloneqq \bbE_{\bu \sim \mu} \Big[ \sup_{\bx \in \cS} \langle \bu, \bx \rangle \Big]
        = \int_{\bbS^{d-1}} \sup_{\bx \in \cS} \langle \bu, \bx \rangle ~d\mu(\bu)
\]
where $\bbS^{d-1}$ is the $d$-dimensional unit sphere and $\mu$ is the Haar probability measure on $\bbS^{d-1}$.

These width notions have the following properties:
\begin{enumerate}[topsep=6pt, itemsep=3pt, label=(\texttt{P\arabic*})]
    \item \label{prop:homog}
    For any $\lambda \in \RRp$ and any set $\cS$, $\vrad(\lambda \cdot \cS) = \lambda \cdot \vrad(S)$ and $w(\lambda \cdot \cS) = \lambda \cdot w(S)$.
    \item \label{prop:invar}
    They remain invariant under translation and rotation.
    \item \label{prop:ordered}
    If $\cS \subseteq \cS'$, then $\vrad(\cS) \leq \vrad(\cS')$ and $w(\cS) \leq w(\cS')$.
\end{enumerate}

\paragraph{Gaussian mean width}
It is often more convenient to consider the Gaussian variant of the mean width because the Gaussian width does not depend on the ambient dimension and it satisfies various well-established Gaussian comparison inequalities, such as the Gordon-Chevet inequality \cite{chevet1978series, gordon1985some}.
\begin{definition}
    For a nonempty, bounded set $\cS \subset \RR^d$, the \emph{Gaussian (mean) width} of $\cS$ is defined as
    \[
        w_G(\cS) \coloneqq \bbE_{\bg} \left[ \sup_{\bx \in \cS} \langle \bg,\,\bx\rangle \right] = \frac{1}{(2\pi)^{d/2}} \int_{\RR^d} \sup_{\bx \in \cS} \langle \bz, \, \bx \rangle ~ \exp \big( - \|\bz\|^2 / 2 \big) \, d\bz,
    \]
    where $\bg$ denotes a standard Gaussian random vector in $\RR^d$.
\end{definition}
It is easy to verify that $w_G(\cS) = \kappa_d \cdot w(\cS)$, where $\kappa_d \coloneqq \bbE_{\bg} \|\bg\|_2 = \frac{ \sqrt{2}\Gamma(\frac{d+1}{2})}{\Gamma(\frac{d}{2})}$; here, $\Gamma$ denotes the gamma function. Note that $\kappa_d$ relies solely on the dimension $d$ and is of order $\sqrt{d}$ -- to be precise, it is known that $\sqrt{d - 1/2} \leq \kappa_d \leq \sqrt{ d - \frac{d}{2d+1}}$ \cite{chu1962modified}.

\subsection{Four lemmas}
The four lemmas below are used to prove Theorem \ref{thm:abundance}. 
Specifically, Lemma \ref{lem:Urysohn} constitutes a classic result, and Lemmas \ref{lem:vrad}, \ref{lem:bipartite}, and \ref{lem:multipartite} are adapted\footnote{To be precise, these are adapted from Theorems 9.1, 9.3, and 9.11 in the reference \cite{aubrun2017alice}.} from \cite[Chapter 9]{aubrun2017alice} with some simplification and modifications in proofs to fit our context of Kronecker-separable covariance.

The classic result below connects the volume radius and the mean width of a set $\cS$.
\begin{lemma}[Urysohn's inequality \cite{urysohn1926memoire}]\label{lem:Urysohn}
    Let $\cS \subset \RR^n$ be a bounded Borel set. Then
    \[
        \vrad(\cS) \leq w(\cS).
    \]
\end{lemma}
Additionally, it is known that among closed sets,  the Urysohn's inequality holds with equality if and only if $\cS$ is a Euclidean ball.

Next, the following lemma establishes a lower bound for the volume radius (as well as the mean width) in a generic dimension.
\begin{lemma}\label{lem:vrad}
    Let $\cD^n \coloneqq \cSsa_+(\CC^n) = \cBsa_+(\CC^n) \cap \cT_1$.   
    For any $n \in \NN$, 
    \[
        \vrad(\cD^n) \geq \frac{1}{2 \sqrt{n}}.
    \]
\end{lemma}
\begin{proof}[Proof of Lemma \ref{lem:vrad}]
    For any $n \in \NN$, it holds that 
    \[
        \vol(\cD^n) = \sqrt{n} \cdot (2\pi)^{\frac{n(n-1)}{2}} \frac{\prod_{j=1}^n \Gamma(j)}{\Gamma(n^2)},
    \]
    where $\Gamma(j) = (j-1)!$ for $j \in \NN$. 
    Note that $\cD^n$ is a $(n^2-1)$-dimensional convex body contained in the unit-trace section (of codimension $1$) of the space of $n \times n$ Hermitian matrices (a vector space of dimension $n^2$ over $\RR$). 
    By utilizing the volume formula for a $d$-dimensional unit Euclidean ball, $V_d \coloneqq \frac{ \pi^{d/2} }{ \Gamma(\frac{d}{2} + 1 ) }$, and employing a routine calculation based on Stirling formula, the conclusion is established.
\end{proof}

Next, we demonstrate an upper bound for the mean width in a bipartite configuration.
\begin{lemma}\label{lem:bipartite}
    For any $d \in \NN$ such that $d \geq 2$, the following inequality holds:
    \[
        w\left(\cSsep(\CC^d \otimes \CC^d) \right) \leq \frac{5}{d^{3/2}}.
    \]
\end{lemma}
\begin{proof}[Proof of Lemma \ref{lem:bipartite}]
    Observe that $\cSsep(\CC^d \otimes \CC^d) = \cSsa_+(\CC^d) \otimes \cSsa_+(\CC^d )$. 
    By the Chevet-Gordon inequality (see e.g., \cite[Proposition 6.37]{aubrun2017alice}), $w_G(\cSsep(\CC^d \otimes \CC^d)) \leq 2 \cdot w_G(\cSsa_+(\CC^d))$. 
    Moreover, we can easily verify $w_G(\cSsa_+(\CC^d)) = \bbE_{\bG} \left[ \|\bG\|_{\mathrm{op}} \right] \leq 2 \sqrt{d }$ through a standard argument based on Slepian's inequality \cite[Chapter 7.3]{vershynin2018high}. 
    Consequently, we arrive at the following inequality (recall from \Cref{sec:width} that $w_G(\cS) = \kappa_d \cdot w(\cS)$ and $\kappa_n \geq \sqrt{n - 1/2}$):
    \begin{align*}
        w \left(\cSsep ( \CC^d \otimes \CC^d) \right) 
            &= \frac{w_G\left( \cSsep(\CC^d \otimes \CC^d) \right)}{\kappa_{d^4 - 1}} 
            \leq \frac{2 \cdot w_G\left( \cSsa_+(\CC^d ) \right)}{\kappa_{d^4 - 1}} 
            \leq \frac{ 4 \sqrt{d} }{ \sqrt{ d^4 - 3/2}}.
    \end{align*}
    To complete the proof, observe that $\sqrt{ d^4 - 3/2} \geq 0.8 d^2$ for all $d \geq 2$.
\end{proof}

\begin{lemma}\label{lem:multipartite}
    For any $K \in \NN$ such that $K \geq 2$, the following inequality holds:
    \[
        w\Big(\cSsep\big( (\CC^2)^{\otimes K} \big) \Big) \leq \sqrt{12 e} \frac{\sqrt{ K \log K }}{2^K}
    \]
    where $e$ is the Euler's number.
\end{lemma}

\begin{proof}[Proof of Lemma \ref{lem:multipartite}]
    Let $\varepsilon > 0$ be an arbitrary positive real number, and $\cN$ be an $\varepsilon$-net of $\bbS^{1}$, i.e., the unit sphere in $\CC^2$, with respect to the geodesic distance on the sphere. 
    It is known that for every $n \in \NN$ and every $\varepsilon \leq 1$, there exists such an $\varepsilon$-net for $\bbS^{n-1}$ with cardinality most $(2/\varepsilon)^n$ \cite[Lemma 5.3]{aubrun2017alice}. 
    Letting $\cS_{\mathrm{sym}} \coloneqq \conv \left( \cS \cup (-\cS) \right)$, we observe that $\cSsep\big( (\CC^2)^{\otimes K} \big)_{\mathrm{sym}} = \big( \cSsa_+ (\CC^2)_{\mathrm{sym}} \big)^{\otimes K}$, and therefore, 
    \[
        \big[ \cos(2\varepsilon) \big]^K \cdot \cSsep\big( (\CC^2)^{\otimes K} \big)_{\mathrm{sym}}
            \subset \cP
            \subset \cSsep\big( (\CC^2)^{\otimes K} \big)_{\mathrm{sym}}
    \]
    where
    \[
        \cP \coloneqq \conv \left\{ \pm \big( \bv_1 \otimes \cdots \otimes \bv_K \big) \otimes \big( \bv_1 \otimes \cdots \otimes \bv_K \big)^*: \bv_1, \dots, \bv_K \in \cN  \right\}.
    \]

    Observe that $\cP$ is contained in the Hilbert-Schmidt unit ball, and thus,
    \[
        w_G\big( \cP \big) = \bbE_g \left[ \sup_{x \in P} \langle g, x \rangle \right]
            \leq \sqrt{2 \log |\cP|}
    \]
    by the well-known upper bound on the expected supremum of (sub-)Gaussians. 
    Moreover, observe that $|\cP| = 2 |\cN|^k \leq 2 \cdot ( 2/\varepsilon )^{2K}$. 
    
    Noticing that $\cos(2\varepsilon) \geq 1 - \frac{(2\varepsilon)^2}{2}$ and $\kappa_d \geq \sqrt{ d - 1/2}$, we obtain
    \begin{align}
        w\Big(\cSsep\big( (\CC^2)^{\otimes K} \big) \Big)
            &\leq w\Big(\cSsep\big( (\CC^2)^{\otimes K} \big)_{\mathrm{sym}} \Big) \nonumber\\
            &= \frac{w_G\Big(\cSsep\big( (\CC^2)^{\otimes K} \big)_{\mathrm{sym}} \Big)}{\kappa_{2^{2K}}} \nonumber\\
            &\leq \frac{1}{\big[ \cos(2\varepsilon) \big]^K} \frac{w_G\left( \cP \right) }{\kappa_{2^{2K}}}
                &&\because \text{Properties }\ref{prop:homog} ~\&~ \ref{prop:ordered} \nonumber\\
            &\leq \frac{ \sqrt{2 \log \left( 2^{2K+1} \cdot \varepsilon^{-2K} \right) } }{\left( 1 - 2\varepsilon^2 \right)^K \cdot \sqrt{2^{2K} - 1/2}}.
                \label{eqn:upper_width}
    \end{align}
    Defining $\varphi: \varepsilon \mapsto \frac{2 \log \left( 2^{2K+1} \cdot \varepsilon^{-2K} \right) }{ (1 - 2\varepsilon^2 )^{2K}}$ and considering its first-order condition, we can see that the upper bound in \eqref{eqn:upper_width} is minimized when $\varepsilon$ satisfies $\frac{1}{4\varepsilon^2} - \frac{1}{2} + 2K \log \varepsilon = (2K+1) \log 2$. 
    Solving this equation approximately (by ignoring the logarithmic term and some constants), we choose $\varepsilon = \frac{1}{\sqrt{8K}}$. 
    Then we have
    \begin{align*}
        w\Big(\cSsep\big( (\CC^2)^{\otimes K} \big) \Big)
            &\leq \frac{ \sqrt{2 \log \left( 2^{5K+1} \cdot K^K \right) } }{\left( 1 - \frac{1}{4K} \right)^K \cdot \sqrt{2^{2K} - 1/2}}\\
            &\leq \sqrt{12e} \frac{\sqrt{K \log K}}{ 2^K}   & \forall K \geq 2.
    \end{align*}
\end{proof}

\subsection{Completing proof of Theorem \ref{thm:abundance}}
Using the four lemmas above, we complete the proof of Theorem \ref{thm:abundance}.
\begin{proof}[Proof of Theorem \ref{thm:abundance}]
    First of all, observe that $\rho_v(\cH) = \frac{ \vrad \big( \cSsep(\cH)\big) }{ \vrad\big( \cSsa_+(\cH) \big) } \leq \frac{ w \big( \cSsep(\cH)\big) }{ \vrad \big( \cSsa_+(\cH) \big) }$ by Urysohn's inequality (Lemma \ref{lem:Urysohn}), and similarly, $\rho_w(\cH) \coloneqq \frac{ w \big( \cSsep(\cH) \big) }{ w\big( \cSsa_+(\cH) \big) } \leq \frac{ w \big( \cSsep(\cH)\big) }{ \vrad \big( \cSsa_+(\cH) \big) }$. 
    Then we observe that
    \begin{align*}
        \frac{ w \big( \cSsep(\CC^d \otimes \CC^d)\big) }{ \vrad \big( \cSsa_+(\CC^d \otimes \CC^d) \big) }
            &\leq  \frac{5}{d^{3/2}} \cdot \left( \frac{1}{2 d} \right)^{-1} 
                &&\because \text{Lemmas }\ref{lem:vrad} ~\&~\ref{lem:bipartite}\\
            &= 10 d^{-1/2},\\
        \frac{ w \big( \cSsep\big((\CC^2)^{\otimes K}\big) \big) }{ \vrad \big( \cSsa_+\big((\CC^2)^{\otimes K}\big) \big) }
            &\leq \sqrt{12 e} \frac{\sqrt{ K \log K }}{2^K} \cdot \left( \frac{1}{2 \cdot 2^{K/2}} \right)^{-1}
                &&\because \text{Lemmas }\ref{lem:vrad} ~\&~\ref{lem:multipartite}\\
            &= \sqrt{48 e} \frac{\sqrt{ K \log K }}{2^{K/2}}.
    \end{align*}
\end{proof}

%% file: contents/SuppC_NPHard.tex
\section{Proof sketch of Theorem \ref{thm:NP-hard}}\label{sec:sketch_NPhard}
As briefly discussed in Section \ref{sec:NPhard_wmem}, the NP-hardness of the decision problem of Kronecker-separability is established through polynomial-time reduction from the Clique problem (known to be NP-complete) to the weak membership problem (Definition \ref{def:weak_membership}). Here we outline the proof of Gharibian \cite{gharibian2008strong}, which relies on the following reduction chain with appropriate choice of parameters $\eta, \epsilon, \beta > 0$:
\begin{equation}\label{eqn:reduction_chain}
    \clique \leqK \rsdf_{\eta} \leqK \wopt_{\epsilon}\left( \cSsep( \CC^{d_1} \otimes \CC^{d_2} ) \right) \leqT \wmem_{\beta}\left( \cSsep( \CC^{d_1} \otimes \CC^{d_2} ) \right),
\end{equation}
where the first three problems, and the relations $\leqK$, $\leqT$ are defined shortly in \Cref{sec:definitions_chain}. 
The three links in the reduction chain \eqref{eqn:reduction_chain} are presented as lemmas, along with references, in \Cref{sec:links_chain}.

\subsection{Definitions}\label{sec:definitions_chain}
To facilitate the interpretation of the reduction chain \eqref{eqn:reduction_chain}, we present several definitions, encompassing the definitions of three distinct problems and the concepts of polynomial-time reduction. 

\vspace{5pt}
\paragraph{Auxiliary problems}
We define the three additional problems appearing in the reduction chain \eqref{eqn:reduction_chain}.

\begin{definition}[\clique]\label{defn:clique}
    The \emph{clique problem} is the computational problem to determine whether a given simple undirected graph $G$ on $n$ vertices contains a complete subgraph (=clique) greater than a given value $c \leq n$. 
\end{definition}

\begin{definition}[\rsdf]
    The \emph{robust semidefinite feasibility problem with error parameter $\eta \geq 0$}, denoted as $\rsdf_{\eta}$, is the decision problem that requires the following: 
    Given $k$ number of $l \times l$ symmetric matrices $B_1, \cdots, B_k$, and a number $\zeta \geq 0$, 
    \begin{enumerate}[label=\arabic*)]
        \item
        If $F(B_1, \dots, B_k) \geq \zeta + \eta$, then output \texttt{YES};
        \item 
        If $F(B_1, \dots, B_k) \leq \zeta - \eta$, then output \texttt{NO};
    \end{enumerate}
    where $F(B_1, \dots, B_k) \coloneqq \max_{x \in \RR^l: \|x\|_2 = 1} \sum_{i=1}^k \left( x^{\top} B_i x \right)^2$.
\end{definition}

Note that the \rsdf{} is defined as a promise problem, i.e., the input is promised to belong to one of two distinct cases, possibly separated by a non-zero gap (when $\eta > 0$), and our task is to differentiate solely between these two cases. Alternatively, we can remove the promise in its interpretation, and treat any output as valid for inputs within the "gap" region, i.e., $B_1, \dots, B_k$ with $\zeta - \eta < f(B_1, \dots, B_k) < \zeta + \eta$.

\begin{definition}[\wopt]
    Let $\cK \subset \RR^d$ be a convex body and $\epsilon \geq 0$. 
    The \emph{weak optimization problem for $K$ with error parameter $\epsilon$}, denoted as $\wopt_{\epsilon}(\cK)$, is the decision problem that requires the following: 
    Given a vector $c \in \RR^d$ and a number $\gamma \in \RR$, 
    \begin{enumerate}[label=\arabic*)]
        \item
        If there exists $x \in \cK_{-\delta} \coloneqq \{x \in \cK: B(x, \delta) \subset \cK \}$ with $\langle c, x \rangle \geq \gamma + \epsilon$, then output \texttt{YES};
        \item 
        If $\langle c, x \rangle \geq \gamma + \epsilon$ for all $x \in \cK_{+\delta} \coloneqq \bigcup_{x \in \cK} B(x, \delta)$, then output \texttt{NO};
    \end{enumerate}
    where $B(x, \delta) \coloneqq \{ y \in \RR^d: \| y - x \| \leq \delta \}$.
\end{definition}

\vspace{5pt}
\paragraph{Polynomial-time reduction and NP-hardness} 
Next, let us briefly review the concepts of ``polynomial-time reduction'' from one problem $\ttA$ to another $\ttB$. 
Let $\cO_{\ttB}$ denote an oracle, or black-box subroutine, which can be used to solve problem $\ttB$ and is assigned unit complexity cost. 
A (polynomial-time) \emph{Turing reduction} from $\ttA$ to $\ttB$ is a polynomial-time algorithm for $\ttA$ that can make use of the oracle $\cO_{\ttB}$. 
When a Turing reduction from $\ttA$ to $\ttB$ exists, $\ttA$ is said to be Turing-reducible to $\ttB$ and we denote it as $\ttA \leqT \ttB$. 
Furthermore, a (polynomial-time) \emph{Karp reduction} from $\ttA$ to $\ttB$ is a specific kind of Turing reduction, in which $\cO_{\ttB}$ is called at most once and only at the end of the reduction algorithm so that the answer provided by $\cO_{\ttB}$ becomes the answer to the original instance of problem $\ttA$. 
We write $\ttA \leqK \ttB$ if $\ttA$ is Karp-reducible to $\ttB$.

The notion of NP-completeness identifies the most challenging problems within the NP class. 
NP-completeness can be defined relative to any polynomial-time reduction. 
For instance, \emph{Karp-NP-completeness} is defined as
\[
    \NPC_K \coloneqq \left\{ \ttA \in \NP : \ttA' \leqK \ttA \text{ for all } \ttA' \in \NP \right\}.
\]
Note that the clique problem (\clique; Definition \ref{defn:clique}) is one of Richard Karp's original 21 problems shown NP-complete in his 1972 paper \cite{karp1972reducibility}.

A problem $\ttB$ is classified as \emph{NP-hard} when there exists a Karp-NP-complete problem $\ttA \in \NPC_K$ such that $\ttA \leqT \ttB$. 
In other words, an NP-hard problem is at least as hard as the Karp-NP-complete problems.

\subsection{The three links in the reduction chain}\label{sec:links_chain} 
In this section, we present lemmas that establish the three connections in the reduction chain \eqref{eqn:reduction_chain},  largely following the statements presented in \cite{gharibian2008strong}. 
However, we omit the proofs of these lemmas due to their extensive length and technical intricacies; interested readers are directed to the original reference for detailed proof explanations. 
Notably, the first and third connections (Lemmas \ref{lem:first} and \ref{lem:third}) were previously established in earlier works, and we include references to the original sources for these connections.

\begin{lemma}[{\cite[Theorem 2]{gharibian2008strong}}; originally from \cite{ioannou2007computational}]\label{lem:first}
    There exists a Karp reduction that maps an instance $(G, n, c)$ of $\clique$ to an instance $(B_1, \dots, B_k, \zeta; \eta)$ of $\rsdf_{\eta}$ such that
    \begin{itemize}
        \item 
        $l = n$, $k = \frac{n(n-1)}{2}$;
        \item 
        $\| B_i\|_F = \Theta(1)$ for all $i \in [k]$;
        \item 
        $\zeta = \Theta(1)$, and $\eta = \Omega( n^{-2})$.
    \end{itemize}
\end{lemma}

\begin{lemma}[{\cite[Lemma 4]{gharibian2008strong}}]\label{lem:second}
    There exists a Karp reduction that maps an instance $(B_1, \dots, B_k, \zeta; \eta)$ of $\rsdf_{\eta}$ to an instance $(c, \gamma; \epsilon)$ of $\wopt_{\epsilon}\left( \cSsep( \CC^{d_1} \otimes \CC^{d_2} ) \right)$ such that
    \begin{itemize}
        \item
        $d_1 = k+1$, $d_2 = \frac{l(l-1)}{2} + 1$;
        \item 
        $c = \frac{\hat{c}}{\| \hat{c}\|_2}$ for some $\hat{c} \in \RR^d$ with $\| \hat{c} \|_2^2 = O\left( d\cdot \sum_{i=1}^k \| B_i\|_F^2 \right)$ and $d = d_1^2 d_2^2 - 1$;
        \item
        $\gamma = \frac{1}{2 \| \hat{c} \|_2} \left( \sqrt{\zeta + \eta} + \sqrt{\zeta - \eta} \right)$;
        \item 
        $\epsilon \leq \frac{ \sqrt{\zeta + \eta} - \sqrt{\zeta - \eta}}{4\|\hat{c}\|_2 \left( d_1 d_2 - 1 \right) + 1}$.
    \end{itemize}
\end{lemma}

\begin{lemma}[{\cite[Theorem 5]{gharibian2008strong}}; originally from {\cite[Proposition 2.8]{liu2007complexity}}]\label{lem:third}
    Let $\cK \subset \RR^d$ be a convex body, and suppose there exist $p \in \RR^d$ and $r, R > 0$ such that $B(p, r) \subseteq \cK \subseteq B(0, R)$. 
    Given an instance $\Pi = (\cK, c, \gamma; \epsilon)$ of $\wopt_{\epsilon}(\cK)$, there exists a Turing reduction that solves $\Pi$ using an oracle for $\wmem_{\beta}(\cK)$ with $\beta = \frac{r^3 \epsilon^3}{2^{13} 3^3 d^5 R^4 (R+r)}$, running in polynomial time with respect to $R$ and $\lceil 1/\epsilon \rceil$.
\end{lemma}

%% file: contents/SuppD_Frank_Wolfe.tex
\section{More on the Frank-Wolfe algorithm}\label{sec:Frank_Wolfe}
In this section, we provide a concise overview of the celebrated Frank-Wolfe algorithm \cite{frank1956algorithm} and its basic variants, along with their convergence analysis. 
Our focus here is to offer a foundational understanding of the method and its convergence results concisely. 
Interested readers are encouraged to explore more comprehensive resources such as \cite{jaggi2013revisiting} and \cite{lacoste2015global} for more in-depth information and enhanced linear convergence results pertaining to more sophisticated algorithmic variants (e.g., those exploiting ``away steps'').

\subsection{Generic description of the Frank-Wolfe method}\label{sec:FW_description}
Let us consider the constrained convex optimization problem of the form
\begin{equation}\label{eqn:const_conv_prob}
    \min_{\bx \in \cD} f(\bx)
\end{equation}
where $f$ is a convex, continuously differentiable function, and $\cD \subset \cH$ is a compact, convex set in a Hilbert space $\cH$ over $\RR$. 
The Frank-Wolfe method \cite{frank1956algorithm} described in Algorithm \ref{alg:Frank_Wolfe}, which is also known as the conditional gradient method \cite{levitin1966constrained}, is one of the simplest and earliest known iterative method to solve such a constrained optimization problem.

\begin{algorithm}
\caption{Frank-Wolfe method.}\label{alg:Frank_Wolfe}
    \begin{algorithmic}
    \State \textbf{Input:} $\bx^{(0)} \in \cD$, $T \in \NN$
    \For{ $t = 0, \dots, T$ }
        \State a) Compute $\bs^{(t)} \gets \argmin_{\bs \in \cD} \left\langle \bs, \nabla f(\bx^{(t)}) \right\rangle$
        \State b) Update $\bx^{(t+1)} \gets (1-\gamma^{(t)}) \cdot \bx^{(t)} + \gamma^{(t)} \cdot \bs^{(t)}$, where $\gamma^{(t)} := \frac{2}{t+2}$
    \EndFor
    \end{algorithmic}
\end{algorithm}

Each iteration of the Frank-Wolfe algorithm (Algorithm \ref{alg:Frank_Wolfe}) consists of two steps: (i) finding a promising direction to move, namely, $\bs_k$, based on minimizing the objective function linearized at the current iterate $\bx^{(k)}$, and (ii) updating the incumbent solution $\bx^{(t+1)}$ by taking a convex combination of the old iterate $\bx^{(t)}$ and the promising atom $\bs^{(t)}$. 
It is easy to see that one can always choose $\bs^{(t)}$ located at the boundary of $\cD$ (i.e., $\bs^{(t)}$ is an extreme point of $\cD$), and in particular, $\bs^{(t)} \in \cA$ when $\cD = \conv \cA$.

There have been numerous variations of Algorithm \ref{alg:Frank_Wolfe} that have been explored and found practical utility. 
Here, we present a description of some particularly pertinent ones.

\vspace{5pt}
\paragraph{Approximately solving the linear subproblem in a)}
The complexity of exactly solving the linearized subproblem $\min_{\bs \in \cD} \langle \bs, \nabla f(\bx^{(t)})\rangle$ can be prohibitive, e.g., when the feasible set $\cD$ is complicated. 
In such scenarios, it could be pragmatic to modify step a) by allowing for an approximate solution to the linearized subproblem.  
This involves finding any $\hat{\bs}^{(t)}$ that satisfies the inequality
\begin{equation}\label{eqn:linear_subprob}
    \big\langle \hat{\bs}^{(t)}, \nabla f( \bx^{(t)}) \big\rangle 
        \leq
        \min_{\bs \in \cD} \left\langle \bs, \nabla f\big(\bx^{(t)}\big) \right\rangle + \frac{1}{2} C_f \gamma^{(t)} \eta, 
\end{equation}
where $\gamma^{(t)}$ is the stepsize used in step b), $\eta \geq 0$ represents a predetermined accuracy parameter, and $C_f$ corresponds to the \emph{curvature constant} of $f$ with respect to $\cD$, cf. \cite[Eq. (3)]{jaggi2013revisiting}, defined as follows:
\begin{equation}\label{eqn:curvature}
    C_f \coloneqq \sup_{ \substack{ \bx, \bs \in \cD, \\ \gamma \in [0,1], \\ \by = \bx + \gamma(\bs - \bx) } } \frac{2}{\gamma^2} \Big( f(\by) - f(\bx) - \big\langle \by - \bx, \nabla f (\bx ) \big\rangle \Big).
\end{equation}
Here, the defining term $f(\by) - f(\bx) - \langle \by - \bx, \nabla f(\bx) \rangle$ is known as the Bregman divergence induced by $f$. 
As noted by \cite{jaggi2013revisiting}, the assumption of bounded curvature $C_f$ is closely related to the Lipschitz assumption on the gradient of $f$.

\vspace{5pt}
\paragraph{Adaptive stepsize in b)} 
Instead of relying on the pre-determined stepsizes like $\gamma^{(t)} = \frac{2}{t+2}$, it is possible to follow a more involved approach to optimize over the active set $\cS^{(t)} \coloneqq \{ \bx^{(0)}, \bs^{(0)}, \dots, \bs^{(t)} \}$ to compute $\bx^{(t+1)}$. 
Here we elaborate on the line-search and the fully corrective variants.
\begin{itemize}[topsep=5pt, itemsep=3pt]
    \item 
    \emph{Line-search variant.} 
    In this variant, the approach remains the same as in Algorithm \ref{alg:Frank_Wolfe}, but with the use of an adaptive stepsize:
    \[
        \gamma^{(t)} := \argmin_{\gamma \in [0,1]} f \left( \bx^{(t)} + \gamma \big( \bs^{(t)} - \bx^{(t)}) \right).
    \]
    
    \item
    \emph{Fully corrective variant.} 
    This more intensive variant continually updates the active set $\cS^{(t)} := \{ \bx^{(0)}, \bs^{(0)}, \dots, \bs^{(t)} \}$ and computes $\bx^{(t+1)}$ by re-optimizing $f$ over $\conv \cS^{(t)}$. Various sub-variants exist based on how this correction step is executed, including the extended FW method \cite{holloway1974extension}, the simplicial decomposition method \cite{von1977simplicial, hearn1987restricted}, and the min-norm point algorithm \cite{wolfe1976finding}.
\end{itemize}

\subsection{Convergence analysis}
The convergence behavior of the Frank-Wolfe method (Algorithm \ref{alg:Frank_Wolfe}) has been well-established. 
It is known that the iterates of the basic form of the Frank-Wolfe method satisfy $f(\bx^{(t)}) - f(\bx^*) = O\big(\frac{1}{t}\big)$ where $\bx^*$ denotes an optimal solution to the problem \eqref{eqn:const_conv_prob} \cite{frank1956algorithm, dunn1978conditional}. 

Recently, Frank-Wolfe methods and their variants have seen a renewed surge of interest spanning various domains, including the machine learning community. 
Notably, significant research efforts have been directed towards analyzing the convergence properties of different adaptations of the Frank-Wolfe method. 
Here we adapt the convergence result of \cite{jaggi2013revisiting} for the primal error of the proposed version of the FW algorithm.

\begin{theorem}[{\cite[Theorem 1]{jaggi2013revisiting}}]\label{thm:FW_conv}
    For each $t \in \NN$, the iterates $\bx^{(t)}$ of Algorithm \ref{alg:Frank_Wolfe} and all of its variants (i.e., approximate liner oracle and/or adaptive step-size) described in \Cref{sec:FW_description} satisfy
    \[
        f(\bx^{(t)}) - f(\bx^*) \leq \frac{2 C_f}{t+2}(1+\eta),
    \]
    where $\bx^* \in \cD$ is an optimal solution to the problem \eqref{eqn:const_conv_prob}, $\eta \geq 0$ is the accuracy parameter to which the linear subproblems are solved, cf. \eqref{eqn:linear_subprob}, and $C_f$ is the curvature constant of $f$ with respect to $\cD$, cf. \eqref{eqn:curvature}.
\end{theorem}
This $O(1/t)$ rate is recognized to be worst-case optimal, as shown in \cite[Lemma 3 \& Appendix C]{jaggi2013revisiting}. 
We remark here that certain more sophisticated variations of the Frank-Wolfe method, incorporating techniques such as "away steps" or "pairwise correction" to mitigate zig-zagging phenomena, have been demonstrated to achieve linear convergence under certain conditions \cite{lacoste2015global}.

\subsection{Proof of Theorem \ref{thm:convergence_FW}}\label{sec:proof_thm.5}

\begin{proof}[Proof of Theorem \ref{thm:convergence_FW}]
    Recall from \eqref{eqn:KS_approx_unit} that $\upisep(\bSigma)$ is the unique minimizer of 
    $f(\bX) := \| \bX - \bSigma \|^2$ in the feasible set $\cSsep(\cH)$. 
    It follows from Theorem \ref{thm:FW_conv} that
    \[
        f\big( \hbSigma^{(t)} \big) - f\big( \upisep( \bSigma) \big) \leq  \frac{2 C_f}{t+2}(1+\eta)
    \]
    where $C_f$ is the curvature constant of $f$ with respect to $\cD$, cf. \eqref{eqn:curvature}. 
    
    Observe that
    \begin{align*}
         f(\by) - f(\bx) - \big\langle \by - \bx, \nabla f (\bx ) \big\rangle
            &= \| \by - \bSigma \|^2 - \| \by - \bSigma \|^2 - 2 \langle \by - \bx, \bx - \bSigma \rangle
            = \| \by - \bx \|^2,
    \end{align*}
    and thus, $C_f = 2\, \diam \cBsep(\cH)^2 $. Since 
    \begin{align*}
        \cSsep(\cH) 
            &= \left\{ \bSigma' \in \cBsep(\cH): \tr \bSigma' = 1 \right\}\\
            &= \conv \left\{ \bigg( \bigotimes_{i=1}^K \bv_i \bigg) \otimes \bigg( \bigotimes_{i=1}^K \bv_i \bigg)^* : \bv_i \in \cH_i ~\&~ \| \bv_i \|=1, ~\forall i \in [K]  \right\},
    \end{align*}
    we can verify that
    \begin{align*}
        \diam \cBsep(\cH)
            &= \sup_{\bx, \by \in \cBsep(\cH) } \| \bx - \by \|\\
            &= \sup_{\bv_i, \bw_i \in \cH_i: \|\bv_i\| = \| \bw_i \|=1 } 
                \left\| \bigg(\bigotimes_{i=1}^K \bv_i \bigg) \otimes \bigg( \bigotimes_{i=1}^K \bv_i \bigg)^* 
 - \bigg( \bigotimes_{i=1}^K \bw_i \bigg) \otimes \bigg( \bigotimes_{i=1}^K \bw_i \bigg)^* \right\|\\
            &= \sup_{\bv_i, \bw_i \in \cH_i: \|\bv_i\| = \| \bw_i \|=1 }  \left( 2 - \prod_{i=1}^K \langle \bv_i, \bw_i \rangle^2 \right)^{1/2}\\
            &= \sqrt{2}.
    \end{align*}

    Lastly, by considering the Taylor series expansion of the quadratic objective function $f$ around the optimal solution $\bx^* = \upisep(\bSigma)$, we can deduce that for any feasible point $\bx \in \cSsep(\cH)$,
    \begin{align*}
        f(\bx) 
            &= f(\bx^*) + \langle \nabla f(\bx^*), \bx - \bx^* \rangle + \frac{1}{2} \langle \bx - \bx^*, \nabla^2 f(\bx^*) (\bx - \bx^*) \rangle\\
            &\geq f(\bx^*) + \left\| \bx - \bx^* \right\|^2
    \end{align*}
    because $\langle \nabla f(\bx^*), \bx - \bx^* \rangle \geq 0$ and $\nabla^2 f(\bx^*) = 2 \bI$. 
    
    Consequently,
    \begin{align*}
        \left\| \hbSigma^{(t)} - \upisep(\bSigma) \right\|^2
            \leq f \big( \hbSigma^{(t)} \big) - f \big(  \upisep(\bSigma) \big) 
            \leq \frac{8 (1 +\eta)}{t+2}.
    \end{align*}

\end{proof}

%% file: contents/SuppE_experiments.tex
\section{Additional details on numerical experiments}\label{sec:supplement_experiments}

This section provides further technical details on the numerical implementations of our Kronecker-separable approximation algorithms described in the main text (Section \ref{sec:experiments}) and the design of various ablation studies.

\subsection{Implementation details}

\subsubsection{Linear minimization oracle for the Frank-Wolfe method}
Recall from Section \ref{sec:alg_Frank_Wolfe} of the main text that each iteration of the Frank–Wolfe (FW) method requires solving a \emph{linear minimization subproblem} (LMS):
\[
    \bZ^{(t)} \in \argmin_{\bZ \in \cA} \left\langle \bZ, \hbSigma^{(t)} - \bSigma \right\rangle,
\]
where $\cA$ is the set of rank-1 product projectors, cf. \eqref{eqn:setA}. 
Here we describe one practical approach to solve this LMS using a nonconvex toric parameterization and an alternating maximization scheme, which is summarized in Algorithm \ref{alg:LMO}.

\paragraph{Toric parameterization} 
Let $\bbS_i \coloneqq \left\{ \bv_i \in \cH_i: \| \bv_i \| = 1 \right\}$ denote the unit sphere in $\cH_i$, for each $i \in [K]$. 
Observe that LMS can be equivalently reformulated as a maximization problem over the torus $\bbS_1 \times \cdots \times \bbS_K$. 
Specifically, $\bZ^{(t)} = \Big( \bigotimes_{i=1}^K \bv_i^{(t)} \Big) \otimes \Big( \bigotimes_{i=1}^K \bv_i^{(t)} \Big)^*$ where
\begin{equation}\label{eqn:toric_maximization}
    \left( \bv_i^{(t)}, \cdots, \bv_K^{(t)} \right) \in 
        \argmax_{\bv_i \in \bbS_i, ~\forall i \in [K]}
            \left\langle \left( \bigotimes_{i=1}^K \bv_i^{(t)} \right) \otimes \left( \bigotimes_{i=1}^K \bv_i^{(t)} \right)^*, ~ \ovbSigma - \hbSigma^{(t)} \right\rangle.
\end{equation}
Note that this is the problem of finding a maximum eigenvector $\bv_*^{(t)}$ of $\ovbSigma - \hbSigma^{(t)}$ (represented as a matrix) with a rank-1 tensor product constraint $\bv_*^{(t)} = \bigotimes_{i=1}^K \bv_i^{(t)}$.

\paragraph{Alternating maximization} 
For an ordered tuple $\cV = (\bv_1, \dots, \bv_K) \in \bbS_1 \times \cdots \times \bbS_K$ and for each $i \in [K]$, we define a linear map $\Psi_{\cV, i}: \cB(\cH) \to \cB(\cH_i)$ that takes a partial inner product against $\{ \bv_{j} \}_{j \neq i}$ such that
\begin{equation}\label{eqn:partial_inner}
    \Psi_{\cV, i}: \left( \bigotimes_{i=1}^K \bu_i \right) \otimes \left( \bigotimes_{i=1}^K \bw_i \right)^* \mapsto \left(\prod_{j=1: ~j \neq i}^K \langle \bv_j, \bu_j \rangle \langle \bw_j, \bv_j \rangle \right) \bu_i \otimes \bw_i^*.
\end{equation}
By linearity, this map is well-defined for the entire $\cB(\cH)$: thus, for any $\bSigma' \in \cBsa_+(\cH)$, its image $\Psi_{\cV, i}(\bSigma')$ is obtained by taking partial inner product with $\bv_j \bv_j^*$ for all $j \in [K]$ but $i$. 
To illustrate this map $\Psi_{\cV, i}$, we consider a simple case where $K = 2$, $\cH_1 = \FF^m$, and $\cH_2 = \FF^n$, and represent $\bSigma' \in \FF^{m \times m \times n \times n}$ by fixing a basis. 
Then it follows that
\begin{align*}
    \Psi_{\cV, 1}(\bSigma') \in \FF^{m \times m} &\quad\text{such that}\quad
        \Psi_{\cV, 1}(\bSigma)_{ab} = \sum_{a',b' = 1}^n \Sigma'_{aba'b'} (\bv_2)_{a'}\overline{(\bv_2)_{b'}},\\
    \Psi_{\cV, 2}(\bSigma') \in \FF^{n \times n} &\quad\text{such that}\quad
        \Psi_{\cV, 2}(\bSigma)_{a'b'} = \sum_{a,b = 1}^m \Sigma'_{aba'b'} (\bv_1)_{a}\overline{(\bv_1)_{b}},
\end{align*}
where $(\bv_i)_a$ is the $a$-th coordinate of $\bv_i$, and $\overline{c}$ denotes the complex conjugate of $c$.

Now we describe an alternating maximization algorithm that iteratively solves the problem \eqref{eqn:toric_maximization}. 
Starting from an arbitrary initialization, each iteration updates one $\bv_i$ at a time by finding the top eigenvector of $\Psi_{\cV, i}(\ovbSigma - \hbSigma^{(t')})$. 
Though non-convex, this simple alternating scheme typically converges quickly in practice.
We generally initialize $\cV$ by sampling $\{ \bv_i^{(0)} \}$ uniformly at random on each sphere $\bbS_i$. 
The pseudocod of this procedure is summarized in Algorithm \ref{alg:LMO}.

\begin{algorithm}
\caption{Alternating maximization algorithm for the linear minimization sub-problem (LMS).}
\label{alg:LMO}
\begin{algorithmic}[1]
    \Require 
        $\bSigma' \in \cBsa_+(\cH)$ where $\cH = \bigotimes_{i=1}^K \cH_i$
        \qquad(e.g., $\bSigma' \gets \ovbSigma - \hbSigma^{(t)}$, in Algorithm \ref{alg:frank_wolfe} Line 3)
    \Ensure 
        Approximate solution $\hZ$ to the LMS in FW iterations
    \State Choose an arbitrary sequence of unit vectors $\cV = \left( \bv_1^{(0)}, \cdots, \bv_K^{(0)} \right) \in  \bbS_1 \times \cdots \times \bbS_K$
    \For{ $t' = 0, \dots, T$}
        \For{ i = 1, \dots, K }
            \State 
            Find $\bv_i^{(t'+1)} \in \bbS_i$ such that $\bv_i^{(t'+1)} \in \argmax_{\bv_i \in \bbS_i} \langle \bSigma'_i, ~ \bv_i \bv_i^*  \rangle$ where $\bSigma'_i \coloneqq \Psi_{\cV, i}(\ovbSigma - \hbSigma^{(t')})$
            \State 
            Replace $\bv_i^{(t')}$ in $\cV$ with $\bv_i^{(t'+1)}$
        \EndFor
    \EndFor
    
    \State $\hZ = \big( \bigotimes_{i=1}^K \bv_i^{(T)} \big) \otimes \big( \bigotimes_{i=1}^K \bv_i^{(T)} \big)^*$
\end{algorithmic}
\end{algorithm}

In our code implementation of the FW method, this procedure (Algorithm~\ref{alg:LMO}) is applied each FW iteration to find a good rank-1 product projector $\bZ^{(t)}$. 
Despite the nonconvex tensor-rank constraint, the alternating updates often yield near-optimal atoms within practical runtimes, especially for moderate dimensions.

\subsubsection{Technical details in the implementation of the gradient descent method}
Recall from Section \ref{sec:alg_Gradient_Descent} in the main text that our gradient-descent (GD) method parameterizes a candidate separable state with a mixture of $\Napprox$ rank-1 factors (see Algorithm \ref{alg:gradient_descent}). 
In our PyTorch implementation, we store all trainable parameters collectively as 
\[
    \theta := \left( ( \tilde{\bv}_{i,j} )_{j=1}^{K}, \tilde{w}_i \right)_{i=1}^{\Napprox},
\]
where each $\tilde{\bv}_{i,j}$ is a raw local vector (real or complex), and $\tilde{w}_i$ is a log-weight. 
From these, we derive (1) \emph{mixture weights} $w_i(\theta) =  \softmax_i( \tw_1, \dots, \tw_{\Napprox})$, and (2) \emph{local unit vectors} $\bv_{i,j}(\theta) = \frac{\tilde{\bv}_{i,j}}{\|\tilde{\bv}_{i,j}\|}$. 

\paragraph{Forward/backward pass} 
Given $\theta$, we construct the approximating state 
\[
    \hbSigma(\theta) = \sum_{i=1}^{\Napprox} w_i(\theta) \left( \bv_{i,1} \otimes \dots \otimes \bv_{i,K} \right) \otimes \left( \bv_{i,1} \otimes \dots \otimes \bv_{i,K} \right),
\]
and compute the Hilbert–Schmidt loss (=squared Hilbert-Schmidt distance to the target $\bSigma$): 
\[
    \mathcal{L}(\theta) := \| \hbSigma(\theta) - \bSigma \|^2.
\]
We evaluate this loss entirely in PyTorch, which allows automatic differentiation and backpropagation through all local vectors and weights.

\paragraph{Optimization/training} 
We typically use Adam as our optimizer for rapid convergence, although other standard optimizers (SGD, L-BFGS, etc.) are also possible. 
In most experiments, we set either a fixed learning rate (e.g., $10^{-1}$ or $10^{-3}$) or a diminishing rate that decays by a factor of $\rho$ (e.g., $1/\sqrt{2}$) each epoch. 
Additional hyperparameters include (1) the number of epochs (between 1 and 10), (2) steps per epoch (between 100 and 5000), and (3) early-stopping tolerance (ranging from $10^{-6}$ and $10^{-12}$). 
These hyperparameters are all set at moderate values, and can be adjusted depending on the problem’s dimension or desired accuracy. 

Empirically, this ``shallow'' GD converges reasonably quickly for moderate dimensions and can often find accurate approximations (see Figure xxx).

\subsubsection{Deep-parameterized GD (DP-GD)}\label{sec:DP-GD_details}
We now describe a \emph{deep parameterization} variant of gradient descent—referred to as DP-GD—that replaces the explicit parameter list $\theta$ with a trainable neural network. 
Unlike the shallow GD approach, which directly stores each mixture component’s parameters, $( w_i, (\bv_{i,k})_{k=1}^K )$, DP-GD learns them through a multilayer perceptron (MLP)---i.e., fully connected feedforward neural network with nonlinear activation functions---that \emph{generate} the parameters. 
Specifically, we treat each mixture component index $i \in \{1, \dots, \Napprox\}$ as a one-hot input vector, which the MLP maps to an output $( w_i, (\bv_{i,k})_{k=1}^K )$. 

Here we provide a self-contained account of DP-GD, as it is not covered in the main text.

\paragraph{Network architecture}
We consider a fully connected feedforward network $F_{\phi}$ with one or two hidden layers (each of width 25–400) and rectified linear units (ReLUs) as hidden activations. 
For each mixutre index $i \in \{1, \dots, \Napprox\}$, we feed a one-hot vector $\be_i \in \RR^{\Napprox}$ into $F_{\phi}$ and obtain an output vector
\[
    F_{\phi}(\be_i) =  ( \tilde{w}_i (\phi), \tilde{\bv}_{i,k}(\phi) )_{k=1}^K )
        \in
        \begin{cases}
           \RR^{1 + \sum_k d_k}, & \text{(real mode)},\\
           \RR^{1 + 2\sum_k d_k}, & \text{(complex mode)},
        \end{cases}
\]
which encodes (1) a log-weight $\tilde{w}_i$ and (2) the raw coordinates for the local vectors $\tilde{\bv}_{i,k}$ in real numbers\footnote{A complex number $c \in \CC$ is encoded in two real numbers representing the real and the imaginary parts as $(\Re c, \Im c)$.}. 

In the final layer, we process $F_{\phi}(\be_i)$ to generate $( w_i (\phi), \bv_{i,k}(\phi) )_{k=1}^K )$ through the following output transformations.
\begin{itemize}
    \item 
    \textbf{Real-valued mode:} 
    We append a sigmoid activation in the final layer so that all coordinates lie in $[0,1]$. 
    We then normalize the first entry (the weight coordinate) by their sum (across $i \in [\Napprox]$), and normalize the remaining entries in each block to produce local unit vectors $\bv_{i,k}$. 
    \item
    \textbf{Complex-valued mode:}
    We leave the final layer as linear. 
    We then separately apply \texttt{torch.sigmoid} transform only to $\tilde{w}_i$ (to keep it in the range $[0,1]$). 
    For the local vectors, we aggregate real and imaginary parts in each block and normalize them via $\ell_2$-norm. 
\end{itemize}

\paragraph{Forward/backward pass} 
Once all $\{ w_i, \bv_{i,k}\}$ are generated (in a single forward pass for $i = 1, \dots, \Napprox$, we  construct the approximating state 
$\hbSigma(\phi)$ and compute the Hilbert–Schmidt loss $\mathcal{L}(\phi) := \| \hbSigma(\phi) - \bSigma \|^2$. 
Again, we evaluate this loss entirely in PyTorch, which allows automatic differentiation for backpropagation. 

\paragraph{Optimizer and hyperparameters}
We use Adam with a typical learning rates (e.g., $10^{-3}$ or $10^{-2}$) and optional decay per epoch. 
For many moderate-dimensional problems, a single or two hidden layers suffice, and we compare the performance by varying each hidden layer’s width in $\{25, 50, 100, 200, 400\}$ and depth $\{2, 3\}$ (i.e., one or two hidden layers). 
These values provide enough expressive capacity to handle mid-sized separable approximations in our experiments without incurring severe runtime overhead.

\paragraph{Performance observations}
Empirically, DP-GD can converge to high-quality solutions faster than shallow GD as the width and depth increase, assuming the learning rate is sufficiently small to ensure stable convergence; see Figure xxx. 
It is possibly because the learned network parameters $\phi$ effectively share representation across mixture components, sometimes helping avoid local minima. 
However, we note that this benefit comes with extra overhead in hyperparameter tuning and more complex behaviors (e.g., sensitivity to large learning rates). 
In practice, we find that for moderate-size problems, DP-GD either matches or slightly outperforms shallow GD once the network depth/width and learning rate are set appropriately.
For smaller tasks, shallow GD with a well-chosen (more aggressive) step size can often be simpler and just as effective.

\subsection{Comparing methodological variations}\label{sec:ablation_studies}
Here, we investigate how various design choices (step size rules, optimizers, network architecture) affect the performance of FW, GD, and DP-GD. 
Overall, these ablation results highlight that each method can be tuned for better performance: fully-corrective FW is rapid but expensive per iteration; Adam with a decaying LR is robust in shallow GD; and DP-GD benefits from well-chosen network width, depth, and step sizes.

\subsubsection{Frank-Wolfe stepsize rules}
First of all, we compare the performance of three step-size rules of the Frank-Wolfe algorithm described in Section \ref{sec:alg_Frank_Wolfe}, namely, fixed diminishing stepsize rule $\gamma^{(t)} = \frac{2}{t+2}$, line-search, and fully-corrective updates\footnote{See \Cref{sec:FW_description} for the description of line-search and fully-corrective updates.}. 

\begin{figure}[ht]
    \centering
    \begin{subfigure}[t]{0.32\textwidth}
        \centering
        \includegraphics[width=\linewidth]{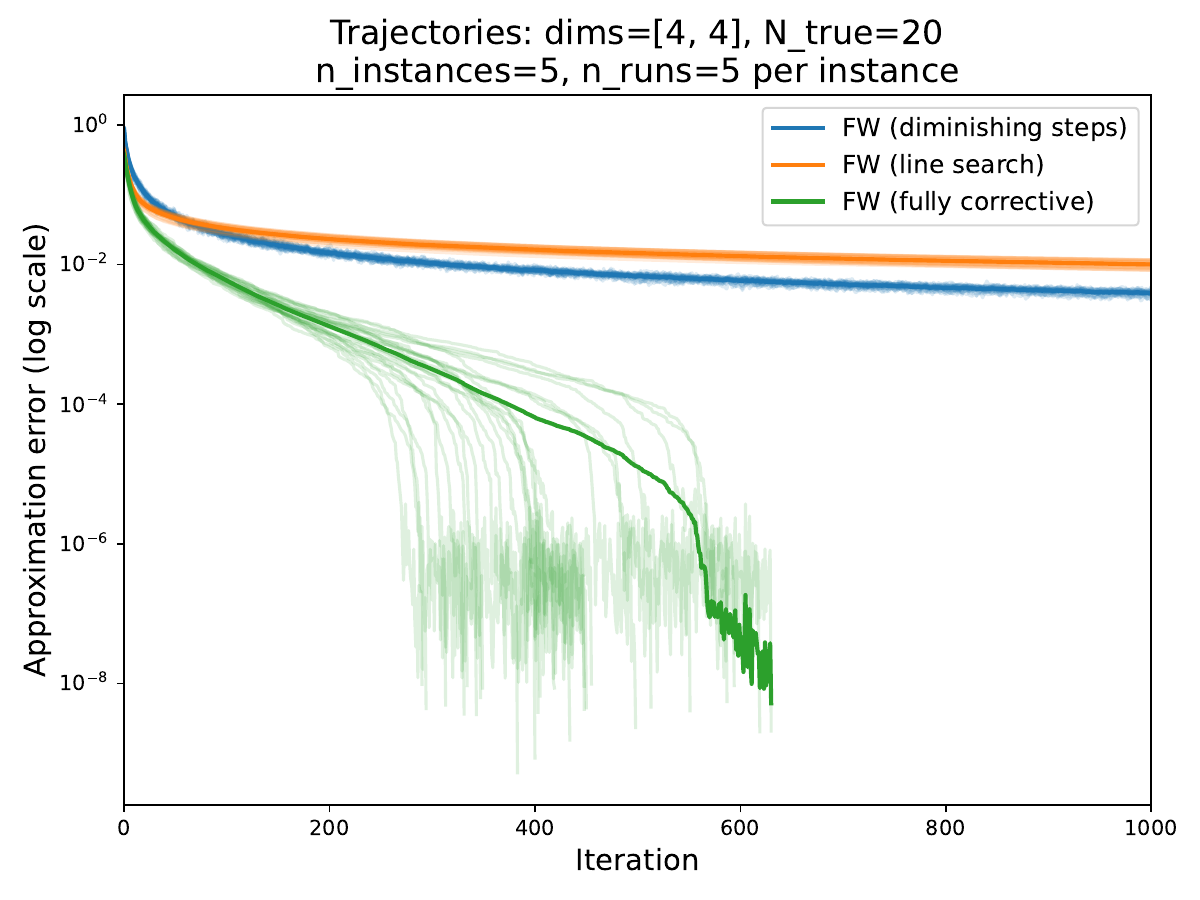}
        \caption{\footnotesize 
        FW config 1 (dim $4\times4$, $\Ntrue=20$).
        }
        \label{fig:FW_stepsizes.1}
    \end{subfigure}
    \hfill
    \begin{subfigure}[t]{0.32\textwidth}
        \centering
        \includegraphics[width=\linewidth]{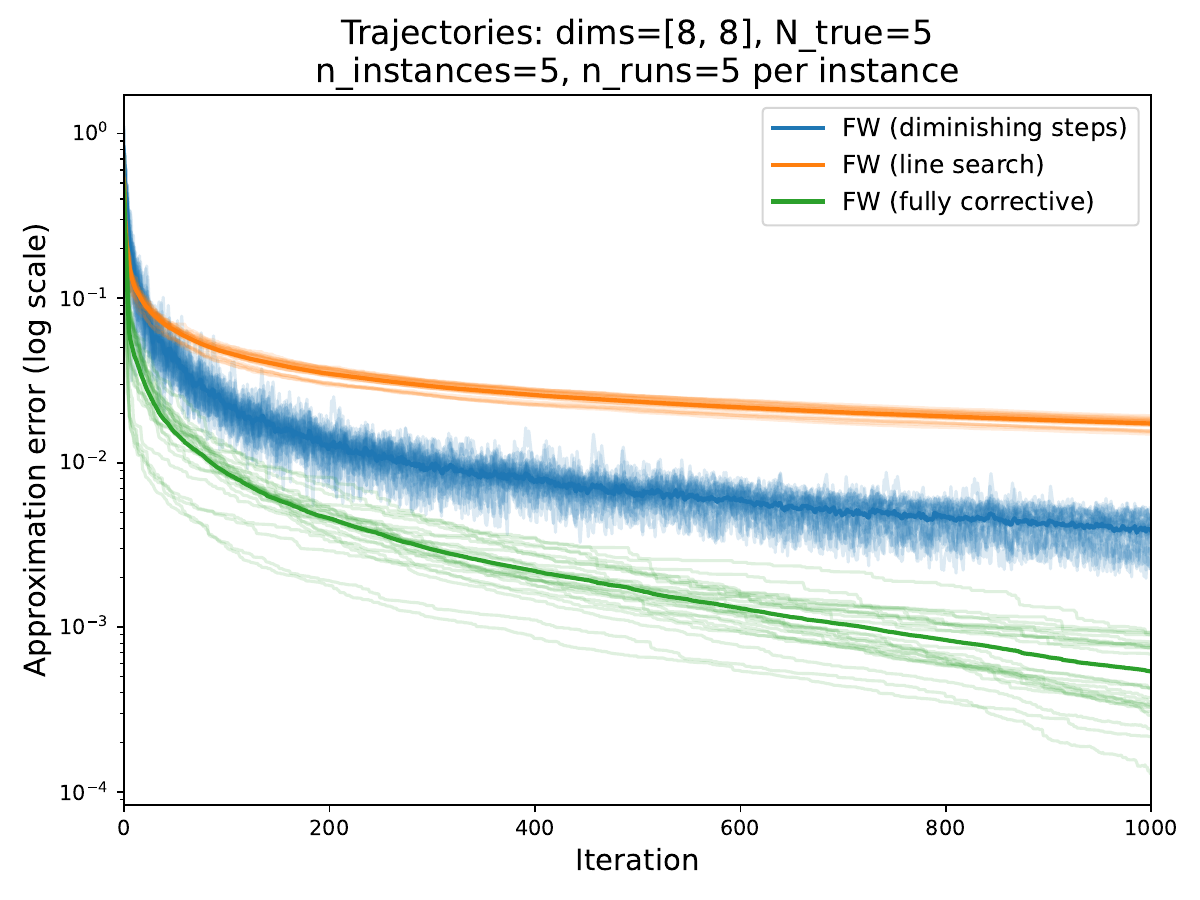}
        \caption{\footnotesize 
        FW config 2 (dim $8\times8$, $\Ntrue=5$).
        }
        \label{fig:FW_stepsizes.2}
    \end{subfigure}
    \hfill
    \begin{subfigure}[t]{0.32\textwidth}
        \centering
        \includegraphics[width=\linewidth]{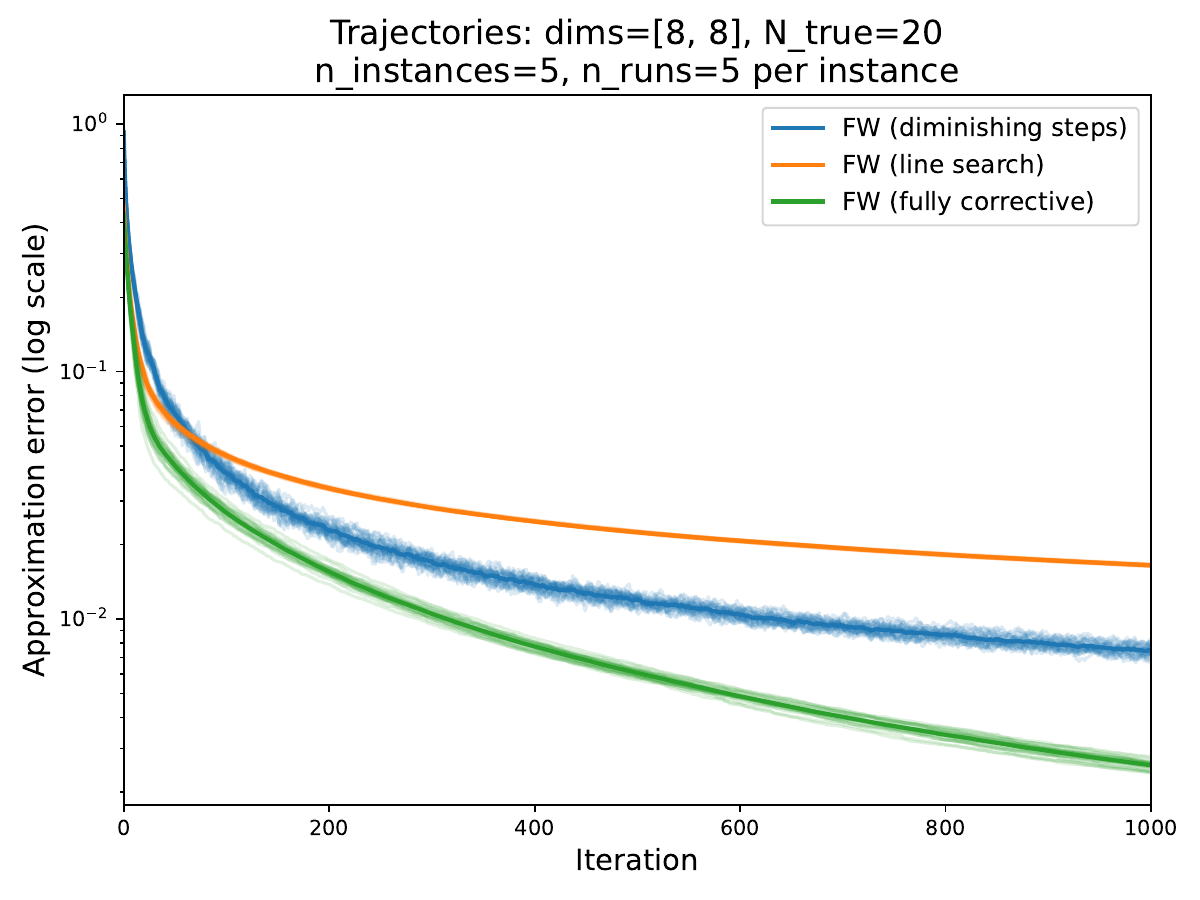}
        \caption{\footnotesize 
        FW config 3 (dim $8\times8$, $\Ntrue=20$).
        }
        \label{fig:FW_stepsizes.3}
    \end{subfigure}

    \caption{\small
    \textbf{Comparison of FW step-size rules:} Fully-corrective updates converge fastest (sometimes even terminating early by reaching the early termination threshold $10^{-8}$) but are costlier per iteration; fixed diminishing offers a good trade-off, and line search is steadier but often slower.
    }
    \label{fig:FW_stepsizes}
\end{figure}

We find that fully-corrective updates converge fastest, though each iteration incurs an expensive least-squares reweighting, while the simpler diminishing and line-search rules show slower but steady progress; see Figure \ref{fig:FW_stepsizes}. 
Table~\ref{tab:frank_wolfe_time} reports average iteration times (across 10 independent runs over 10 random problem instances each). Fully-corrective FW is up to two orders of magnitude slower than simpler rules in higher dimensions.

\begin{table}[t]
    \caption{Mean ($\pm 1$ standard deviation) time, measured in seconds, consumed per iteration (first $50$ iterations) of $10$ independent runs over $10$ random separable covariance instances.}
    \label{tab:frank_wolfe_time}
    \centering
    \begin{adjustbox}{max width=\textwidth}
    \begin{tabular}{l c c c}
        \toprule
        Local dimension $d$        &   Diminishing steps    &   Line-search    &   Fully corrective updates\\
        \midrule
        2       &   $8.4403 \times 10^{-4}$ ($\pm 3.6765 \times 10^{-4}$)       &    $8.7667 \times 10^{-4}$ ($\pm 3.0834 \times 10^{-4}$)   &   $4.0465 \times 10^{-3}$ ($\pm 1.4880 \times 10^{-3}$) \\
        4       &   $1.2426 \times 10^{-3}$ ($\pm 4.5005 \times 10^{-4}$)       &    $1.4844 \times 10^{-3}$ ($\pm 4.6601 \times 10^{-4}$)   &   $8.6128 \times 10^{-3}$ ($\pm 4.8758 \times 10^{-3}$) \\
        8       &   $1.7363 \times 10^{-3}$ ($\pm 6.2351 \times 10^{-4}$)       &    $2.0288 \times 10^{-3}$ ($\pm 7.7111 \times 10^{-4}$)   &   $9.4452 \times 10^{-2}$ ($\pm 3.9241 \times 10^{-2}$) \\
        16      &   $1.6333 \times 10^{-2}$ ($\pm 9.9872 \times 10^{-3}$)       &    $1.4358 \times 10^{-2}$ ($\pm 8.5820 \times 10^{-3}$)   &   $9.1910 \times 10^{-1}$ ($\pm 3.5155 \times 10^{-1}$) \\
        32      &   $1.8818 \times 10^{-1}$ ($\pm 8.6191 \times 10^{-2}$)       &    $1.9480 \times 10^{-1}$ ($\pm 9.0450 \times 10^{-2}$)   &   $1.2051 \times 10^{1}$ ($\pm 3.8209$) \\
        \bottomrule
    \end{tabular}
    \end{adjustbox}
\end{table}

\subsubsection{Gradient descent optimizer and stepsize}
Next, we study how different optimizers (e.g.\ SGD, SGD+momentum, Adadelta, Adam, L-BFGS) and learning rates (fixed rate at various values and decaying schedule each epochs) influence shallow GD.
Figure~\ref{fig:GD_analysis} shows that \textbf{Adam} and \textbf{L-BFGS} often achieve better early progress than the FW baseline, although L-BFGS is more sensitive to initialization and its per-iteration computational cost can be more expensive. 
Learning-rate examinations with Adam confirm that large constant rates may plateau early (around the level of $10^{-2}$), while too-small rates lead to slow convergence. 
A moderate fixed value of learning rate, or a gradually decaying schedule appear to balance speed and stability.

\begin{figure}[ht]
    \centering
    \begin{subfigure}[t]{0.32\textwidth}
        \centering
        \includegraphics[width=\linewidth]{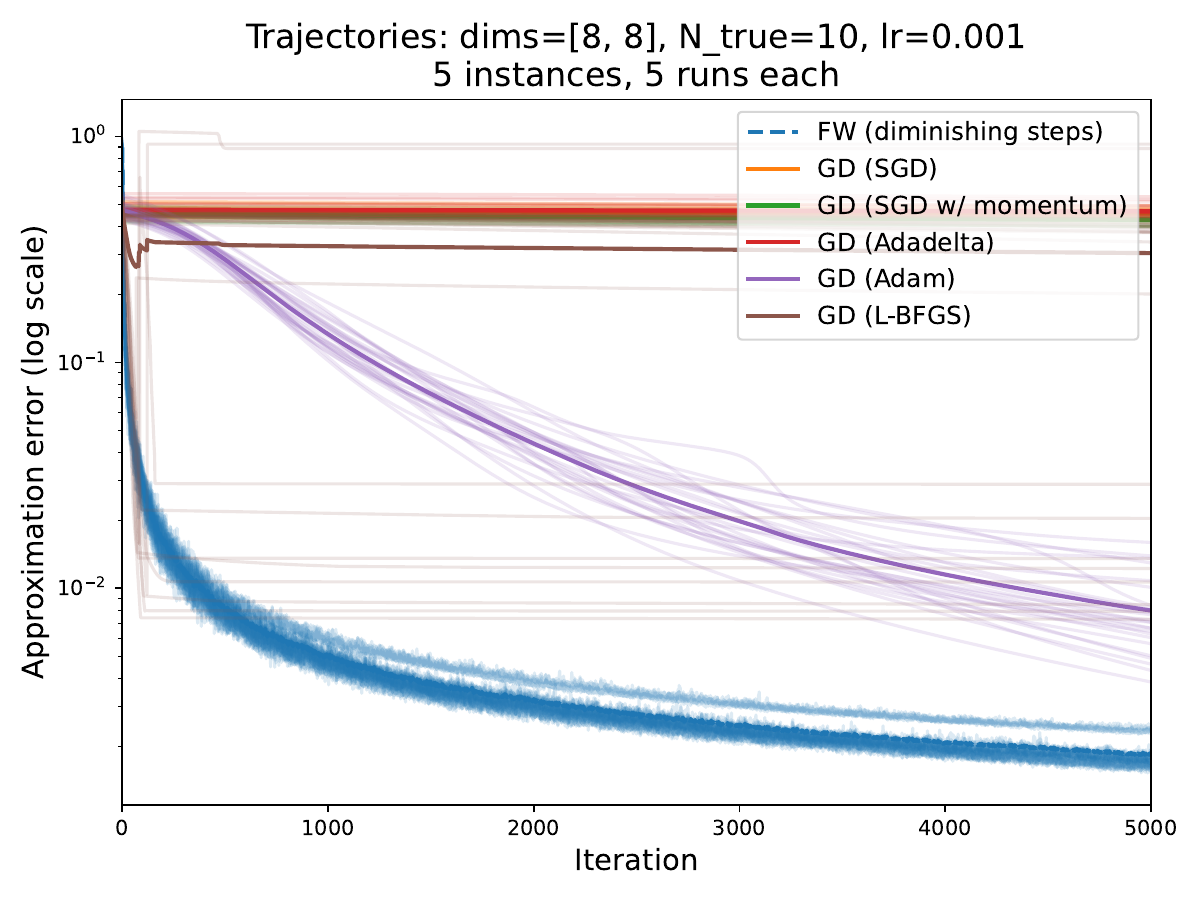}
        \caption{\footnotesize 
        Various optimizers with a modest learning rate (0.001, default). GD methods are too slow compared to the FW.
        }
        \label{fig:GD_optimizers.1}
    \end{subfigure}
    \hfill
    \begin{subfigure}[t]{0.32\textwidth}
        \centering
        \includegraphics[width=\linewidth]{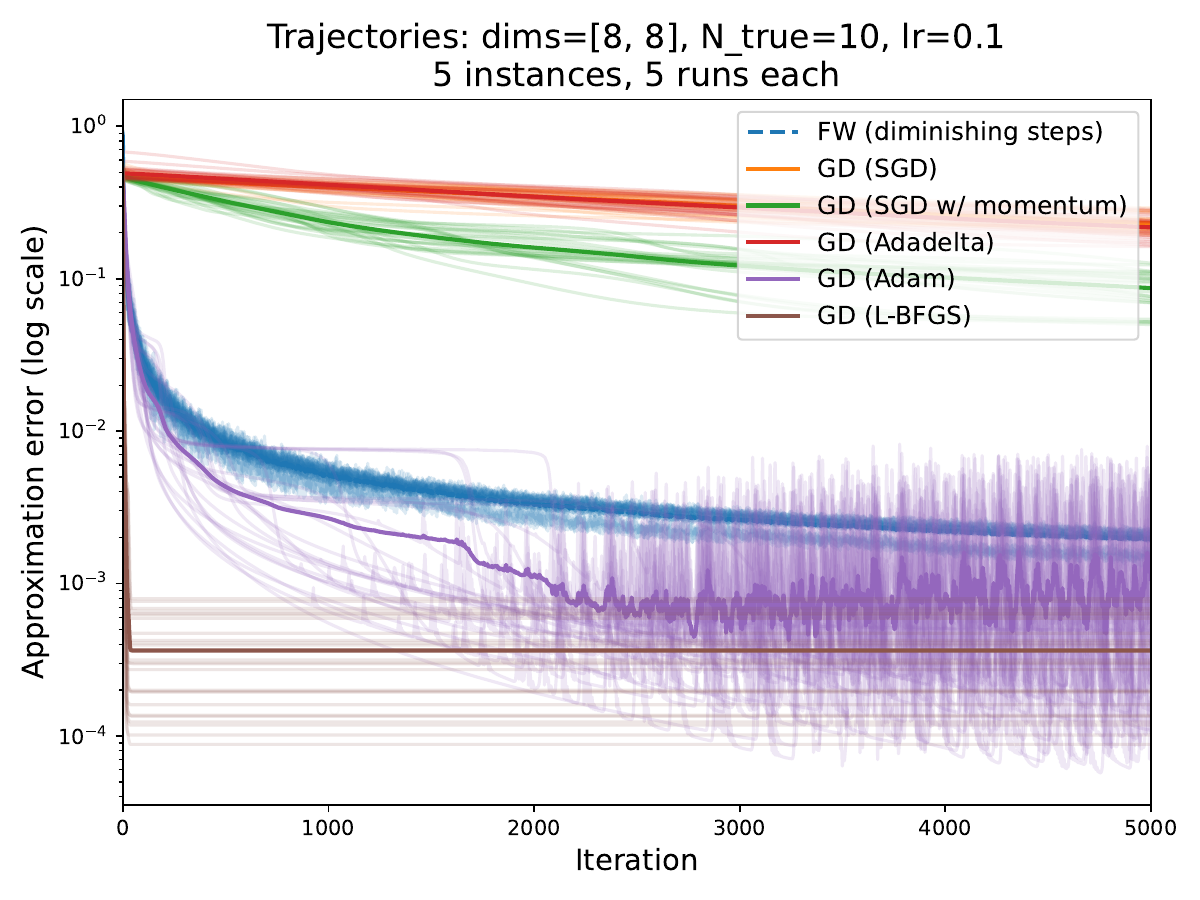}
        \caption{\footnotesize 
        Various optimizers with a more aggressive learning rate (0.1); Adam and L-BFGS lead but risk stalling.
        }
        \label{fig:GD_optimizers.2}
    \end{subfigure}
    \hfill
    \begin{subfigure}[t]{0.32\textwidth}
        \centering
        \includegraphics[width=\linewidth]{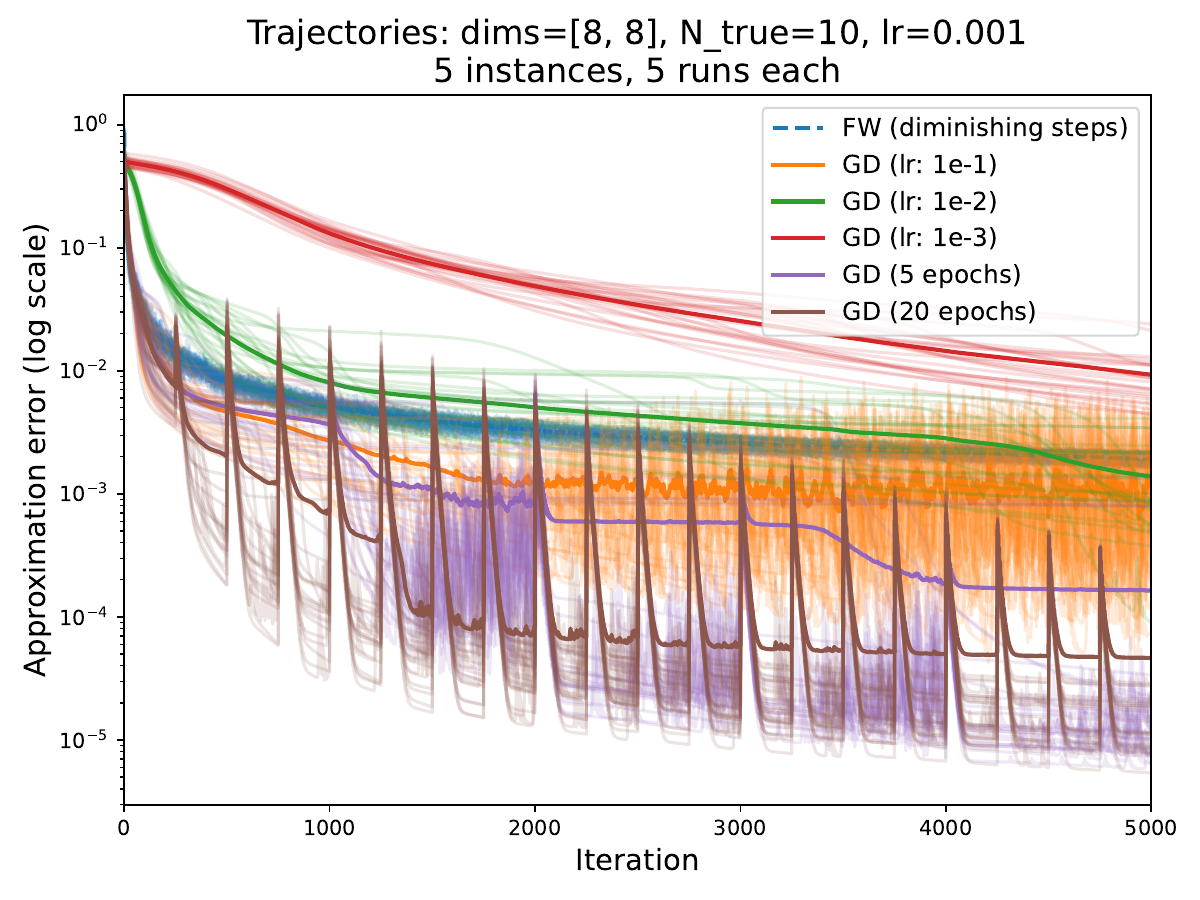}
        \caption{\footnotesize 
        Various learning rate schedules for Adam. 
        For decaying schedules, we gradually reduce LR from $10^{-1}$ to $10^{-3}$ each epoch geometrically.
        A gradually decaying LR with 20 epochs (in brown) is often best overall. 
        }
        \label{fig:GD_learning_rates}
    \end{subfigure}

    \caption{\small
    \textbf{Comparison of GD optimizer and stepsizes:} Adam typically outperforms simpler methods but may plateau if learning rate is too large. L-BFGS can excel yet is computationally heavier. 
    }
    \label{fig:GD_analysis}
\end{figure}

\begin{table}[t]
    \caption{Mean ($\pm 1$ standard deviation) time, measured in seconds, consumed per iteration (first $100$ iterations) of $10$ independent runs over $10$ random separable covariance instances.}
    \label{tab:gradient_descent_time}
    \centering
    \begin{adjustbox}{max width=\textwidth}
    \begin{tabular}{l c c c c}
        \toprule
        Local dimension $d$        &   FW (Diminishing)         &   GD (SGD)        &   GD (Adam)       &   GD (L-BFGS) \\
        \midrule
        2       &   $8.4431 \times 10^{-4}$ ($\pm 3.4301 \times 10^{-4}$)       &    $3.5084 \times 10^{-3}$ ($\pm 4.0567 \times 10^{-3}$)   &   $4.3180 \times 10^{-3}$ ($\pm 2.9685 \times 10^{-4}$)  &   $7.3818 \times 10^{-3}$ ($\pm 1.4514 \times 10^{-2}$) \\
        4       &   $1.2893 \times 10^{-3}$ ($\pm 4.2442 \times 10^{-4}$)       &    $3.5381 \times 10^{-3}$ ($\pm 4.1040 \times 10^{-3}$)   &   $4.3346 \times 10^{-3}$ ($\pm 2.8784 \times 10^{-4}$)  &   $2.8998 \times 10^{-2}$ ($\pm 3.5452 \times 10^{-2}$) \\
        8       &   $1.7519 \times 10^{-3}$ ($\pm 6.0140 \times 10^{-4}$)       &    $3.9294 \times 10^{-3}$ ($\pm 4.2241 \times 10^{-3}$)   &   $4.7363 \times 10^{-3}$ ($\pm 3.0675 \times 10^{-4}$)  &   $2.6015 \times 10^{-2}$ ($\pm 3.5911 \times 10^{-2}$) \\
        16      &   $1.0302 \times 10^{-2}$ ($\pm 5.1913 \times 10^{-3}$)       &    $3.3981 \times 10^{-2}$ ($\pm 1.6551 \times 10^{-2}$)   &   $3.1619 \times 10^{-2}$ ($\pm 6.5792 \times 10^{-3}$)  &   $1.4338 \times 10^{-1}$ ($\pm 2.1434 \times 10^{-1}$) \\
        32      &   $1.8255 \times 10^{-1}$ ($\pm 8.4548 \times 10^{-2}$)       &    $1.1819 \times 10^{-1}$ ($\pm 2.8048 \times 10^{-2}$)   &   $1.1282 \times 10^{-1}$ ($\pm 1.5549 \times 10^{-2}$)  &   $7.0365 \times 10^{-1}$ ($\pm 9.7290 \times 10^{-1}$) \\
        \bottomrule
    \end{tabular}
    \end{adjustbox}
\end{table}

\subsubsection{Deep parameterization width and depth}
Lastly, we examine the impact of MLP architecture in DP-GD to see if deep-neural-network parameterization could bring any benefits compared to the shallow gradient descent method. 
We fix the same $\Napprox$ as the shallow GD baseline, and vary the width ($\{25, 50, 100, 200, 400\}$) or the number (depth 2 or 3) of hidden internal layers.  
We continue to use Adam as our optimizer, and set learning rates as a fixed values, e.g., 0.1, 0.01 or 0.001. 

Figures~\ref{fig:DP_GD_width} and \ref{fig:DP_GD_depth} show that:
\begin{itemize}
    \item 
    Larger widths can converge faster if the learning rate is tuned sufficiently small (e.g., $10^{-2}$ or $10^{-3}$), but large learning rates often destabilize training.

    \item 
    Deeper networks (e.g. of depth 3, with 2 hidden layers) can outperform 2-layer nets, but only when step sizes are appropriately small.
\end{itemize}

\begin{figure}[ht]
    \centering
    \begin{subfigure}[t]{0.32\textwidth}
        \centering
        \includegraphics[width=\linewidth]{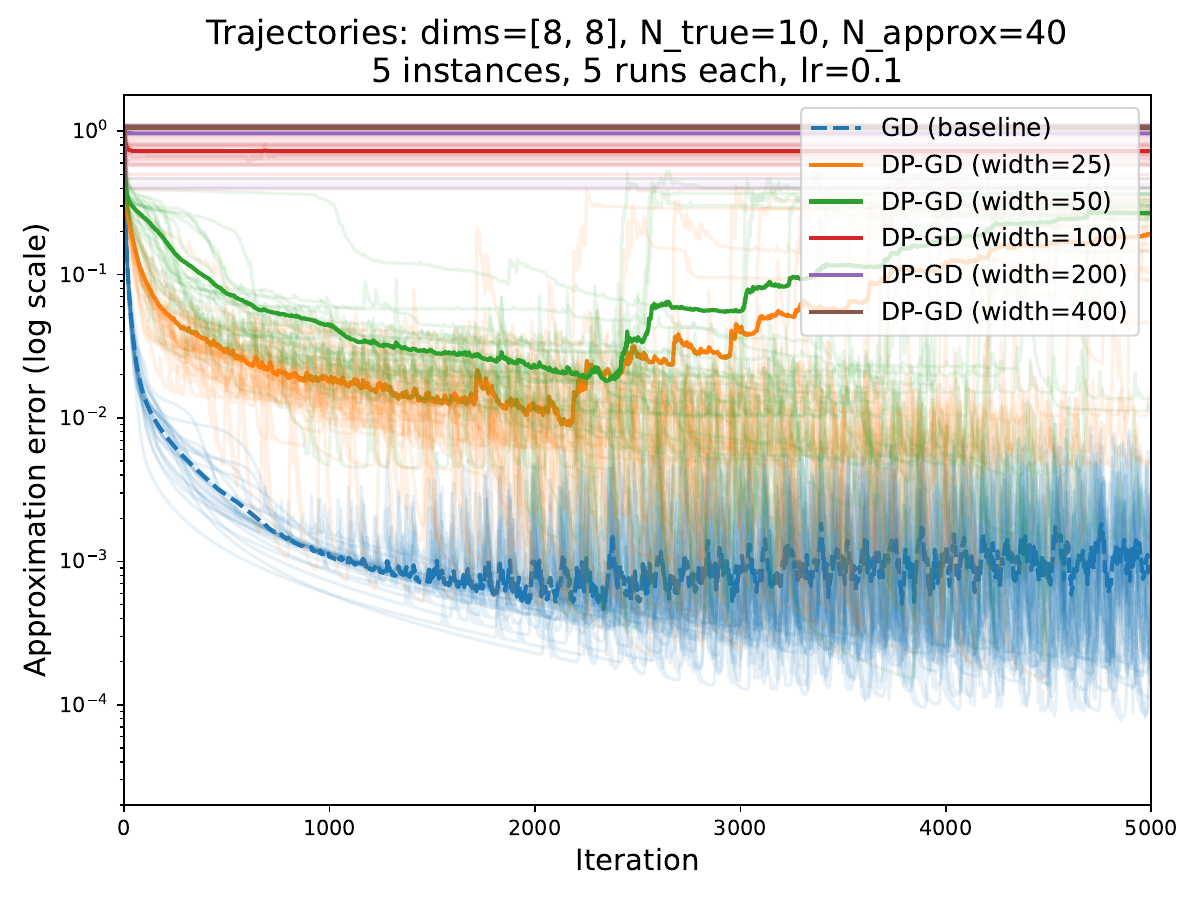}
        \caption{\footnotesize 
        $d=8$, $\Napprox=40$, $\eta=0.1$.
        }
        \label{fig:DPGD_width.1}
    \end{subfigure}
    \hfill
    \begin{subfigure}[t]{0.32\textwidth}
        \centering
        \includegraphics[width=\linewidth]{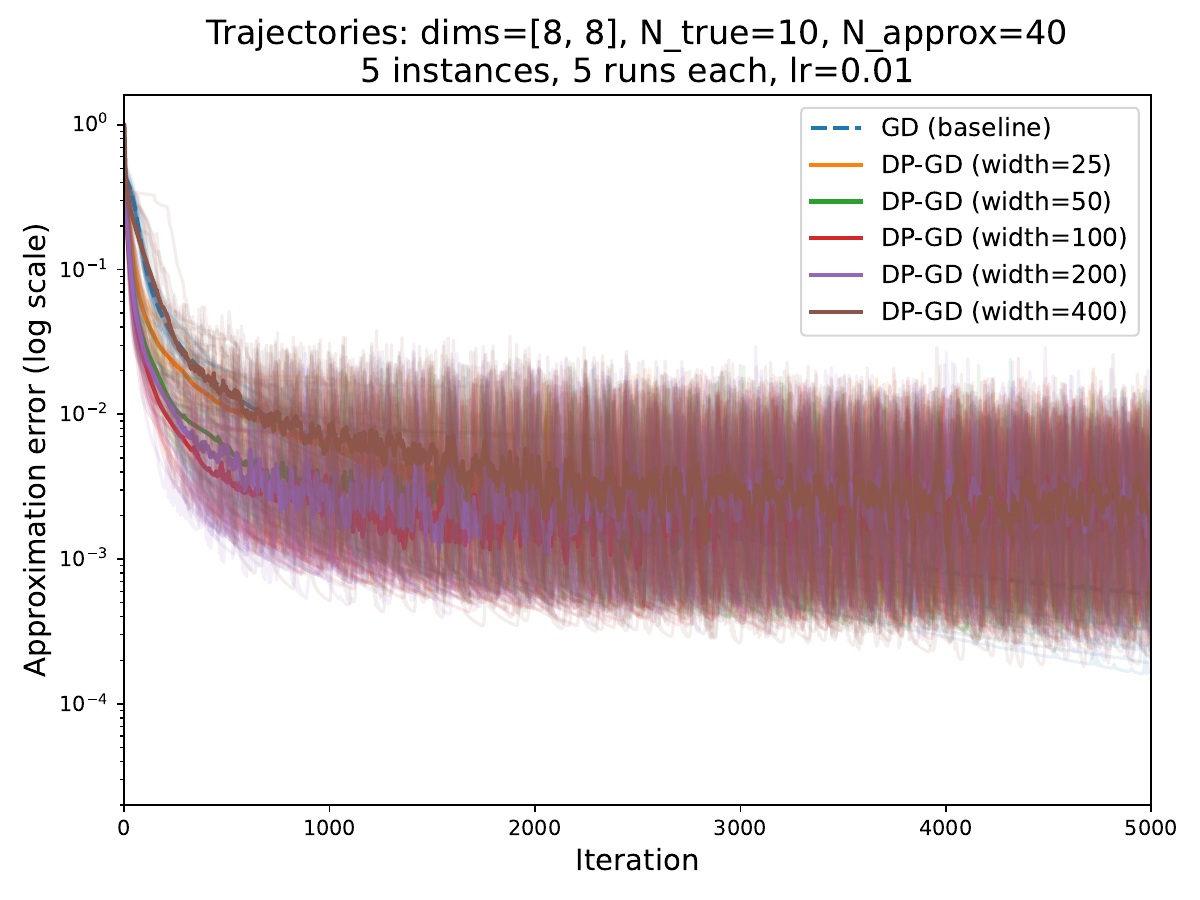}
        \caption{\footnotesize 
        $d=8$, $\Napprox=40$, $\eta=0.01$.
        }
        \label{fig:DPGD_width.2}
    \end{subfigure}
    \hfill
    \begin{subfigure}[t]{0.32\textwidth}
        \centering
        \includegraphics[width=\linewidth]{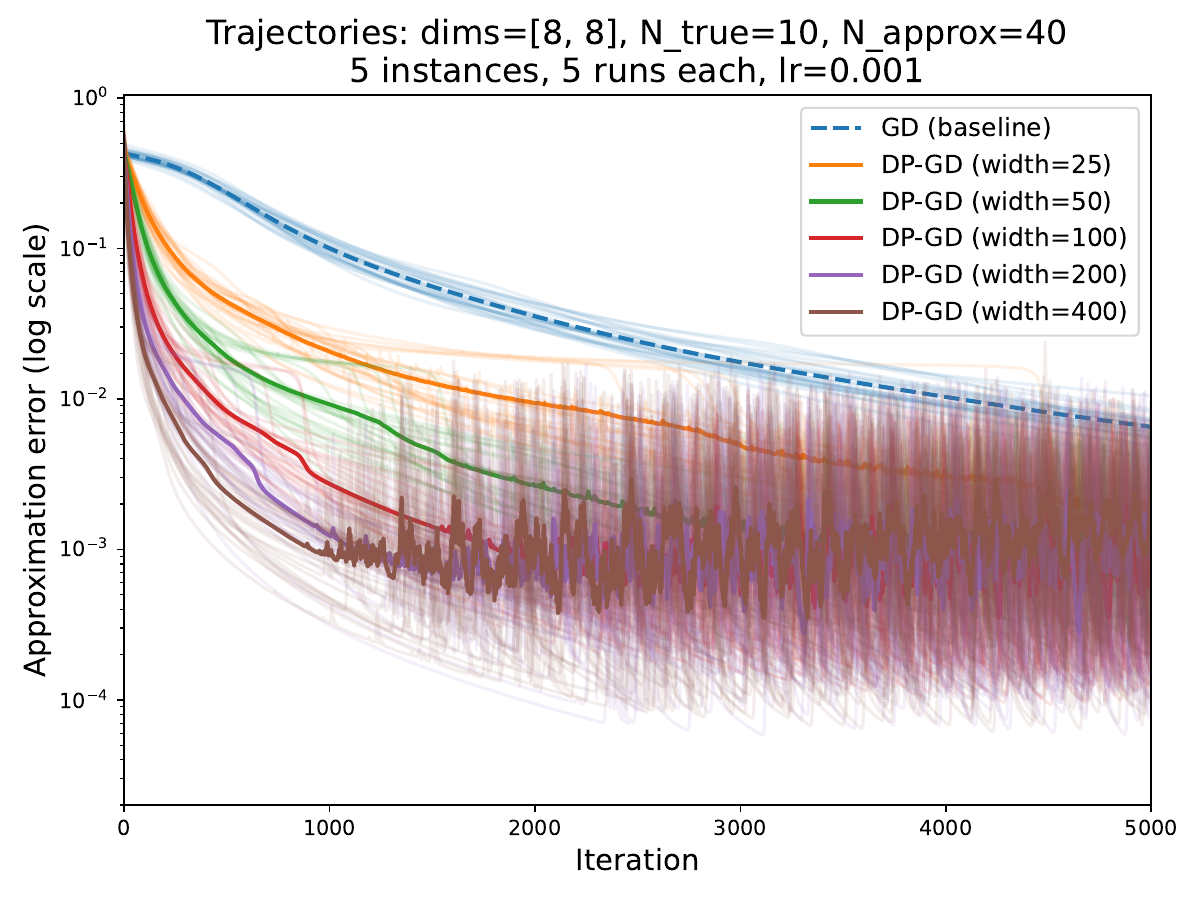}
        \caption{\footnotesize 
        $d=8$, $\Napprox=40$, $\eta=0.001$.
        }
        \label{fig:DPGD_width.3}
    \end{subfigure}
    \\
    \begin{subfigure}[t]{0.32\textwidth}
        \centering
        \includegraphics[width=\linewidth]{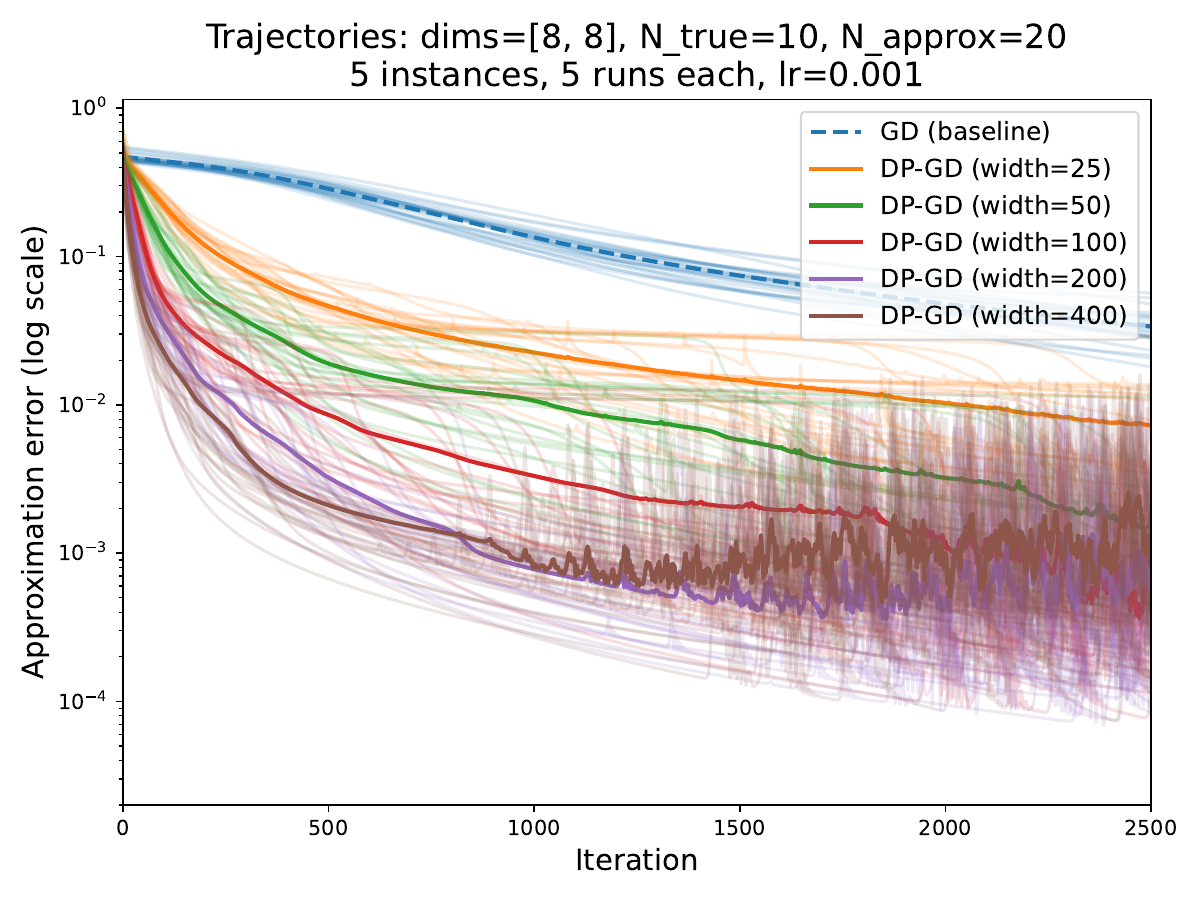}
        \caption{\footnotesize 
        $d=8$, $\Napprox=20$, $\eta=0.001$.
        }
        \label{fig:DPGD_width.4}
    \end{subfigure}
    \hfill
    \begin{subfigure}[t]{0.32\textwidth}
        \centering
        \includegraphics[width=\linewidth]{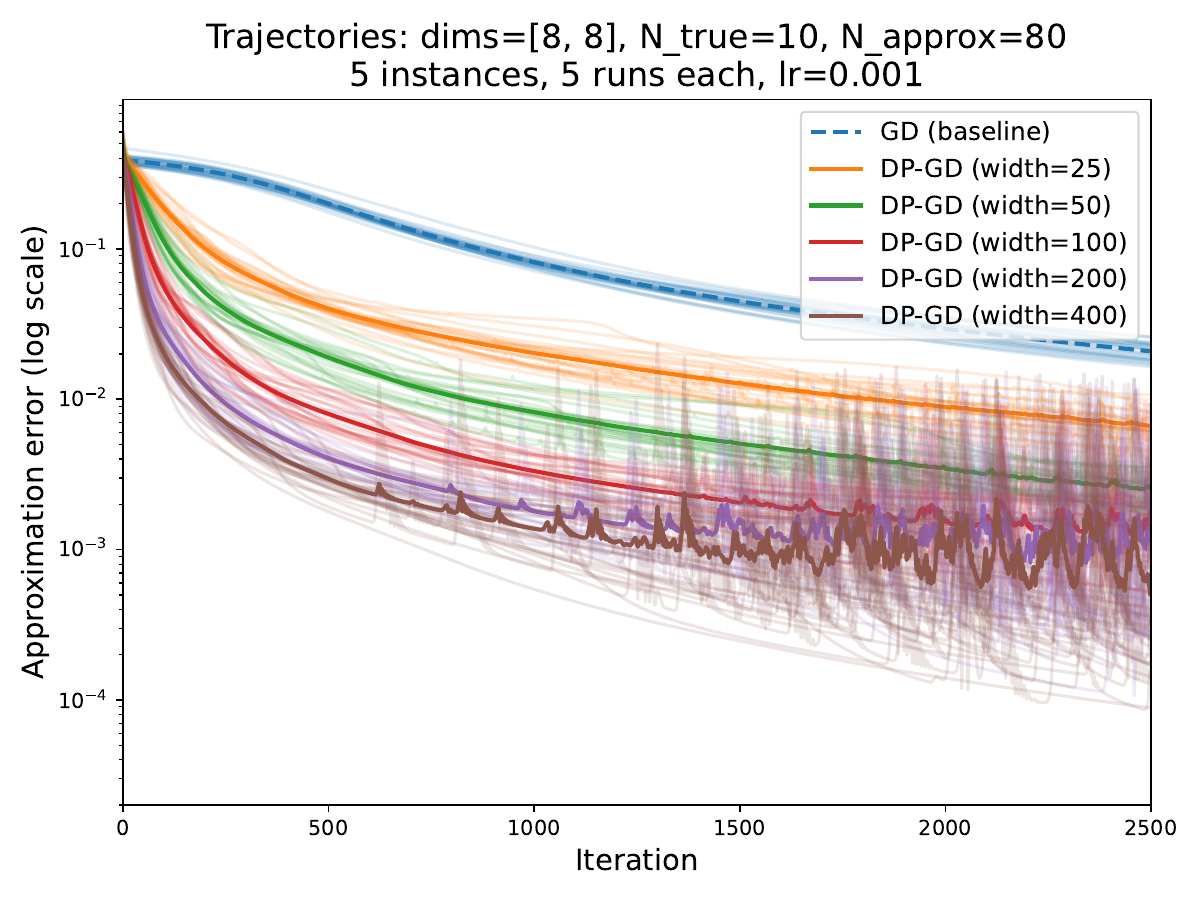}
        \caption{\footnotesize 
        $d=8$, $\Napprox=80$, $\eta=0.001$.
        }
        \label{fig:DPGD_width.5}
    \end{subfigure}
    \hfill
    \begin{subfigure}[t]{0.32\textwidth}
        \centering
        \includegraphics[width=\linewidth]{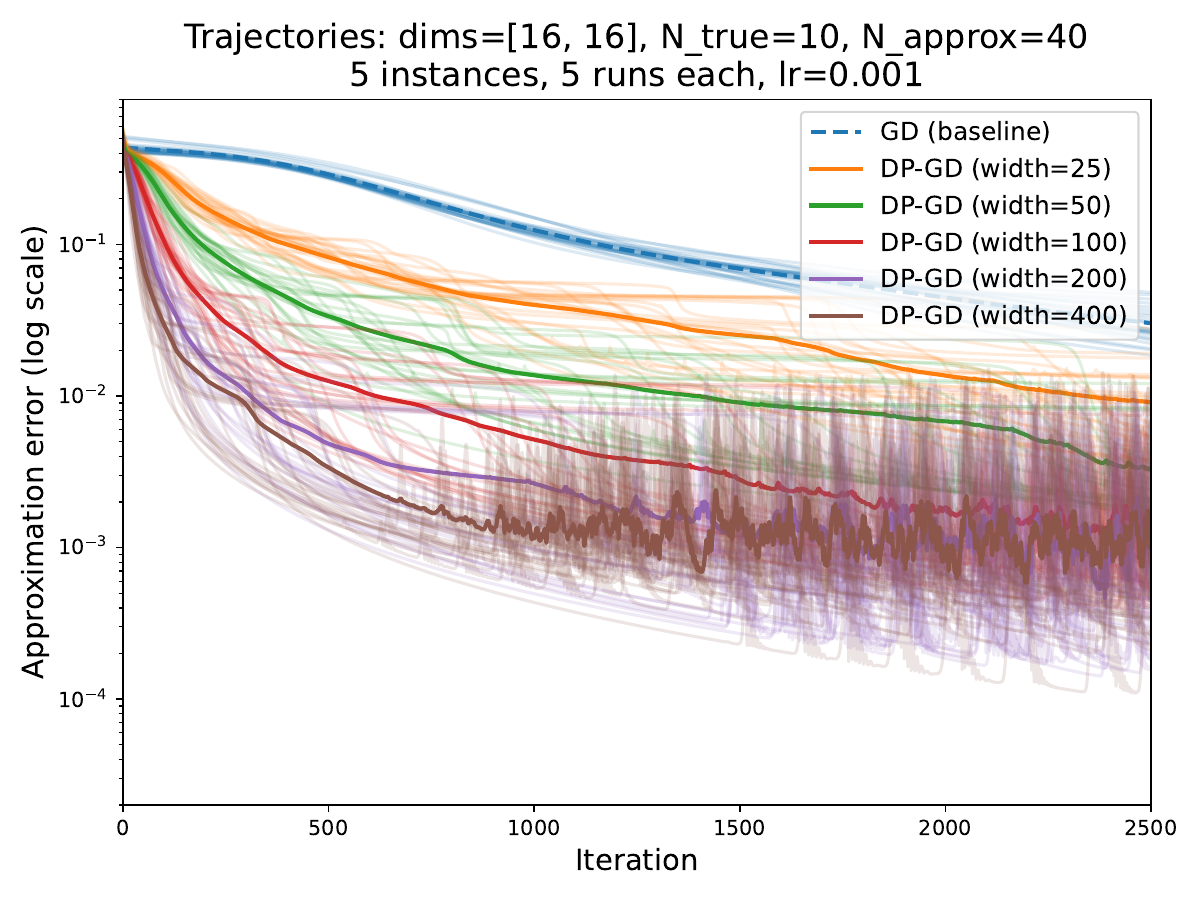}
        \caption{\footnotesize 
        $d=16$, $\Napprox=40$, $\eta=0.001$.
        }
        \label{fig:DPGD_width.6}
    \end{subfigure}

    \caption{\small
    When learning rate is large, DP-GD is unstable.  For a small learning rate, DP-GD with wider widths seem to make fastest progress early, but all plateau around $10^{-3}$. 
    }
    \label{fig:DP_GD_width}
\end{figure}

\begin{figure}[ht!]
    \centering
    \begin{subfigure}[t]{0.32\textwidth}
        \centering
        \includegraphics[width=\linewidth]{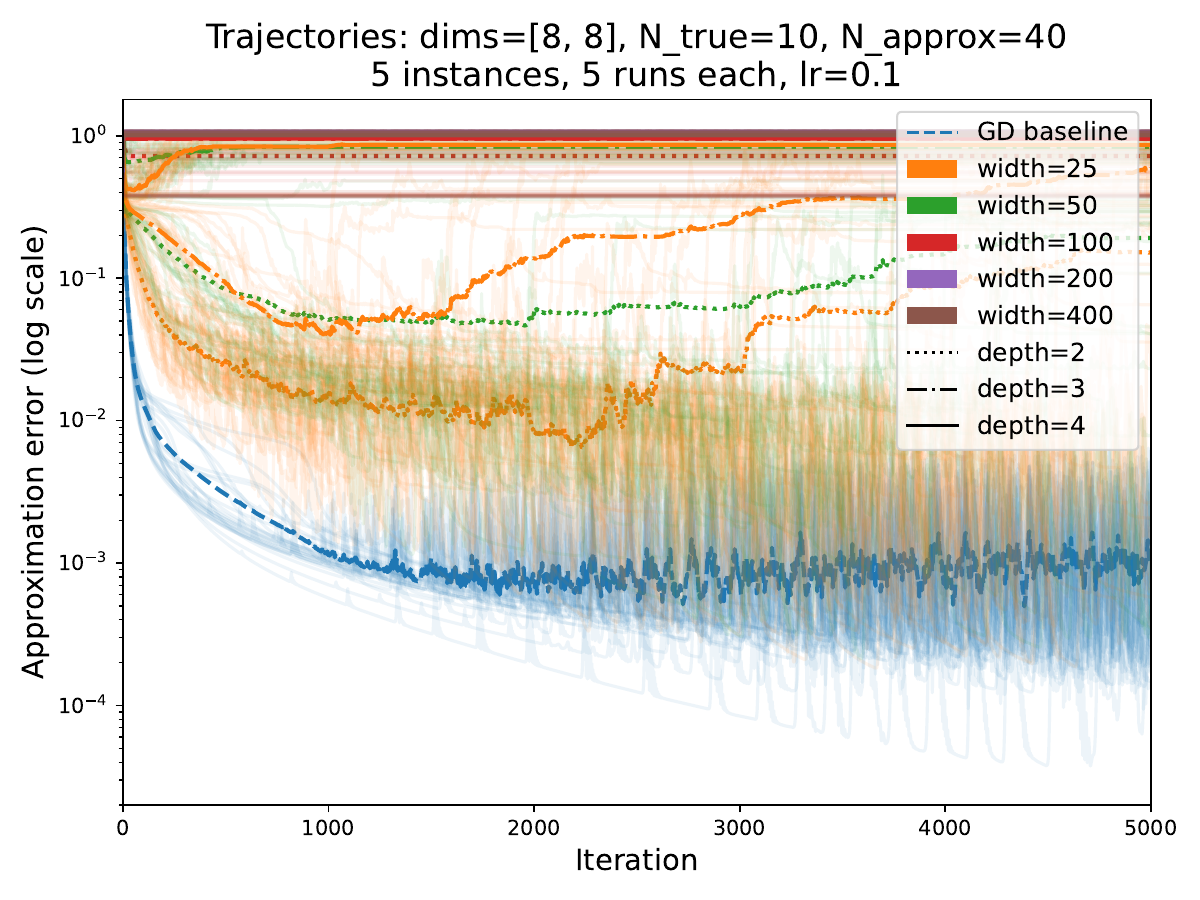}
        \caption{\footnotesize 
        $d=8$, $\Napprox=40$, $\eta=0.1$.
        }
        \label{fig:DPGD_depth.1}
    \end{subfigure}
    \hfill
    \begin{subfigure}[t]{0.32\textwidth}
        \centering
        \includegraphics[width=\linewidth]{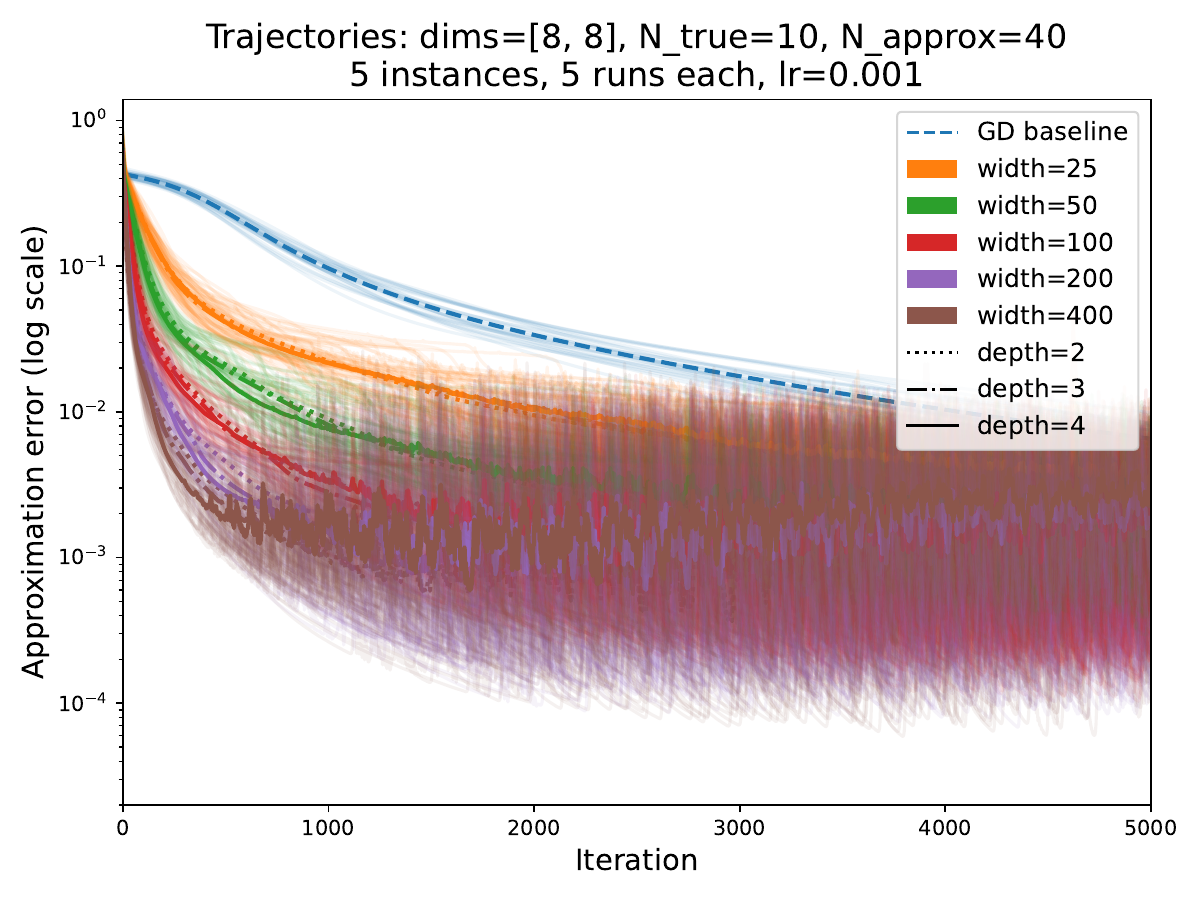}
        \caption{\footnotesize 
        $d=8$, $\Napprox=40$, $\eta=0.001$.
        }
        \label{fig:DPGD_depth.2}
    \end{subfigure}
    \hfill
    \begin{subfigure}[t]{0.32\textwidth}
        \centering
        \includegraphics[width=\linewidth]{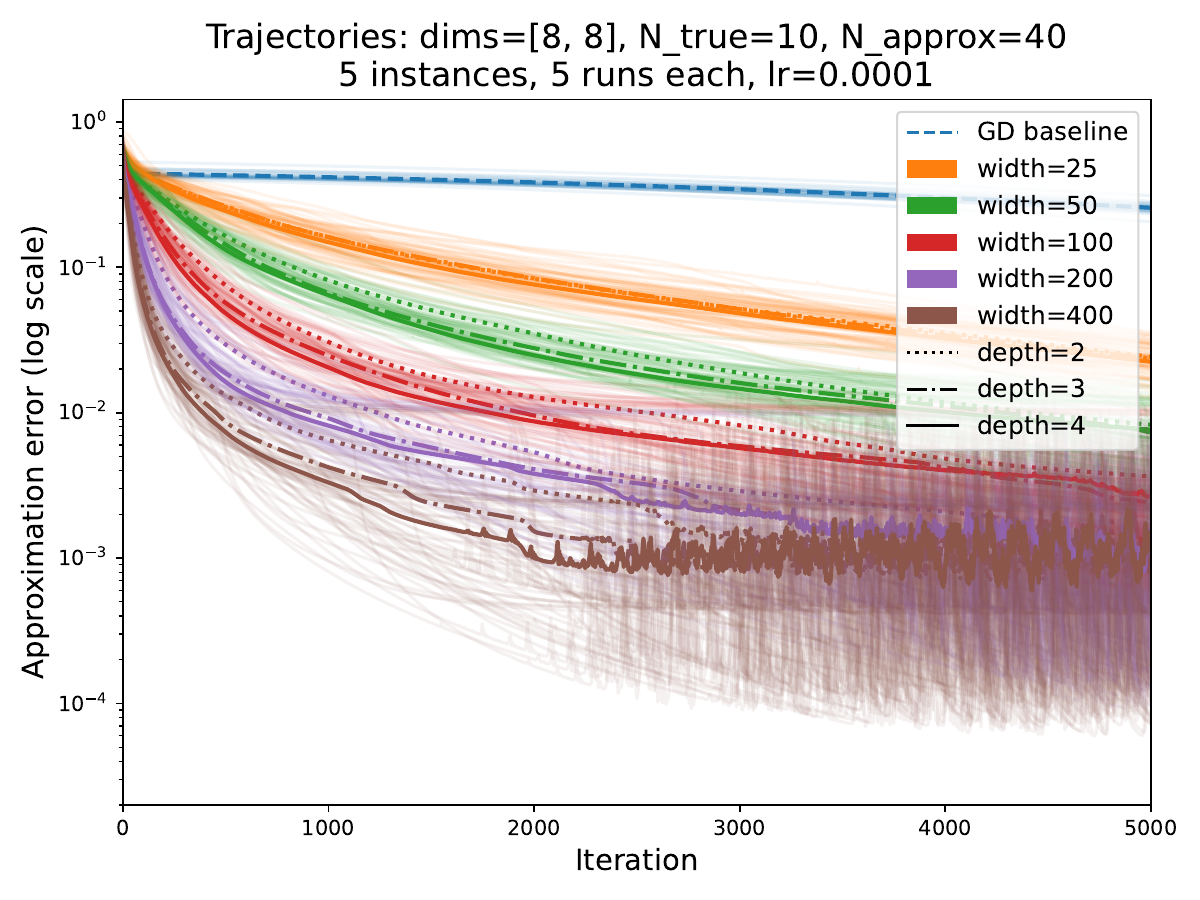}
        \caption{\footnotesize 
        $d=8$, $\Napprox=40$, $\eta=0.0001$.
        }
        \label{fig:DPGD_depth.3}
    \end{subfigure}
    \\
    \begin{subfigure}[t]{0.32\textwidth}
        \centering
        \includegraphics[width=\linewidth]{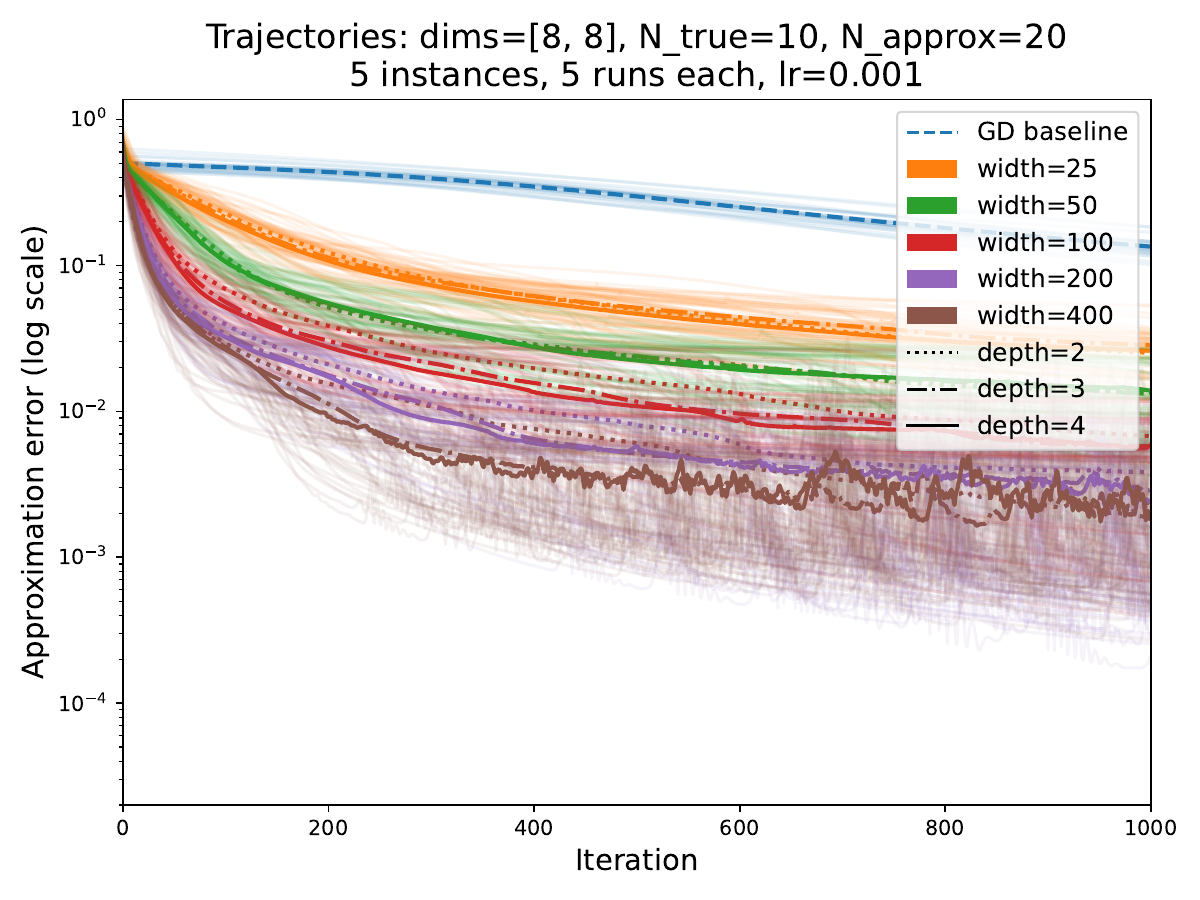}
        \caption{\footnotesize 
        $d=8$, $\Napprox=20$, $\eta=0.001$.
        }
        \label{fig:DPGD_depth.4}
    \end{subfigure}
    \hfill
    \begin{subfigure}[t]{0.32\textwidth}
        \centering
        \includegraphics[width=\linewidth]{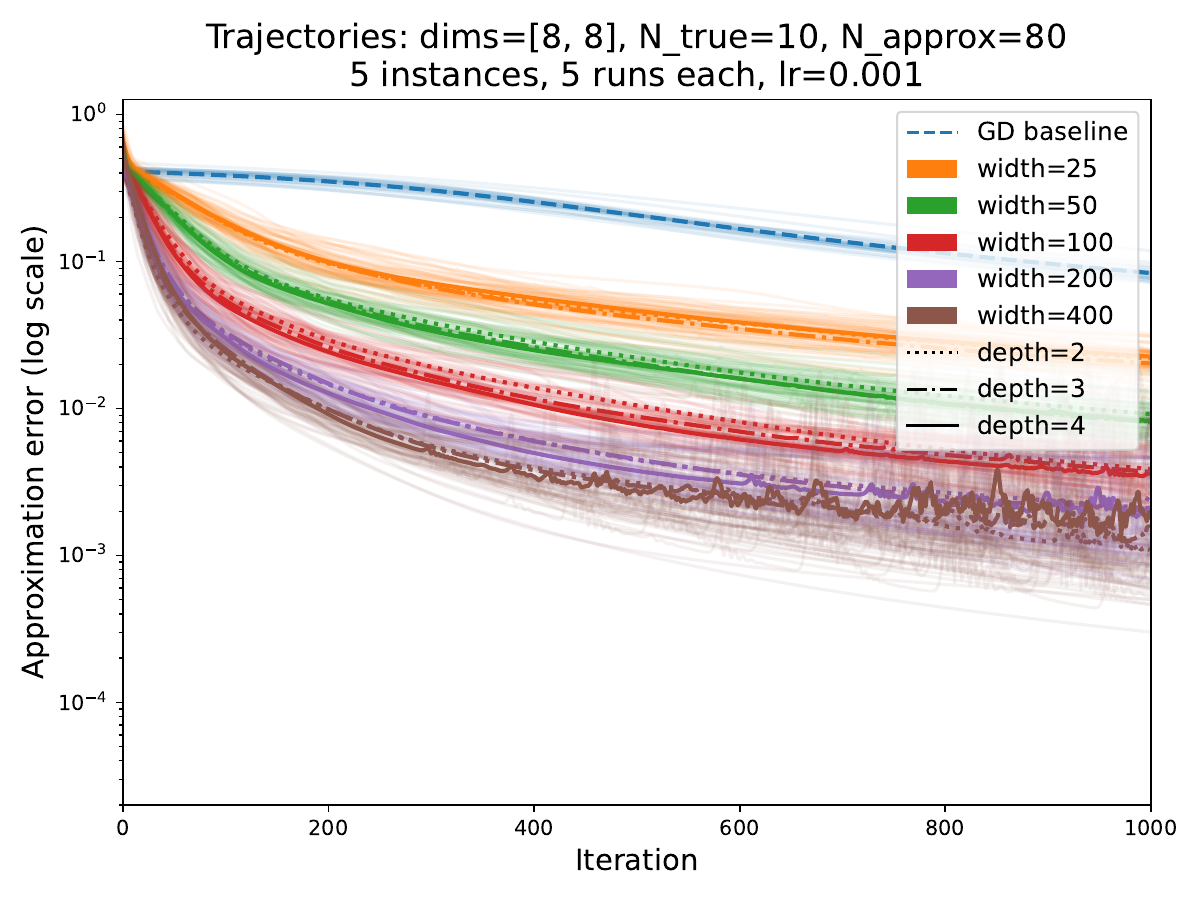}
        \caption{\footnotesize 
        $d=8$, $\Napprox=80$, $\eta=0.001$.
        }
        \label{fig:DPGD_depth.5}
    \end{subfigure}
    \hfill
    \begin{subfigure}[t]{0.32\textwidth}
        \centering
        \includegraphics[width=\linewidth]{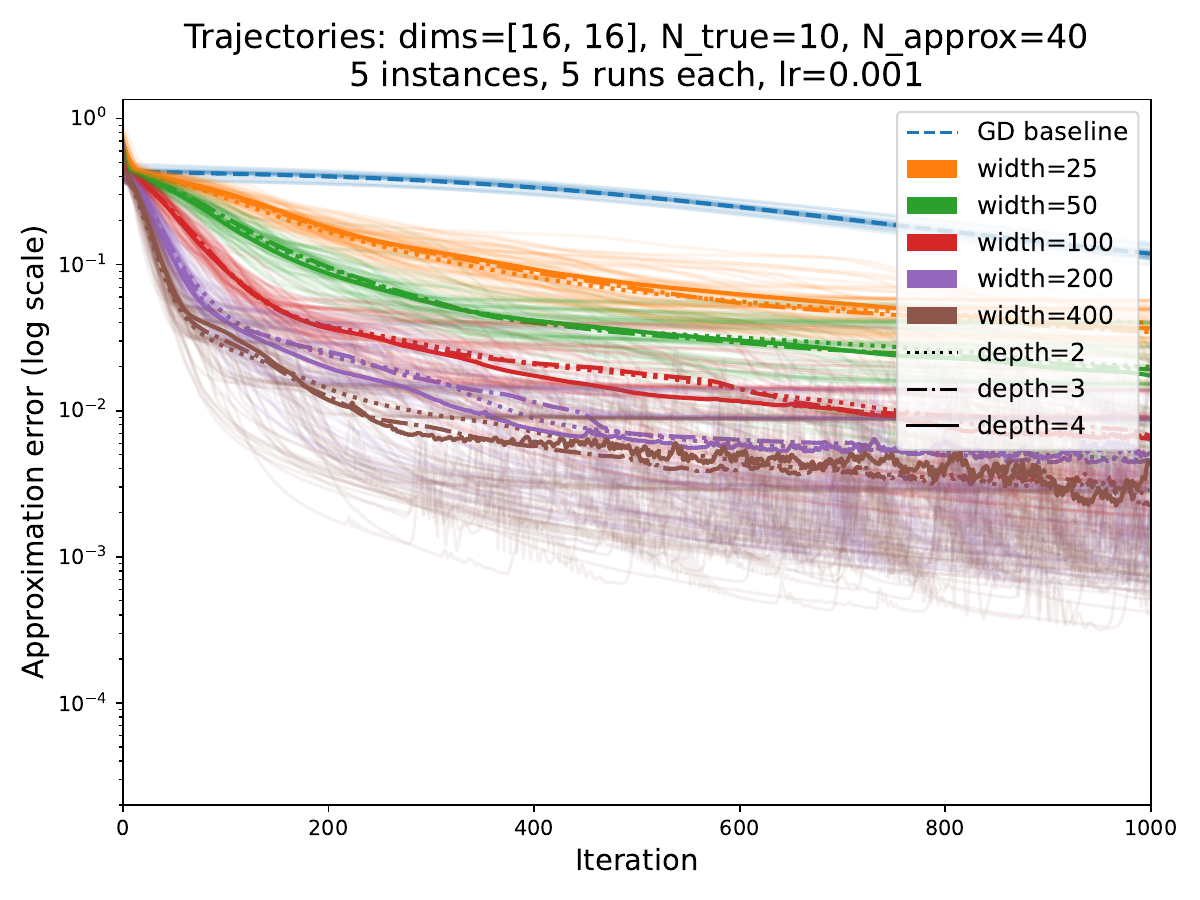}
        \caption{\footnotesize 
        $d=16$, $\Napprox=40$, $\eta=0.001$.
        }
        \label{fig:DPGD_depth.6}
    \end{subfigure}

    \caption{\small
    When learning rate is large, DP-GD is unstable.  For a small learning rate, DP-GD with wider widths seem to make fastest progress early, but all plateau around $10^{-3}$. 
    }
    \label{fig:DP_GD_depth}
\end{figure}

Interestingly, when the learning rate is small (e.g., 0.001), we observe a seemingly multi-phase trend in training: the trajectory plateaus near a certain level (approximately $6 \times 10^{-3}$) and then continues to converge down to $0$.

\subsection{Additional results on approximating quantum states}\label{sec:additional_quantum_exp}

This section provides extended experimental data and further discussions related to the quantum-state approximation experiments in Section \ref{sec:quantum_states}.
Specifically, we include:
\begin{itemize}
    \item 
    Experiments with larger mixture budgets to demonstrate that GD and DP-GD can succeed in higher dimensions (e.g.\ $d=8,16$) if $\Napprox$ is increased from 50 to 100 or 200, 
    \item 
    Detailed examples (especially at $d=2$) and discussions to illustrate how restricting rank-1 atoms to be real can yield systematically suboptimal approximations, despite the input states being real symmetric.
\end{itemize}

\subsubsection{Additional plots and extended experiments}
Recall from Section \ref{sec:quantum_states} that in Figure~\ref{fig:quantum_separability} all three methods (FW, GD, DP-GD) with \emph{complex} parameterization accurately reproduced the known separability thresholds and best-separable-approximation (BSA) distances in moderate dimensions ($d\in{2,3,4,5,6}$). 
However, at higher dimensions ($d=8,16$), GD and DP-GD often failed to reach the correct threshold, possibly due to a limited rank budget $\Napprox$ (set to $50$), whereas FW continued to recover known results.

Figure \ref{fig:extended} provides more detailed curves of the approximation error vs.\ $\alpha$ for dimensions up to $d=16$, this time varying the approximation budget $\Napprox$ from 50 to 100 and 200, respectively. 
We see that both GD and DP-GD become noticeably more accurate once allowed a larger representational capacity $\Napprox$. 
Also, observe that FW (with fixed stepsize rules with at most $1500$ iterations ) remains robust throughout all dimensions $d = 2,4,8,16$ , and at even larger dimension ($d=32$, see Figure \ref{fig:quantum_full_fw.complex}); we conjecture FW will continue to perform well at even larger dimensions, albeit with increased iteration costs.

\begin{figure}[ht!]
    \centering
    \begin{subfigure}[t]{0.48\textwidth}
        \centering
        \includegraphics[width=\linewidth]{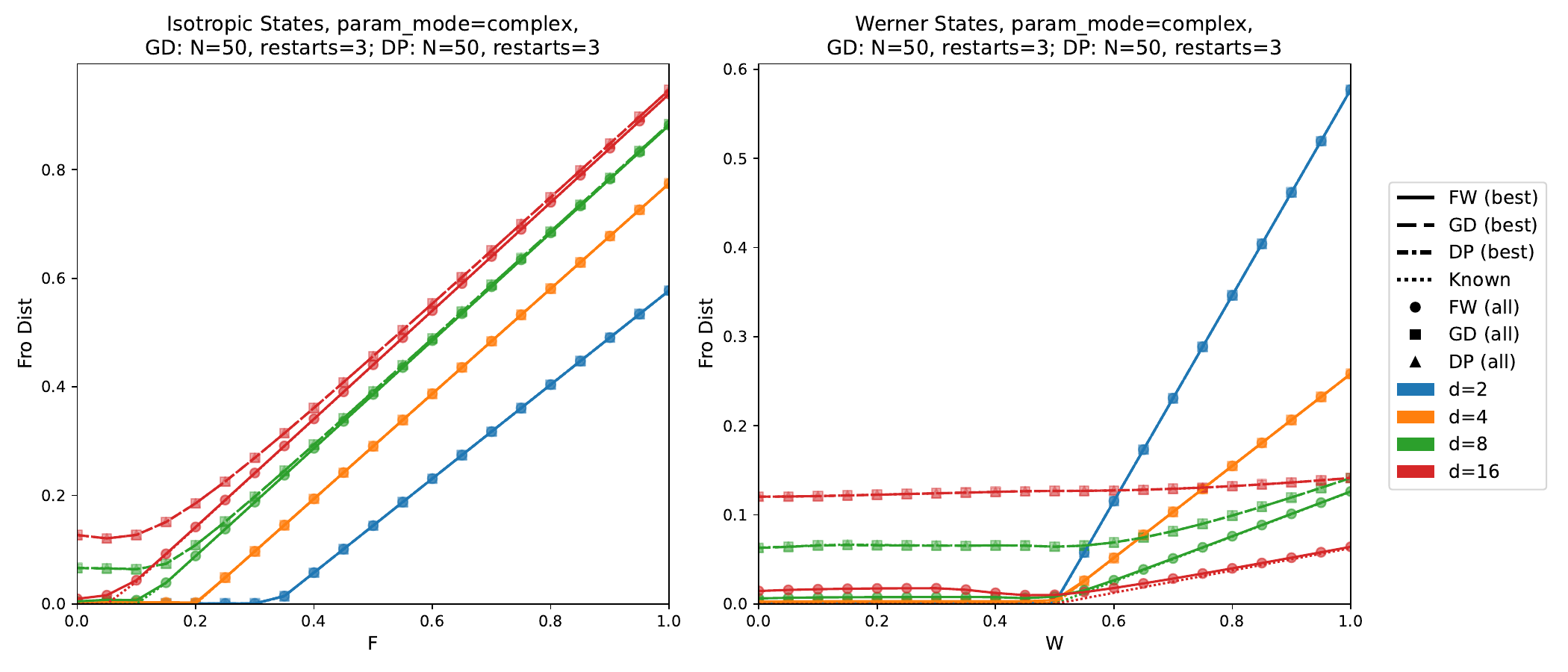}
        \caption{\footnotesize 
            \textbf{Budget $\Napprox =50$:} 
            FW accurately matches the known BSA results for $d \in \{2, 4, 8, 16\}$, whereas GD and DP-GD are off at $d=8,16$.
        }
        \label{fig:extended.1}
    \end{subfigure}
    \hfill
    \begin{subfigure}[t]{0.48\textwidth}
        \centering
        \includegraphics[width=\linewidth]{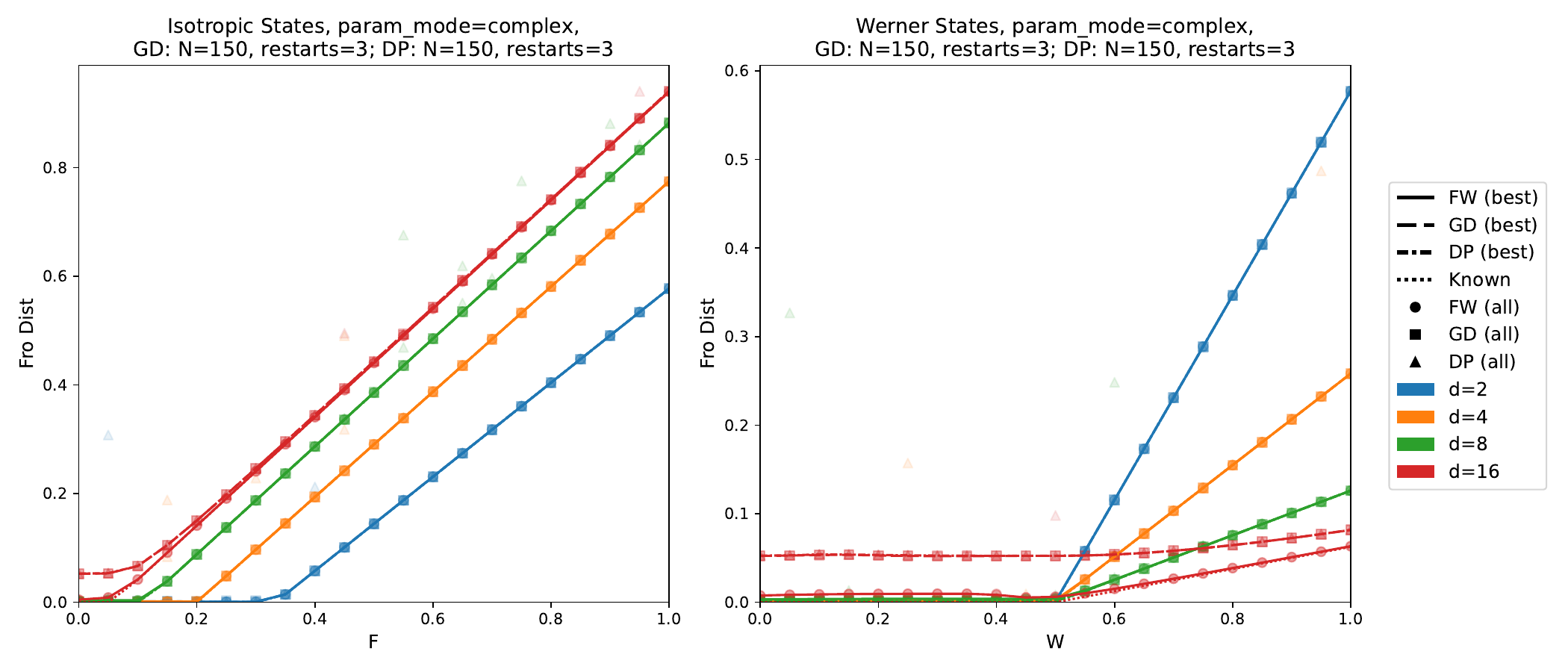}
        \caption{\footnotesize 
            \textbf{Budget $\Napprox =100$:} GD and DP-GD benefit from increased $\Napprox$ with improved accuracy at $d=8$ ($d=16$ still off).
        }
        \label{fig:extended.2}
    \end{subfigure}
    \\
    \begin{subfigure}[t]{0.48\textwidth}
        \centering
        \includegraphics[width=\linewidth]{figures/quantum/Fig_isotropic_werner_complex_Ngd150_steps250_Rgd3_Ndp150_Wdp200_Rdp3_full3.pdf}
        \caption{\footnotesize 
            \textbf{Budget $\Napprox =150$:} Better performance of GD/DP-GD, approaching FW’s accuracy, for both $d=8, 16$.
        }
        \label{fig:extended.3}
    \end{subfigure}
    \hfill
    \begin{subfigure}[t]{0.48\textwidth}
        \centering
        \includegraphics[width=\linewidth]{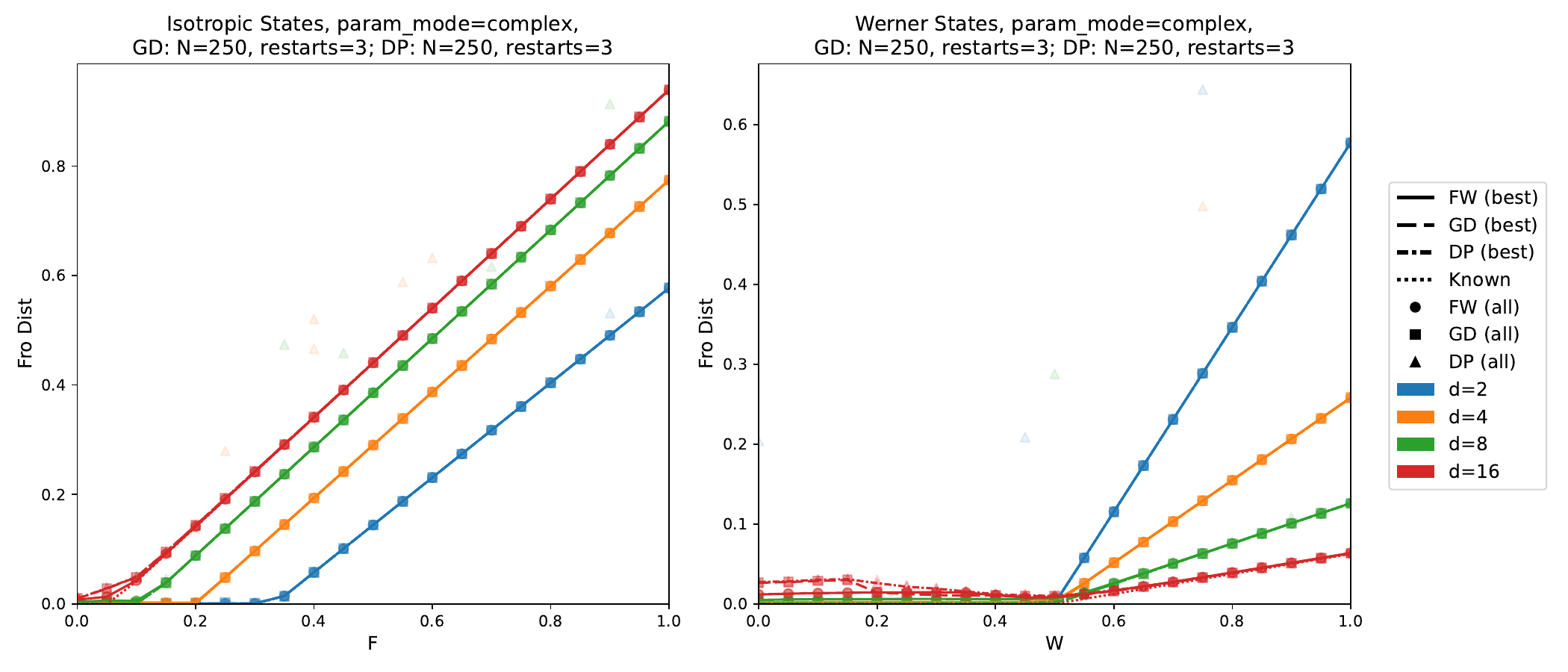}
        \caption{\footnotesize 
            \textbf{Budget $\Napprox =250$:} Even better performance of GD/DP-GD, approaching FW’s accuracy, for both $d=8, 16$.
        }
        \label{fig:extended.4}
    \end{subfigure}
    \caption{
        \textbf{Extended results on isotropic/Werner approximations:} Performance of GD and DP-GD improves as the rank budget $\Napprox$ increases from 50 to 100 150 or 250, especially for $d=8,16$.
        }
    \label{fig:extended}
\end{figure}

\subsubsection{Real vs.\ complex K-S approximation}\label{sec:real_vs_complex}
Quantum states are typically represented in a complex Hilbert space, so the isotropic and Werner states, $\bSigma^{\iso}_d(\alpha)$ and $\bSigma^{\Werner}_d(\alpha)$, appear as self-adjoint (Hermitian) operators on the product Hilbert space $\CC^d \otimes \CC^d$. 
Nevertheless, for each $\alpha \in [0,1]$, we can also write these operators as \emph{real symmetric} matrices. 
Thus, a natural question arises: \emph{Does enforcing rank-1 factors $\bv_a^{(k)}$ in \eqref{eqn:separable.2} to be purely real vectors (rather than complex) yield an equally good BSA?}

To clarify this question, observe that $\bSigma^{\iso}_d(\alpha) \in \cBsa_+(\CC^d \otimes \CC^d)$ is said to be Kronecker-separable (K-S) if $\bSigma^{\iso}_d(\alpha) \in \cBsa_+(\CC^d) \otimes \cBsa_+(\CC^d)$, when viewed as a complex operator (see Definition~\ref{defn:separable}). 
At the same time, $\bSigma^{\iso}_d$ is real symmetric, so one might wonder whether $\bSigma^{\iso}_d(\alpha) \in \cBsa_+(\RR^d) \otimes \cBsa_+(\RR^d)$---i.e., whether it admits a real K-S decomposition.

In fact, for isotropic states, we numerically find that this real K-S property holds \emph{only} when $\alpha = 0$, whereas the standard (complex) K-S property holds for all $\alpha \leq \frac{1}{d+1}$. 
Therefore, although $\bSigma^{\iso}_d(\alpha)$ is real symmetric, requiring a real-based K-S decomposition imposes strictly more stringent constraints than the complex-based one.
A similar discrepancy arises for the Werner states as well.

Empirically, for all $\alpha \in [0,1]$, our experiments show that restricting K-S expansions to real vectors generally produces a different---and typically worse---approximation than the known (complex) BSA. 
This gap persists even at moderate dimensions ($d \leq 6$), suggesting that real-only subspaces either lack the expressive power of complex expansions, or trap algorithms in suboptimal local minima.
Our additional (unreported) experiments indicate that limited expressive capacity is likely the more probable culprit. 
As illustrated in Figure~\ref{fig:quantum_complex_vs_real}, real-parameter expansions consistently fail to match the exact threshold or BSA distance achieved by complex-parameter methods.

\begin{figure}[ht]
    \centering
    \begin{subfigure}[t]{0.49\textwidth}
        \centering
        \includegraphics[width=\linewidth]{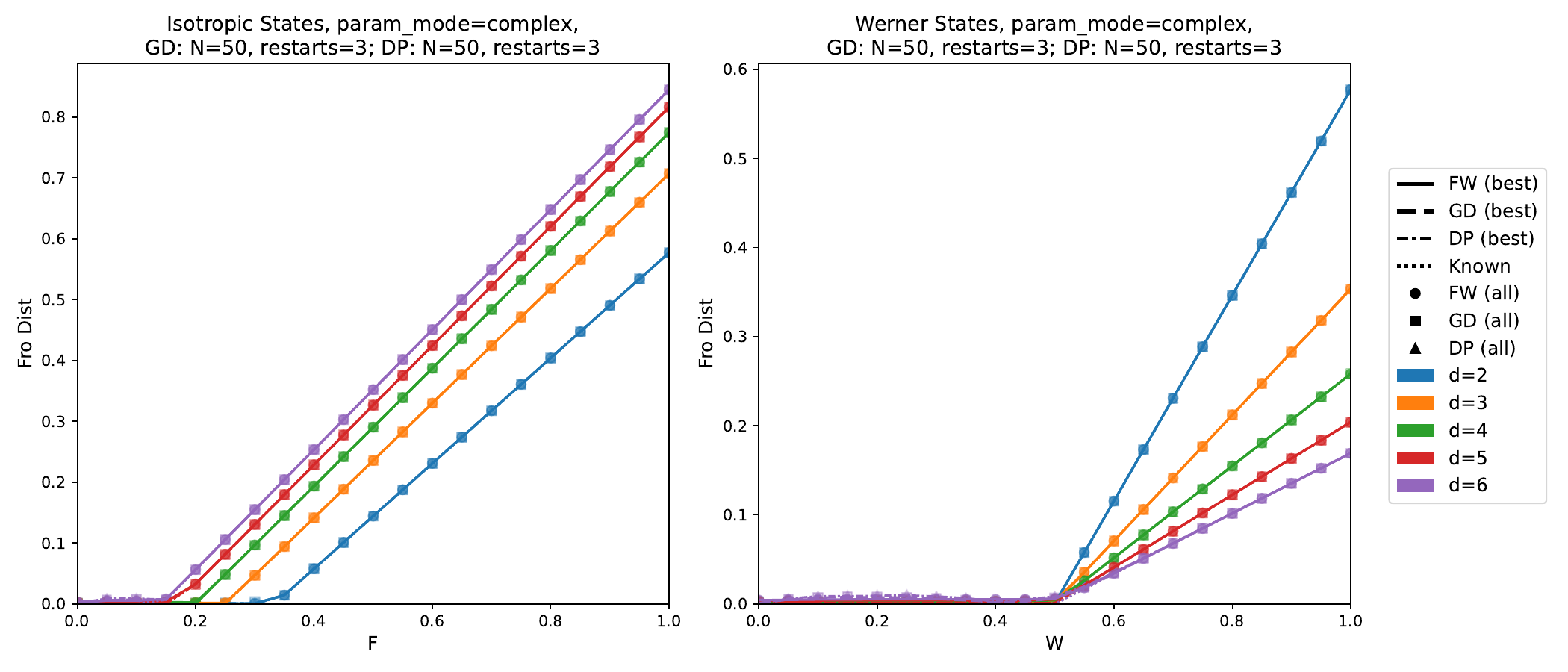}
        \caption{\footnotesize 
            \textbf{Complex param:} $d=2,3,4,5,6$. FW, GD, DP-GD closely match known BSAs.
        }
        \label{fig:quantum_small.complex}
    \end{subfigure}
    \hfill
    \begin{subfigure}[t]{0.49\textwidth}
        \centering
        \includegraphics[width=\linewidth]{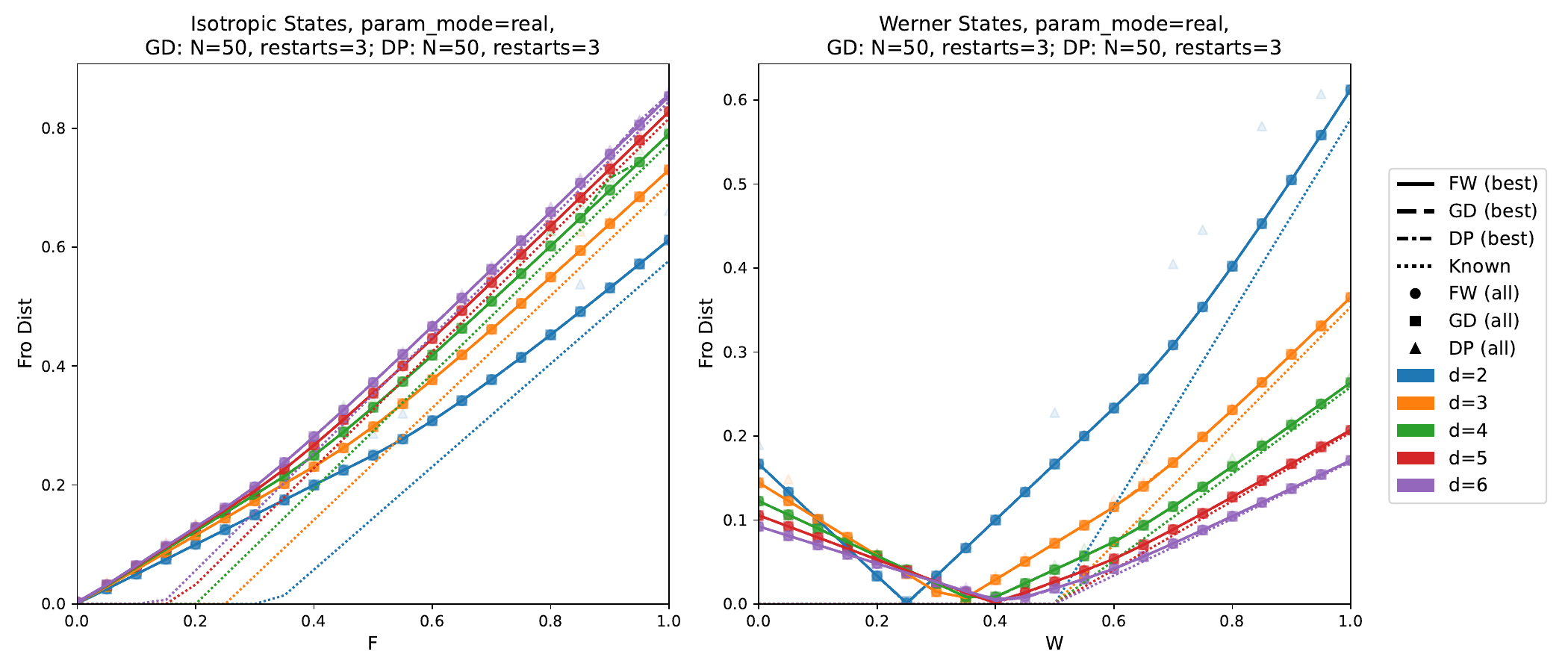}
        \caption{\footnotesize 
            \textbf{Real param:} $d=2,3,4,5,6$. All methods produce suboptimal solutions in most regimes.
        }
        \label{fig:quantum_small.real}
    \end{subfigure}
    \\
    \begin{subfigure}[t]{0.49\textwidth}
        \centering
        \includegraphics[width=\linewidth]{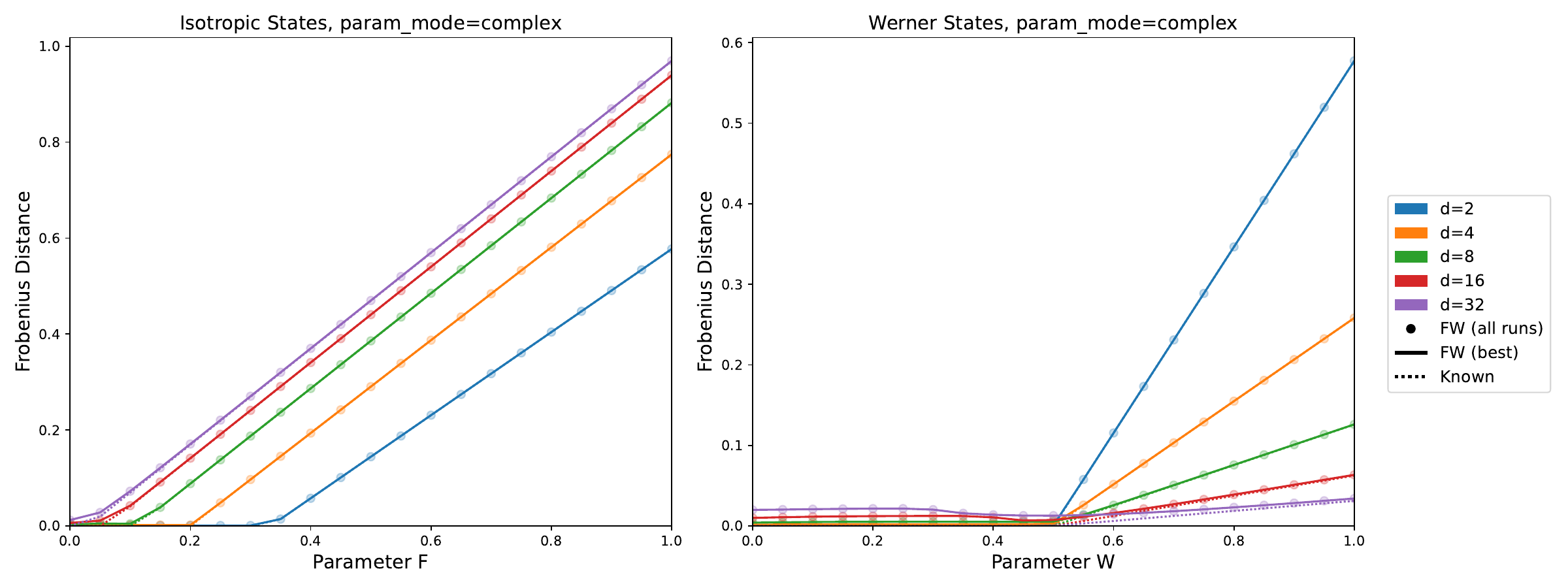}
        \caption{\footnotesize 
            \textbf{FW (complex):} $d=2,4,8,16,32$. FW shows near-exact threshold recovery under complex parameterization.
        }
        \label{fig:quantum_full_fw.complex}
    \end{subfigure}
    \hfill
    \begin{subfigure}[t]{0.49\textwidth}
        \centering
        \includegraphics[width=\linewidth]{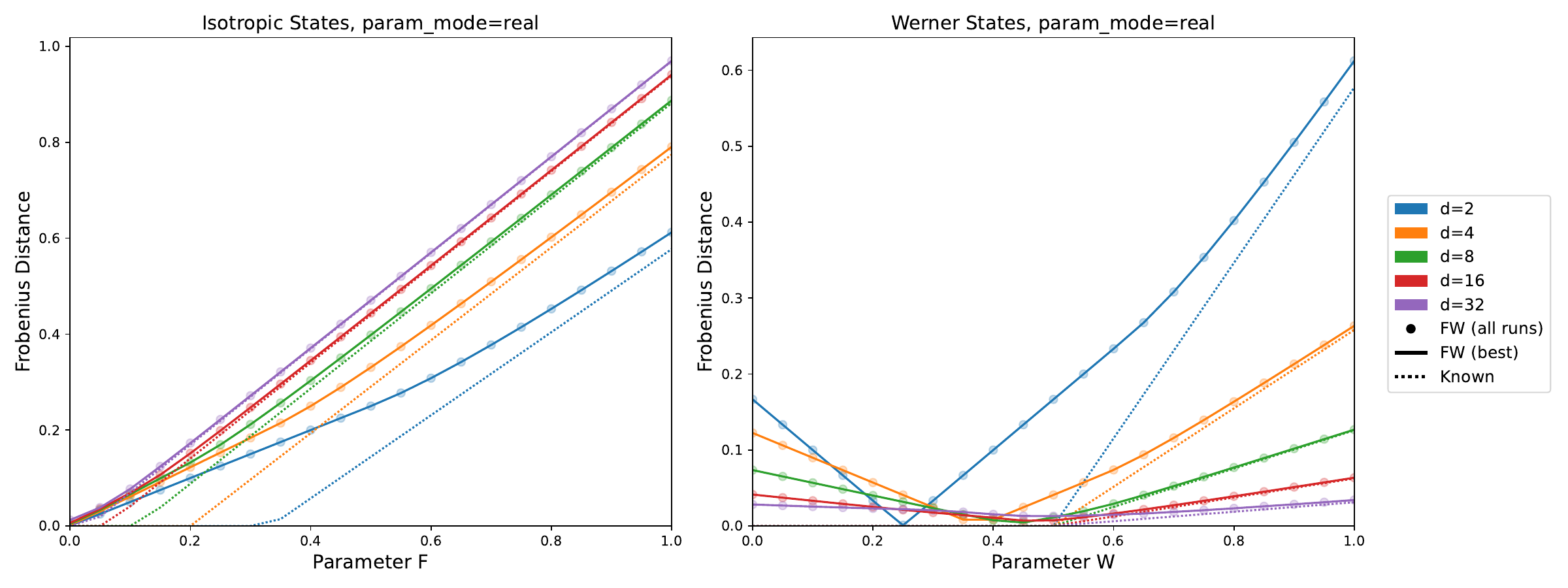}
        \caption{\footnotesize 
            \textbf{FW (real):} $d=2,4,8,16,32$. FW tends to converge to ``incorrect'' BSAs, deviating from complex parameterization.
        }
        \label{fig:quantum_full_fw.real}
    \end{subfigure}
    \caption{\small
         \textbf{Comparing real vs.\ complex parameterizations:}
         (a)–(b) Using complex vectors aligns with the known BSA in both isotropic and Werner families, while real vectors can systematically ``miss'' the correct solution.
         (c)–(d) Similar trends are observed at higher local dimensions ($d=8, 16, 32)$.  
        }
    \label{fig:quantum_complex_vs_real}
\end{figure}

\paragraph{Illustrations at $d=2$} 
To see this real vs complex discrepancy concretely, we examine $\bSigma^{\iso}_2(\alpha)$. 
Recall its known BSA expressions from Section \ref{sec:quantum_states}: for any $\alpha \in [0,1]$,
\begin{align*}
        \upisep\big( \bSigma^{\iso}_2(\alpha)  \big) = \bSigma^{\iso}_2(\alpha') 
        = \begin{bmatrix} 
            \frac{1+\alpha'}{4} & & & \frac{\alpha'}{2} \\
            & \frac{1-\alpha'}{4} & & \\
            & & \frac{1-\alpha'}{4} & \\
            \frac{\alpha'}{2} & & & \frac{1+\alpha'}{4} 
        \end{bmatrix},
        \quad\text{where}\quad \alpha' = \min\{ \alpha, 1/3 \}.
\end{align*}
While FW, GD, DP-GD converge to these BSAs under complex parameterization (Figure~\ref{fig:quantum_small.complex}), we observed that under real parameterization all three methods instead converge to ``incorrect'' matrices (Figure~\ref{fig:quantum_small.real}) of the following form:
\[
    \widetilde{\bSigma}^{\iso}_2(\alpha'') 
        = \begin{bmatrix} 
            \frac{1+\alpha''}{4} & & & \frac{\alpha''}{4} \\
            & \frac{1-\alpha''}{4} & \frac{\alpha''}{4} & \\
            & \frac{\alpha''}{4} & \frac{1-\alpha''}{4} & \\
            \frac{\alpha''}{4} & & & \frac{1+\alpha''}{4} 
        \end{bmatrix},
        \quad\text{where}\quad \alpha'' = \min\{ \alpha, 1/2 \}.
\]

A similar phenomenon appears for the Werner state $\bSigma^{\Werner}_d(\alpha)$. 
For any $\alpha \in [0,1]$, the known BSA expression for Werner state is
\begin{align*}
        \upisep\big( \bSigma^{\Werner}_2(\alpha)  \big) = \bSigma^{\Werner}_2(\alpha') 
        = \begin{bmatrix} 
            \frac{1-\alpha'}{3} & & & \\
            & \frac{1+2\alpha'}{6} & \frac{1-4\alpha'}{6} & \\
            & \frac{1-4\alpha'}{6} & \frac{1+2\alpha'}{6} & \\
            & & & \frac{1-\alpha'}{3} 
        \end{bmatrix},
        \quad\text{where}\quad \alpha' = \min\{\alpha, 1/2\}.
\end{align*}
Again, all three methods converge to these BSAs under complex parameterization, whereas they systematically produce a distinct matrix under real parameterization:
\[
    \widetilde{\bSigma}^{\Werner}_2(\alpha'') 
        = \begin{bmatrix} 
            \frac{1-\alpha''}{3} & & & \frac{1-4\alpha''}{12} \\
            & \frac{1+2\alpha''}{6} & \frac{1-4\alpha''}{12} & \\
            & \frac{1-4\alpha''}{12} & \frac{1+2\alpha''}{6} & \\
            \frac{1-4\alpha''}{12} & & & \frac{1-\alpha''}{3} 
        \end{bmatrix},
        \quad\text{where}\quad \alpha'' = \min\{ \alpha, 5/8 \}.
\]

\paragraph{Possible cause: Limited expressive power} 
Based on further (unpresented) experiments, we believe these persistent discrepancies primarily result from the \emph{lack of expressive power} in purely real expansions, rather than from local minima alone.
We conjecture that capturing the requisite phase relationships for certain entangled structures is essential in Kronecker-separable approximation---possibly the imaginary parts are canceled out, leaving richer expressive powers in approximation---which appears lacking in purely real expansions.
Thus, even though the exact BSA belongs to the real convex hull, constructing it with “simpler” real parameters (like small-rank sums of rank-1 factors) seems infeasible. 
This aligns with broader evidence that real-only frameworks can be incomplete for quantum structures \cite{renou2021quantum}, reinforcing the conclusion that complex degrees of freedom are crucial to represent relevant Kronecker-product structures accurately.